\documentclass[11pt,reqno]{amsart}
\usepackage{amscd,amsmath,amsopn,amssymb,amsthm,multicol}
\usepackage[backref=page, breaklinks=true,colorlinks=true,linkcolor=blue,citecolor=blue,urlcolor=blue]{hyperref}

\usepackage{enumerate,stmaryrd}

\DeclareMathOperator{\ad}{ad}
\DeclareMathOperator{\Id}{Id}

\DeclareMathOperator{\diag}{diag}

\DeclareMathOperator{\spann}{span}

\DeclareMathOperator{\Der}{Der}
\DeclareMathOperator{\Cl}{Cl}
\DeclareMathOperator{\rank}{rank}

\renewenvironment{proof}[1][Proof]{\textbf{#1.} }
{\ \rule{0.5em}{0.5em}}
\newtheorem{theorem}{Theorem}
\newtheorem{prop}{Proposition}
\newtheorem{lemma}{Lemma}
\newtheorem{corollary}{Corollary}

\theoremstyle{definition}
\newtheorem{definition}{Definition}
\newtheorem{remark}{Remark}

\begin{document}

\title
[Geodesic orbit pseudo-Riemannian nilmanifolds] {Geodesic orbit pseudo-Riemannian $H$-type nilmanifolds: case of minimal admissible Clifford modules}

\author[K.~Furutani, I.~Markina, Yu.G.~Nikonorov]{Kenro Furutani, Irina Markina, and Yurii Nikonorov}

\address{K.~Furutani. Osaka Central Advanced Mathematical Institute, Osaka Metropolitan University,
Sugimoto, Sumiyoshi-ku, Osaka 558-8585, Japan}
\email{kf46089@gmail.com}
\address{I.~Markina. Department of Mathematics, University of Bergen, P.O.~Box 7803,
Bergen N-5020, Norway}
\email{irina.markina@uib.no}
\address{Yu.G.~Nikonorov. Southern Mathematical Institute of VSC RAS,
53 Vatutina St., Vladikavkaz, 362025, Russia}
\email{nikonorov2006@mail.ru}

%\thanks{}



\begin{abstract}
    We investigate the geodesic orbit property of pseudo-Riemannian nilmanifolds, specifically those known in the literature as pseudo $H$-type
    Lie groups -- i.e., 2-step nilpotent Lie groups of Heisenberg type equipped with a left invariant pseudo-Riemannian metric. The study of homogeneous
    geodesics on Riemannian $H$-type Lie groups was completed by C.~Riehm in 1984. In this work, we extend these results to the pseudo-Riemannian $H$-type
    Lie groups and provide a complete characterization of the geodesic orbit property for the case where the underlying Lie algebras are constructed from
    the admissible Clifford modules of minimal dimension.
\end{abstract}

\maketitle

\tableofcontents

%\today
\section{Introduction and main results}

A Riemannian manifold $(M,g)$ is called {\it a  manifold with
homogeneous geodesics} or {\it a geodesic orbit manifold} (shortly,  {\it GO-manifold}) if any
geodesic $\gamma $ of $M$ is an orbit of a 1-parameter subgroup of
the full isometry group of $(M,g)$. A Riemannian manifold $(M=G/H,g)$, where $H$ is a compact subgroup
of a Lie group $G$ and $g$ is a $G$-invariant Riemannian metric,
is called  {\it a geodesic orbit space}
(shortly,  {\it GO-space}) if any geodesic $\gamma $ of $M$ is an orbit of a
1-parameter subgroup of the group $G$.
Hence, a Riemannian manifold $(M,g)$ is  a geodesic orbit Riemannian manifold,
if it is a geodesic orbit space with respect to its full connected isometry group. This terminology was introduced in
\cite{KV} by O.~Kowalski and L.~Vanhecke, who initiated a systematic study of such spaces.
In the same paper, O.~Kowalski and L.~Vanhecke classified all geodesic orbit Riemannian manifolds
of dimension $\leq 6$.

We refer to \cite{KV}, \cite{Arv17}, \cite{Du1}, \cite{Nik2017}, and \cite{BerNik20} for expositions on general properties of geodesic
orbit Riemannian manifolds and historical surveys.
Many interesting results on geodesic-orbit Riemannian spaces and their subclasses can be found in~\cite{Gor96, Zil96, Tam03, GorNik2018, CN2019, CNN2023, Nik2024, NikZilW25}, and in the references
therein. It should be noted that symmetric spaces, weakly symmetric spaces, naturally reductive homogeneous spaces, normal homogeneous spaces,
generalized normal homogeneous spaces (but not only) are subclasses of the class of geodesic orbit Riemannian spaces.
\smallskip

This paper is devoted to the study of a particularly important class of geodesic orbit pseudo-Riemannian spaces, namely,
pseudo $H$-type nilpotent Lie groups.
On the other hand, many of the results obtained below can also be used for more general classes of geodesic orbit pseudo-Riemannian nilmanifolds.

Some important results on geodesic orbit pseudo-Riemannian spaces were obtained in \cite{DK2007,DK2007n,Du2009,Bar,DBO,WC22,CWZ2022,NW23,CNWZ2024}.
Here we recall some important results related to homogeneous geodesics of pseudo-Riemannian manifolds.
We note that weakly symmetric pseudo-Riemannian spaces and naturally reductive homogeneous pseudo-Riemannian spaces
are important subclasses of geodesic orbit pseudo-Riemannian spaces~\cite{Ov2013, WC22}.

\begin{definition}\label{def:GO-pseudo}
A pseudo-Riemannian homogeneous reductive manifold $(G/H,g)$ is called  geodesic orbit (GO) if
a geodesic through the point $eH$ with any initial vector $\xi$ is an orbit of a 1-parameter isometry group of $(G/H,g)$.
\end{definition}

We recall the following useful criterion for the property to be a geodesic orbit pseudo-Riemannian manifold.

\begin{prop}[Geodesic Lemma \cite{DK2007}]\label{pr.geodlem1}
Let $(M= G/H, g)$ be a homogeneous reductive pseudo-Riemannian
manifold, with the corresponding reductive decomposition $\mathfrak{g} = \mathfrak{h}\oplus \mathfrak{m}$. Then $M$
is a $G$-geodesic orbit space if and only if, for any $T\in  \mathfrak{m}$, there exist $P = P(T) \in \mathfrak{h}$
and $k = k(T) \in \mathbb{R}$ such that
\begin{equation} \label{geodlem1}
\langle[T + P, Q]_{\mathfrak{m}}, T \rangle = k \,\langle T , Q \rangle\quad \text{for any}\quad Q \in \mathfrak{m},
\end{equation}
where  $\langle \cdot\,, \cdot \rangle$ denotes the inner product on  $\mathfrak{m}$ defined by $g$, and the subscript $\mathfrak{m}$ in \eqref{geodlem1}
means taking the $\mathfrak{m}$-component in $\mathfrak{g} = \mathfrak{h}\oplus \mathfrak{m}$.
Note that $k(T) = 0$ unless $T$ is a null vector {\rm(}it suffices to substitute $Q = T$ in \eqref{geodlem1}{\rm)}.
\end{prop}

We recall also the following definition.

\begin{definition}\label{def:naturaly-reductive}
Let $(M= G/H, g)$ be a homogeneous reductive pseudo-Riemannian
manifold.  Then $M$ is said to be a naturally reductive if there is a reductive decomposition $\mathfrak{g} = \mathfrak{h}\oplus \mathfrak{m}$ such that
\begin{equation} \label{eq:naturaly-reductive}
\langle[T, Q]_{\mathfrak{m}}, R \rangle =\langle Q,[T ,R]_m\rangle
\end{equation}
for the corresponding scalar product and any $T,Q,R\in\mathfrak m$.

\end{definition}
\smallskip
All naturally reductive $H$-type Lie groups, endowed with left invariant Riemannian metrics were classified  by A.~Kaplan in \cite{Kap83}.
In the same paper, the first examples of geodesic orbit but not naturally reductive Riemannian manifolds were obtained:
these are $H$-type groups with $2$-dimensional center (a minimal dimension of such groups is $6$).
The complete classification of geodesic orbit $H$-type groups with left invariant Riemannian metrics was obtained by C.~Riehm in \cite{Rie84}:

\begin{theorem}[\cite{Rie84}]\label{th.class.riemmannain}
Let $N$ be $H$-type Lie group {\rm(}equipped with a left invariant Riemannian metric{\rm)}
with the $H$-type Lie algebra $\mathfrak{n}=\mathfrak{z}\oplus \mathfrak{v}$, $m=\dim(\mathfrak{z})$, $n=\dim(\mathfrak{v})$, where $\mathfrak z$ is the centre.
Then $N$ is geodesic orbit if and only if one of the following three conditions holds:

{\rm 1)} $m = 1, 2, 3$ and $n$ is any possible;

{\rm 2)} $m = 5, 6$ and $n = 8$;

{\rm 3)} $m = 7$, $n =8, 16, 24$ and $\mathfrak{v}$ is an isotypic Clifford module {\rm(}in this case it is equivalent to the following property:
if $Z_1,Z_2,\dots,Z_7$ is an orthonormal basis for $\mathfrak{z}$, then
the linear transformation $X \mapsto J_{Z_1}J_{Z_2}\cdots J_{Z_7}(X)$ of $\mathfrak{v}$ is either $\Id$ or $-\Id${\rm)}.

Moreover, $N$ is naturally reductive if and only if $m=1$ or $m=3$.
\end{theorem}

The main result of the present work is as follows.

\begin{theorem}\label{th.class.pseudo}
Let $N_{r,s}$ be a pseudo $H$-type Lie group, where $(r,s)$, $s\geq 1$, is the signature of the left invariant pseudo-Riemannian metric restricted to the centre
of the group. Let $\mathfrak{n}_{r,s}=\mathfrak{z}\oplus \mathfrak{v}$ be the Lie algebra of $N_{r,s}$, where $\mathfrak z$ is the centre and $\mathfrak v$ is
a minimal admissible module for the Clifford algebra $\Cl(\mathbb R^{r,s})$.
Then the following three assertions hold:

{\rm 1)} $N_{r,s}$ is naturally reductive {\rm(}hence, geodesic orbit{\rm)} if and only if $(r,s)\in\{(0,1), (1,2)\}$;
% and~$\mathfrak{v}$ is any admissible module.

{\rm 2)} $N_{r,s}$ is geodesic orbit but not naturally reductive if and only if $(r,s)=(3,4)$;

{\rm 3)} $N_{r,s}$ with $(r,s)\notin \{(0,1), (1,2), (3,4)\}$ is not a geodesic orbit pseudo-Riemannian manifold.
\end{theorem}

The third assertion of Theorem~\ref{th.class.pseudo} follows immediately from the first two.

The paper is organized as follows. In Section \ref{se.auiliary}, we recall some important results on
pseudo-Riemannian geodesic orbit metrics on nilpotent Lie groups. The main role here is played by the notion of the transitive normalizer condition, which
Riemannian version was used by C.~Gordon in order to study geodesic orbit Riemannian metric on nilpotent Lie groups.
In Section~\ref{se.rseudo.htnlg},
we recall some important properties of pseudo $H$-type Lie groups, their isometry and automorphism groups.
In Section \ref{sec.example}, we check the geodesic orbit property for pseudo-Riemannian H-type groups with small dimension of the centre.
More exactly, the groups ${N}_{0,1}$, ${N}_{1,1}$,  ${N}_{0,2}$,  ${N}_{1,2}$,  ${N}_{2,1}$, and  ${N}_{0,3}$
are studied, after which the classification of naturally reductive pseudo-Riemannian $H$-type Lie groups is obtained, see Proposition~\ref{pr.natred.1}.
Section~\ref{se.tot.geod} is devoted to the description of some important properties of geodesic orbit pseudo-Riemannian manifolds.
As in the case of Riemannian manifolds, it is proved that any geodesically complete totally geodesic submanifold of a given
geodesic orbit pseudo-Riemannian manifold is geodesic orbit itself, see Theorem~\ref{th_ps_totgeod}. This result is further refined for the case of two-step nilpotent
pseudo-Riemannian groups, see Theorem~\ref{pr_totgeod_h1} and Corollary~\ref{co.totgeod.1}.
In Section~\ref{sec:mod4}, we obtain some auxiliary results on the geodesic orbit property for  pseudo-Riemannian $H$-type manifolds,
that enable us to systematically analyze all pseudo-Riemannian $H$-type manifolds,
except of ${N}_{3,4}$, for the property to be geodesic orbit.
Finally, in Section~\ref{sec.N3.4}, we prove that the $15$-dimensional pseudo $H$-type nilmanifold ${N}_{3,4}$ is geodesic orbit (Theorem~\ref{the.0case}).

\section{Auxiliary results}\label{se.auiliary}

Here we recall some important results related to
$2$-step nilpotent Lie groups supplied by  left-invariant pseudo-Riemannian metrics. Such groups will be referred to as 2-step pseudo-Riemannian nilmanifolds.
Finally, we formulate some useful statements on GO properties of $2$-step  pseudo-Riemannian nilmanifolds.

Let $(N,g)$ be a 2-step pseudo-Riemannian nilmanifold with the Lie algebra $\mathfrak{n}$ and  the scalar product  $\langle \cdot\,,\cdot \rangle$ (a symmetric non-degenerate bilinear form) generating the pseudo-Riemannian left invariant metric $g$.
If the centre $\mathfrak{z}$ of $\mathfrak{n}$  is non-degenerate with respect to $\langle \cdot\,,\cdot \rangle$, then we write $\mathfrak{n}=\mathfrak{z}\oplus \mathfrak{v}$,
where
$\mathfrak{v}=\mathfrak{z}^\perp$ relative to $\langle \cdot\,,\cdot \rangle$.
In this case, $\mathfrak{v}$ is also non-degenerate, see~\cite[Lemma 2.60]{Lee}.

Whenever $(N,g)$ is connected simply connected, we do not distinguish between the
group of automorphisms of the nilmanifold $(N,g)$ and of its Lie algebra $\mathfrak{n}=\mathfrak{z} \oplus \mathfrak{v}$.
Note that each skew-symmetric derivation of $\mathfrak{n}$ leaves each of $\mathfrak{v}$ and $\mathfrak{z}$ invariant.
For any $Z\in \mathfrak{z}$, we consider the operator
\begin{equation}\label{eq_j_z_1}
J_Z:\mathfrak{v} \rightarrow \mathfrak{v}, \mbox{\,\,\,\, such that\,\,\,\,}
\langle J_Z(X),Y \rangle=\langle [X,Y],Z\rangle, \quad X,Y\in \mathfrak{v}.
\end{equation}
The map $J_Z$ is skew-symmetric and $J_Z(Y)=(\ad Y)'(Z)$, where $(\ad Y)'$ is adjoint to $\ad Y$ with respect to $\langle \cdot\,, \cdot \rangle$. The map $J\colon Z \rightarrow \mathfrak{so}(\mathfrak v, \langle \cdot\,, \cdot \rangle_{\mathfrak v})$, sending $Z\mapsto J_Z$ is linear.
We denote $\mathbf{V}=J(\mathfrak{z})$ the linear subspace of $\mathfrak{so}(\mathfrak{v},\langle \cdot\,, \cdot \rangle _{\mathfrak{v}})$.
The group of isometries of the nilmanifold $(N,g)$ is given by

\begin{equation}\label{isotgr1}
H = \{(\varphi,\psi)\in  O(\mathfrak{z},\langle \cdot\,, \cdot \rangle _{\mathfrak{z}})\times O(\mathfrak{v},  \langle \cdot\,, \cdot \rangle _{\mathfrak{v}})
\,|\, \psi J_Z \psi^{-1} = J_{\varphi(Z)},\  Z \in  \mathfrak{z}\},
\end{equation}
while its Lie algebra is
\begin{equation}\label{isotalg1}
\mathfrak{h}= \Der(\mathfrak{n}) \cap \mathfrak{so}(\mathfrak{n},\langle \cdot\,, \cdot \rangle)=
\{D=(C,A)\in  \mathfrak{so}(\mathfrak{z}, \langle \cdot\,, \cdot \rangle _{\mathfrak{z}})\times\mathfrak{so}(\mathfrak{v}, \langle \cdot\,, \cdot \rangle _{\mathfrak{v}})
\,|\, [A,J_Z]= J_{C(Z)},\  Z \in  \mathfrak{z}\}.
\end{equation}

The next result is well known, see e.g. Corollary 3.5 in \cite{CP2009} or Proposition 2.3 in \cite{Ov2013}.

\begin{prop}\label{pr.isgr.1}
Let $(N,g)$ be a pseudo-Riemannian nilmanifold with
non-degenerate center. Then the connected isometry group of $(N,g)$ is
$N\rtimes H$,
where $N$ is the set of left translations by elements of $N$ and the isotropy
subgroup $H$ is given by the isometric automorphisms \eqref{isotgr1} with Lie algebra $\mathfrak{h}$ as in \eqref{isotalg1}.
\end{prop}

In this case,
the isotropy group of $(N,g)$ of the identity element $e$ is exactly $H$ with the embedding $a \in H \mapsto (e,a)\in N\rtimes H$.

\bigskip

Applying Proposition~\ref{pr.geodlem1} we note the following. If a nilmanifold $(N,g)$ is geodesic orbit, then for any $X\in\mathfrak{v}$ and
any ${Z}\in \mathfrak{z}$
there exist $D\in\mathfrak{h}$ and $k \in \mathbb{R}$ such that
$$
\langle [X+Z+D,\widetilde{X}+\widetilde{Z}], X+Z\rangle=k\langle X+Z,\widetilde{X}+\widetilde{Z}\rangle=
k\left(\langle X,\widetilde{X}\rangle+\langle Z,\widetilde{Z}\rangle\right)
$$
for all $\widetilde{X}\in \mathfrak{v}$ and all $\widetilde{Z}\in \mathfrak{z}$.
It is easy to see that
$$
[X+Z+D,\widetilde{X}+\widetilde{Z}]=[X,\widetilde{X}]+[D,\widetilde{X}]+[D,\widetilde{Z}],
$$
where $[D,\widetilde{X}]\in \mathfrak{v}$ and $[X,\widetilde{X}]+[D,\widetilde{Z}]\in \mathfrak{z}$.
Hence,
$$
k\left(\langle X,\widetilde{X}\rangle+\langle Z,\widetilde{Z}\rangle\right)=
\langle [D,\widetilde{X}], X \rangle+\langle [X,\widetilde{X}]+[D,\widetilde{Z}], Z\rangle=-\langle \widetilde{X}, [D,X] \rangle+\langle J_Z X, \widetilde{X}\rangle
-\langle \widetilde{Z}, [D,Z] \rangle.
$$
Since $\widetilde{X}\in \mathfrak{v}$ and $\widetilde{Z}\in \mathfrak{z}$ are arbitrary, then
\begin{equation}\label{eq.geodl.2}
[D,X]+k\, X=J_Z X, \qquad [D,Z]+k\,Z=0.
\end{equation}
Note, that $[D,Z]=0$ implies $[D,J_Z]=0$ according to~\eqref{isotalg1}.
If, in addition, $X+Z$ is not a null vector, then $k=0$.

If we write $D=(C,A)\in\mathfrak h$, then equations~\eqref{eq.geodl.2} can be written in the form
$$
(C+k\Id)\,Z=0,\quad (A+k\Id)\,X=J_ZX.
$$

On the other hand, relations~\eqref{eq.geodl.2} imply Proposition ~\ref{pr.geodlem1} for all $X\in\mathfrak{v}$ and ${Z}\in \mathfrak{z}$.
These observations allow us to rewrite Proposition~\ref{pr.geodlem1} for $2$-step nilpotent pseudo-Riemannian groups in the spirit of work~\cite{Gor96},
where this idea was used for Riemannian metrics on nilpotent Lie groups.

\begin{prop}[\cite{Bar}, Theorem 3.3]\label{gonil1n.2}
In the above notations, $(N,g)$ is a geodesic orbit 2-step pseudo-Riemannian nilmanifold if and only if for any $Z\in \mathfrak{z}$ and
$X\in \mathfrak{v}$ there is $D=(C,A)\in\mathfrak h$ such that
$(C+k\Id)\,Z=0$, $(A+k\Id)\, X=J_Z(X)$.
\end{prop}

Consider an action $\rho$ of the isotropy algebra $\mathfrak h$ in~\eqref{isotalg1} on $\mathfrak{so}(\mathfrak{v},\langle \cdot\,, \cdot \rangle _{\mathfrak{v}})$; that is $\rho\colon\mathfrak{h} \to \mathfrak{so}(\mathfrak{v},\langle \cdot\,, \cdot \rangle _{\mathfrak{v}})$. We may reformulate the condition
of Proposition \ref{gonil1n.2} as follows.
We know that $\mathbf{V}=J(\mathfrak{z})$ is a linear subspace in $\mathfrak{so}(\mathfrak{v},\langle \cdot\,, \cdot \rangle _{\mathfrak{v}})$.
Further, for every $D\in \mathfrak{h}$ and $Z\in \mathfrak{z}$ we get
$J_{[D,Z]}=[\rho(D),J_Z]$ (it easy follows from the condition on $D$ to be a skew-symmetric derivation), hence,
the subspace $\mathbf{V}=J(\mathfrak{z})$ is normalized by the subalgebra $\mathbf{N}:=\rho(\mathfrak{h})$ in
$\mathfrak{so}(\mathfrak{v},\langle \cdot\,, \cdot \rangle _{\mathfrak{v}})$ by the fact that
$$
[\rho(D),J_Z]=J_{[D,Z]}\in \mathbf{V}.
$$
The equality $J_{[D,Z]}=[\rho(D),J_Z]$ implies that
the representation $\rho\colon\mathfrak{h} \to \mathfrak{so}(\mathfrak{v},\langle \cdot\,, \cdot \rangle _{\mathfrak{v}})$ is faithful
(otherwise, some non-trivial $D\in \mathfrak{h}$ acts trivially both on $\mathfrak{v}$
and on $\mathfrak{z}$, hence, on $\mathfrak{n}$). Moreover, any element of the normalizer $\mathbf{N}$ of $\mathbf{V}=J(\mathfrak{z})$ in $\mathfrak{so}(\mathfrak{v},\langle \cdot\,, \cdot \rangle _{\mathfrak{v}})$
can be considered as an image of some element $D\in \mathfrak{h}$ under the map  $\rho\colon\mathfrak{h} \to \mathfrak{so}(\mathfrak{v},\langle \cdot\,, \cdot \rangle _{\mathfrak{v}})$,
which follows from~\eqref{isotalg1}.
Therefore, we have the following statement.

\begin{corollary}\label{cor.isom.1}
If $(N,g)$ is a geodesic orbit pseudo-Riemannian nilmanifold, then the Lie algebra~$\mathfrak{h}$ is isomorphic to the normalizer $\mathbf{N}$ of $\mathbf{V}=J(\mathfrak z)$
in $\mathfrak{so}(\mathfrak{v},\langle \cdot\,, \cdot \rangle _{\mathfrak{v}})$ under the representation $\rho\colon\mathfrak h\to \mathfrak{so}(\mathfrak{v},\langle \cdot\,, \cdot \rangle _{\mathfrak{v}})$.
\end{corollary}

Therefore, if $(N,g)$ is a geodesic orbit pseudo-Riemannian nilmanifold, then we have
a Lie subalgebra $\mathbf{N}\subset \mathfrak{so}(\mathfrak{v},\langle \cdot\,, \cdot \rangle _{\mathfrak{v}})$ and
an $\ad(\mathbf{N})$-invariant module $\mathbf{V}$ in $\mathfrak{so}(\mathfrak{v},\langle \cdot\,, \cdot \rangle _{\mathfrak{v}}))$,
such that
for every $X\in \mathfrak{v}$ and $W\in \mathbf{V}$  there is $B \in \mathbf{N}$ with the following properties:
$[B,W]=0$ and $B(X)=W(X)$.

\begin{remark}\label{re.rec.oper.1}
Let $\mathfrak h$ be as in~\eqref{isotalg1} and $\mathbf{V}=J(\mathfrak{z})$. Then every element of $\mathfrak{h}$
is determined by a skew-symmetric operator $A \in \mathfrak{so}(\mathfrak{v}, \langle \cdot\,, \cdot \rangle _{\mathfrak{v}})$ with the property
$[A,\mathbf{V}]\subset \mathbf{V}$ (i.e., $A$ is from the normalizer $\mathbf{N}$ of $\mathbf{V}$ in $\mathfrak{so}(\mathfrak{v}, \langle \cdot\,, \cdot \rangle _{\mathfrak{v}})$).
In this case the operator $C\in \mathfrak{so}(\mathfrak{z}, \langle \cdot\,, \cdot \rangle _{\mathfrak{z}})$ can be recovered from the equality
$J_{C(Z)}=[A,J_Z]$ for all $Z\in  \mathfrak{z}$.
\end{remark}

\smallskip

A special case of geodesic orbit pseudo-Riemannian spaces are naturally reductive homogeneous pseudo-Riemannian spaces (Definition~\ref{def:naturaly-reductive}),
see also~\cite{Ov2013} and references therein.
The naturally reductive Riemannian homogeneous spaces are generalizations of normal homogeneous Riemannian spaces
and symmetric spaces, see, e.g., \cite{DZ}, \cite{KV85}, \cite{Gor85}, \cite{AFF}, \cite{Storm19}, \cite{Storm20}.
It should be noted that a complement $\mathfrak{m}$ in the definition of naturally reductive (pseudo) Riemannian manifold is not unique in general.
For instance, any invariant Riemannian metric on the Ledger--Obata space $(F\times F\times F)/\diag(F)$,
where $F$ is any simple compact Lie group, is naturally reductive with respect to a suitable reductive complement~\cite{NN2019}.

The following result gives us a useful criterion of a 2-step Lie group with pseudo-Riemannian left-invariant metric to be naturally reductive.

\begin{theorem}[\cite{Ov2013}, Theorem 3.2]\label{th:Ovanto}
Let $(N,g)$ denote a 2-step pseudo-Riemannian nilmanifold with a non-degenerate center. Assume that the map $J\colon \mathfrak{z}\to \mathfrak{so}(\mathfrak{v})$ is injective, see~\eqref{eq_j_z_1}.
Then the metric is naturally reductive with respect to
$G= N\rtimes H$ (see Proposition~\ref{pr.isgr.1}), if and only if
\begin{itemize}
\item[(i)] $\mathbf {V}=J(\mathfrak{z})$ is a Lie subalgebra of $\mathfrak{so}(\mathfrak{v})$ and
\item[(ii)] $[J(Z_1), J(Z_2)] = J\big(\tau_{Z_1}(Z_2)\big)$ where $\tau_{Z_1}\in\mathfrak{so}(\mathfrak{z})$ for any $Z_1\in \mathfrak{z}$.
\end{itemize}
\end{theorem}
Since the map $J$ is supposed to be injective, the map $\tau$ can be easily recovered from ${\rm(ii)}$.
%%%%%%%%%%%%%%%%%%%%%%%%%%%%%%%%%%%%%%%%%%%%%

\section{Pseudo-Riemannian $H$-type nilmanifolds}\label{se.rseudo.htnlg}

Let $(N,g)$ be a 2-step pseudo-Riemannian nilmanifold and $\mathfrak n=\mathfrak z\oplus\mathfrak v$ be its Lie algebra endowed with a scalar product $\langle.\,,.\rangle$
making the center non-degenerate. We identify $(\mathfrak z,\langle.\,,.\rangle_{\mathfrak z})$ with the pseudo Euclidean vector space
$\mathbb{R}^{r,s}=(\mathbb{R}^{r+s},\langle.\,,.\rangle_{r,s})$, where
$$
\langle Z, W \rangle_{r,s} =
\sum_{i=1}^r z_i\, w_i-\sum_{j=1}^s z_{r+j}\, w_{r+j}, \quad Z=(z_1,\ldots,z_{r+s}),\ W=(w_1,\ldots,w_{r+s}).
$$
If the linear operator $J\colon\mathfrak{z} \rightarrow \mathfrak{so}(\mathfrak{v},\langle\cdot\,,\cdot\rangle_{\mathfrak v})$ defined by
\begin{equation}\label{eq:JJ}
\langle J_Z(X),Y \rangle_{\mathfrak v}=\langle [X,Y],Z\rangle_{r,s}, \quad X,Y\in \mathfrak{v},\ Z\in \mathfrak{z}=\mathbb{R}^{r,s},
\end{equation}
satisfies $J_Z^2 (X)=-\langle Z,Z \rangle_{r,s} X$ for any $Z \in \mathbb{R}^{r,s}$ and all $X \in \mathfrak{v}$,
then $\mathfrak{n}=\mathfrak n_{r,s}$ is called the pseudo $H$(eisenberg)-type Lie algebra. It is easy to check that
this definition %the definition of {\color{blue}pseudo} $H$-type Lie algebra
implies
\begin{equation}\label{eq:JJJ}
\langle J_Z(X),J_W(X) \rangle_{\mathfrak v}=\langle Z,W \rangle_{r,s}\langle X,X \rangle_{\mathfrak v}.
\end{equation}
We denote by ${N}_{r,s}$ the connected simply connected Lie group, whose Lie algebra is the pseudo $H$-type Lie algebra $\mathfrak n_{r,s}$.
The $H$-type Lie algebras $N_{r,0}$ with a positive definite scalar product $\langle\cdot\,,\cdot\rangle_{r,0}$ on the centre were introduced in~\cite{Kap80}
and with an arbitrary indefinite scalar product $\langle\cdot\,,\cdot\rangle_{r,s}$ on $\mathfrak z$ in~\cite{Cia}, see also~\cite{GKM13}.

The
Lie algebras $\mathfrak n_{r,s}$ are related to the representations of the
Clifford algebras in the following way. Let $J\colon\operatorname{Cl}(\mathbb{R}^{r,s}) \rightarrow \operatorname{End} (\mathfrak{v})$ be a representation of
the Clifford algebra $\operatorname{Cl}(\mathbb{R}^{r,s})$ generated by the pseudo Euclidean vector space $\mathbb{R}^{r,s}$.
If there is a scalar product $\langle \cdot\,,\cdot \rangle_{\mathfrak{v}}$ on the representation space $\mathfrak{v}$ (Clifford module)
such that the linear map $J_Z$ is skew-symmetric for any $Z\in \mathbb{R}^{r,s}$; that is
$$
\langle J_Z(X),Y \rangle_{\mathfrak v}=-\langle X,J_Z(Y) \rangle_{\mathfrak v},\quad Z\in \mathbb{R}^{r,s},\quad X,Y\in\mathfrak v,
$$
then we get a pseudo $H$-type Lie algebra with the commutators defined in~\eqref{eq:JJ},
see details in~\cite{Kap81, Cia, KapTir, GMKMV, FuMa16, FuMa17, FuMa19, FuMa21}. The scalar product $\langle \cdot\,,\cdot \rangle_{\mathfrak{v}}$ in this case
is called admissible and $\mathfrak v$ is called an admissible (Clifford) module.

Note that, see e.g.~\cite{Cia}, or~\cite[Propososition 2.2.2]{FuMa21}
the signature of the scalar product space $(\mathfrak{v},\langle \cdot\,,\cdot \rangle_{\mathfrak{v}})$
is neutral and $(\mathfrak{v},\langle \cdot\,,\cdot \rangle_{\mathfrak{v}})$ is isometric to $\mathbb{R}^{l,l}$ for some $l \in \mathbb{N}$ if $s>0$,
whereas the corresponding signature is either $(l,0)$ or $(0,l)$ for some $l \in {\color{blue}2}\mathbb{N}$ if $s=0$.
We use the symbol $^t$ for the transposition according to the standard Euclidean product on $\mathfrak{v}=\mathbb{R}^{2l}$.
On the other hand, we use the symbol $ ^{\tau}$ for the transposition according to scalar product $\langle \cdot\,,\cdot \rangle_{l,l}$
on $\mathfrak{v}=\mathbb{R}^{l,l}$ or $\langle \cdot\,,\cdot \rangle_{r,s}$
on $\mathfrak{z}=\mathbb{R}^{r,s}$.
In particular, for any operator $A$ on $\mathfrak{v}$, we get $\langle A(X),Y \rangle=\langle X, A^{\tau}(Y) \rangle$, $X,Y \in \mathfrak{v}$, and
$J_Z^{\tau}=-J_Z$ for any $Z\in \mathfrak{z}$.
If $\eta=\diag(\Id_l,-\Id_l)$, then $J_Z^{\tau}=\eta J_Z^t \eta$ and
$A^{\tau}=\eta A^{t} \eta$.

Thus we use the identification $\mathbf{V}=J(\mathfrak{z}) \subset \mathfrak{so}(\mathfrak{v},,\langle\cdot\,,\cdot\rangle_{\mathfrak v})\cong\mathfrak{so}(l,l)$.
Recall that $\dim \mathbf{V}=\dim \mathfrak{z}=r+s$.

\begin{prop}\label{pr.triple.1}
Let $\mathfrak{n}_{r,s}$ be a pseudo $H$-type Lie algebra, $\mathbf{N}$ and $\mathbf{Z}$ the normalizer and the centralizer of $\mathbf{V}=J(\mathfrak z)$ in $\mathfrak{so}(l,l)$. Then we have the following properties:

1. $[\mathbf{V},\mathbf{V}]$ and $\mathbf{L}:=[\mathbf{V},\mathbf{V}]+\mathbf{V}$ are Lie subalgebras in $\mathfrak{so}(l,l)$;
\smallskip

2. the Lie algebra $[\mathbf{V},\mathbf{V}]$ is isomorphic to $\mathfrak{so}(r,s)$;
\smallskip

3. $[\mathbf{V},\mathbf{V}]\subset \mathbf{N}$ and $\dim \mathbf{N} \geq (r+s)(r+s-1)/2$;
\smallskip

4. $\Bigl(\mathbf{L},[\mathbf{V},\mathbf{V}]\Bigr)$ is a symmetric pair,
i.e., $\mathbf{V}$ is a Lie triple system;
\smallskip

5. the Lie algebra $\mathbf{L}~\text{is commutative if}~
(r, s)\in\left\{(1,0), (0,1)\right\}$,
it is simple if\\ %\linebreak
\hspace{0.8cm} $(r, s)\not\in \left\{(1,0), (0,1),(3,0), (1,2)\right\}$,
and it is semisimple if $(r, s)\in \left\{(3,0), (1,2)\right\}$;
\smallskip

6. $\mathbf{N}=[\mathbf{V},\mathbf{V}] \oplus \mathbf{Z}$ {\rm(}a direct sum of Lie algebras{\rm)}.
\end{prop}

\begin{proof} We give an outline of the proof.
For every pair of orthogonal vectors $Z', Z'' \in \mathfrak{z}$, the map $\Phi_{Z',Z''}$, defined by
\begin{equation}\label{eq:Phi}
\Phi_{Z',Z''}(X + Z) = J_{Z'}J_{Z''}(X)+2 \langle Z',Z\rangle Z''-2 \langle Z'',Z\rangle Z',\quad Z\in \mathfrak{z},\quad X\in \mathfrak{v},
\end{equation}
is a skew-symmetric derivation of $(\mathfrak{n}_{r,s}, \langle \cdot\,, \cdot \rangle)$,
see e.g. Lemma 2.2 in~\cite{CiaCo} or \cite{Rie82}.
Since $[J_{Z'},J_{Z''}]=2J_{Z'}J_{Z''}$, then $[\mathbf{V},\mathbf{V}]\subset \mathbf{N}$, due to the fact that $\Phi_{Z',Z''}$ and, hence, $[J_{Z'},J_{Z''}]$
is in the isotropy Lie algebra $\mathfrak{h}$, see \eqref{isotalg1} and Proposition \ref{pr.isgr.1}. Moreover,
$[\mathbf{V},\mathbf{V}]$ is a Lie subalgebra in $\mathbf{N}$, that is isomorphic to  $\mathfrak{so}(r,s)$.
Indeed, by Lemma 5.1 in~\cite{AFMV}, we can choose an orthonormal basis $Z_i$, $i=1,\dots, r+s$, for  $\mathfrak{z}=\mathbb{R}^{r,s}$, such that
$\mathbf{V} =J(\mathfrak{z})=J(\mathbb{R}^{r,s}) \subset \mathfrak{so}(l,l)$ has a basis of the following type:
$\{J_{Z_i}\}$, $i=1,\dots, r+s$, while
$\{J_{Z_j}J_{Z_k}\}$, $j,k = 1,\dots, r+s$, $j<k$, constitute a basis in
$[\mathbf{V},\mathbf{V}]\subset \mathbf{N}$. Hence,
$[\mathbf{V},\mathbf{V}]$ is isomorphic to  $\mathfrak{so}(r,s)$, then $\dim  \mathbf{N}\geq \dim [\mathbf{V},\mathbf{V}]=\dim \mathfrak{so}(r,s)=(r+s)(r+s-1)/2$.
This result also follows from \cite[Proposition 3.2.4]{FuMa21}.

It is straightforward to verify that
$[[\mathbf{V},\mathbf{V}],\mathbf{V}]\subset \mathbf{V}$, which means that $[\mathbf{V},\mathbf{V}]\subset \mathbf{N}$.
Moreover, $\mathbf{L}=[\mathbf{V},\mathbf{V}]+\mathbf{V}$ is a Lie algebra
and $\mathbf{V}$ is a Lie triple system, see Proposition 5.2 in~\cite{AFMV}.
\smallskip

Finally, the Lie algebra
$\mathbf{L}$ is commutative and $1$-dimensional if $(r, s)\in \left\{(1,0), (0,1)\right\}$, it is simple if $(r, s)\not\in \left\{(1,0), (0,1), (3,0), (1,2)\right\}$,
and it is semisimple if $(r, s)\in \left\{(3,0), (1,2)\right\}$
by Theorem 5.1 in \cite{AFMV}.

It is known that  each $B\in \mathbf{N}$ (i.e., each skew-symmetric derivation of $\mathfrak{z}$)
decomposes as a sum $B_0 + B_1$, where $B_0\in \mathbf{Z}$ and $B_1 \in [\mathbf{V},\mathbf{V}]$,
see details in~\cite[Corollary 2.6]{CiaCo}. Therefore, $\mathbf{N}=[\mathbf{V},\mathbf{V}] \oplus \mathbf{Z}$.
On the level of automorphism groups, similar results were obtained in~\cite[Subsection 3.2]{FuMa21}.
\end{proof}

It should be noted that the intersection of $\mathbf{V}$ and $[\mathbf{V},\mathbf{V}]$ can be non-trivial.
We have the following result.

\begin{lemma}[see e.g. Lemma 6 in \cite{BerNik09}]\label{le.ideal_h1}
Suppose that $K:=\mathbf{V} \cap [\mathbf{V},\mathbf{V}]$ is non-trivial. Then $K$ is an ideal in $\mathbf{L}= \mathbf{V}+[\mathbf{V},\mathbf{V}]$.
If, in addition, $\mathbf{L}$ is simple, then $K=\mathbf{V}=[\mathbf{V},\mathbf{V}]=\mathbf{L}$.
\end{lemma}

\begin{proof}
We see that
$$
[K,\mathbf{V}]\subset [\mathbf{V},\mathbf{V}], \quad [K,\mathbf{V}]\subset [[\mathbf{V},\mathbf{V}],\mathbf{V}] \subset \mathbf{V},
$$
$$
[K,[\mathbf{V},\mathbf{V}]]\subset[\mathbf{V},[\mathbf{V},\mathbf{V}]] \subset \mathbf{V}, \quad
[K,[\mathbf{V},\mathbf{V}]]\subset [[\mathbf{V},\mathbf{V}],[\mathbf{V},\mathbf{V}]]\subset [\mathbf{V},\mathbf{V}].
$$
Therefore,
$$
[K,\mathbf{V}]\subset\mathbf{V} \cap [\mathbf{V},\mathbf{V}]= K, \quad [K,[\mathbf{V},\mathbf{V}]]\subset\mathbf{V} \cap [\mathbf{V},\mathbf{V}]= K,
\quad [K,\mathbf{L}]\subset K.
$$
Hence, $K$ is an ideal in $\mathbf{L}$.
\end{proof}

Note that the constant $k=k(T)$ in  Proposition~\ref{pr.geodlem1} can be different from zero for pseudo-Riemannian manifolds, see~\cite{Bar},
where it is shown that $k$ depends on the reparametrization of a homogeneous geodesic and it could be $k(T)\neq 0$ for a null initial velocity $T$ of such a geodesic.
\smallskip

Recall that the isometry group of the pseudo $H$-type Lie group $N_{r,s}$ with a given left invariant pseudo Riemannian metric is identified with the group
$G=N_{r,s}\rtimes H$, where $H$ is defined in~\eqref{isotgr1} and its Lie algebra $\mathfrak h$ is given by
~\eqref{isotalg1}.

\subsection{Integral basis, periodicity, and the automorphism groups of pseudo $H$-type Lie algebras}\label{sec:31}
The pseudo $H$-type Lie algebras are closely related to Clifford algebras $\operatorname{Cl}(\mathbb{R}^{r,s})$ and their representation spaces $\mathfrak v$. Now we describe a convenient basis for pseudo $H$-type Lie algebras. We fix an orthonormal basis $B_{r,s}=\{Z_1,\ldots, Z_r,Z_{r+1},\ldots,Z_{r+s}\}$ for $\mathbb R^{r,s}$, where
\begin{equation}\label{eq:brs}
    \begin{cases}
        &Z_1,\ldots,Z_r\quad \text{are positive, i.e.,}~\langle Z_{i},\,Z_{i}\rangle_{r,s}=1,\ \ i=1,\ldots,r,
        \\
        &Z_{r+1},\ldots,Z_{r+s}~\text{are negative, i.e.,}~\langle Z_{i},\,Z_{i}\rangle_{r,s}=-1,\ \ j=r+1,\ldots,r+s.
    \end{cases}
\end{equation}
Consider a finite subgroup $G(B_{r,s})$ of the Pin group in $\operatorname{Cl}(\mathbb{R}^{r,s})$ generated by the basis $B_{r,s}$:
\begin{equation*}\label{group generated by basis}
        G(B_{r,s})=\big\{\pm {\mathbf 1},\ \pm Z_{i_1}\cdots Z_{i_k}\mid\
      1\leq i_1<\cdots<i_k\leq r+s, \quad k=1,\ldots, r+s\big\}.
\end{equation*}
In the present work we will consider only {\it \bf minimal admissible modules}, which are pairs $(\mathfrak v,\langle \cdot\,,\cdot\rangle_{\mathfrak v})$ of the representation space $\mathfrak v$ of minimal dimension, where an (non-degenerate) admissible scalar product can be constructed. The construction of admissible scalar products and their description can be found in~\cite{Cia} and~\cite[Section 2]{FuMa19}.
\begin{definition}\label{invariant basis}
Fix an orthonormal basis $B_{r,s}$ of $\mathbb R^{r,s}$. An orthonormal basis $\mathfrak{B}(\mathfrak v)$ of a minimal admissible module $\mathfrak v$ is called invariant basis if it is invariant under the action of
$G(B_{r,s})$; that is for any $X_i\in\mathfrak{B}(\mathfrak v)$ and $Z_j\in B_{r,s}$, there exists $X_k\in \mathfrak{B}(\mathfrak v)$ such that $J_{Z_j}(X_i)=\pm X_{k}$.
\end{definition}
According to Definition~\ref{invariant basis} the maps $J_{Z_j}$, $Z_j\in B_{r,s}$ act on an invariant basis $\mathfrak{B}(\mathfrak v)$ by permutations up to the sign $\pm$. To construct an invariant basis for $\mathfrak v$ we consider a maximal subgroup $\mathcal S$ of $G(B_{r,s})$ consisting of elements $p\in G(B_{r,s})$ satisfying
\begin{itemize}
\item[P1.] $p^2=1\in \operatorname{Cl}(\mathbb{R}^{r,s})$;
\item[P2.] if $\langle X,X\rangle_{\mathfrak v}>0$ \big($\langle X,X\rangle_{\mathfrak v}<0$\big) then $\langle J_p(X),J_p(X)\rangle_{\mathfrak v}>0$ \big($\langle J_p(X),J_p(X)\rangle_{\mathfrak v}<0$\big).
\end{itemize}
Elements $p\in\mathcal S$ are called {\it positive involutions}. We denote the generating set for the maximal subgroup $\mathcal S$ by $PI$ and number of elements in $PI$ by $\ell=\ell(r,s)$. The set of generators $PI$ for $\mathcal S$ is not unique, but the number of involutions $\ell=\ell(r,s)$ in $PI$ is unique for the maximal subgroup~$\mathcal S$. As an example of a set $PI$ for the purpose of the present work we list the minimal length positive involutions, which can be  classified in the following types.
\begin{equation}\label{eqT1T2}
    \begin{array}{lll}
        &\text{type}\  T_1\,\,
        \begin{cases}
            p=Z_{i_{1}}Z_{i_{2}}Z_{i_{3}}Z_{i_{4}}, ~\text{where all $Z_{i_{k}}$ are positive basis vectors;}\\
            p=Z_{i_{1}}Z_{i_{2}}Z_{i_{3}}Z_{i_{4}}, ~\text{where all $Z_{i_{k}}$ are negative basis vectors;}\\
            p=Z_{i_{1}}Z_{i_{2}}Z_{i_{3}}Z_{i_{4}},
            ~\text{where two $Z_{i_{k}}$ are positive and two $Z_{i_{l}}$}\\
            \text{\qquad\qquad  \qquad\qquad\qquad\qquad\qquad\qquad
                are negative basis vectors;}
        \end{cases}\\%\end{equation}%\begin{equation}
        &\text{type}\ T_2\,\,
        \begin{cases}
            q=Z_{i_{1}}Z_{i_{2}}Z_{i_{3}},~\text{where all $Z_{i_{k}}$ are positive basis vectors;}\\
            q=Z_{i_{1}}Z_{i_{2}}Z_{i_{3}},~\text{where one $Z_{i_{k}}$ is positive and two $Z_{i_{l}}$}\\
            \text{\qquad\qquad \qquad\qquad\qquad \qquad\qquad\qquad
                are negative basis vectors.}
        \end{cases}
    \end{array}
\end{equation}
Here we always assume that $i_k\neq i_m$ for $k\neq m$.
A combinatorial computation shows that generally positive involutions contain either $3 \mod 4$ or $4\mod 4$ basis vectors,~\cite{FuMa23}.
\begin{prop}\label{prop:basis}\cite[Section 3]{FuMa23}
Let $\mathcal{S}$, $PI=\{p_{1},\ldots,p_{\ell}\}$ and $E^{+1}(p_k)=\{X\in\mathfrak v\mid J_{p_k}(X)=X\}$ be given. Then the intersection $E^{+1}_{r,s}=\bigcap_{k=1}^{\ell}E^{+1}(p_k)$ over $p_k\in\mathcal S$
    contains a non-null vector $v$. Moreover, there is a set
    $\Sigma\subset G(B_{r,s})$ such that the family
    $\{J_{\sigma}v\}_{\sigma\in \Sigma}$, $\|v\|^2=1$, is an orthogonal invariant basis for $\mathfrak v$.
\end{prop}

Consider the following example of pseudo $H$-type Lie algebras $\mathfrak n_{\mu,\nu}$, $(\mu,\nu)\in\{(8,0),(0,8),(4,4)\}$ with the minimal admissible module $\mathfrak v_{\mu,\nu}$. Let us choose the orthonormal basis $B_{\mu,\nu}=\{\zeta_k\}_{k=1}^{8}$ on the center of $\mathfrak n_{\mu,\nu}$ such that
\begin{equation}\label{eq:zeta-basis}
    \begin{array}{lllll}
&\langle \zeta_k,\zeta_k\rangle_{8,0}=-\langle \zeta_k,\zeta_k\rangle_{0,8}=1,\quad &k=1,\ldots,8,
\\
&\langle \zeta_k,\zeta_k\rangle_{4,4}=-\langle \zeta_{k+4},\zeta_{k+4}\rangle_{4,4}=1,\quad &k=1,\ldots,4.
\end{array}
\end{equation}

The set $PI_{\mu,\nu}$ generating the maximal subgroup $\mathcal S\subset G(B_{\mu,\nu})$ of positive involutions  consists of four elements and it is given by
\begin{equation}\label{mcpi2}
p_{1}=\zeta_1\zeta_2\zeta_3\zeta_4,\quad p_{2}=\zeta_1\zeta_2\zeta_5\zeta_6,\quad
p_{3}=\zeta_1\zeta_2\zeta_7\zeta_8,\quad p_{4}=\zeta_1\zeta_3\zeta_5\zeta_7.
\end{equation}
%{\color{magenta} It would be better to recall what are $\zeta_i$. Is it any (or some special) basis in the center? Its elements are considered as operators via the map $J$?}

The dimension of minimal admissible modules $\mathfrak v_{\mu,\nu}$ equals 16 and the modules
are decomposed into $16$ one dimensional common eigenspaces of
four involutions $p_{k}$, $k=1,2,3,4$ under the action of the maps $J_{p_k}$. We denote
\begin{equation}\label{eq:decom08e}
E^{+1}_{\mu,\nu}=\{X\in \mathfrak v_{\mu,\nu}:\ J_{p_k}(X)=X,\ k=1,2,3,4\}.
\end{equation}
Let $v\in E^{+1}_{\mu,\nu}$ be such that $\langle v, v\rangle_{\mathfrak v_{\mu,\nu}}=1$. Then other common
eigenspaces
are spanned by $J_{\zeta_{i}}(v)$, $i=1,\ldots, 8$,
and $J_{\zeta_1}J_{\zeta_j}(v)$, $j=2,\ldots,8$. Hence we have
\begin{equation}\label{eq:decom08}
\mathfrak v_{\mu,\nu}=E^{+1}_{\mu,\nu}\bigoplus_{i=1}^{8}\,J_{\zeta_i}(E^{+1}_{\mu,\nu})
    \bigoplus_{j=2}^{8} J_{\zeta_1}J_{\zeta_{j}}(E^{+1}_{\mu,\nu}).
\end{equation}
The basis
\begin{equation}\label{eq:basis-munu}
\begin{array}{lllllll}
&v_1=v,\quad &v_2=J_{\zeta_{1}}v,\quad &v_3=J_{\zeta_{2}}v,\quad &v_4=J_{\zeta_{3}}v,
\\
&v_5=J_{\zeta_{4}}v,\quad & v_6=J_{\zeta_{5}}v,\quad &v_7=J_{\zeta_{6}}v,\quad &v_8=J_{\zeta_{7}}v,
\\
&v_9=J_{\zeta_{8}}v,\quad & v_{10}=J_{\zeta_{1}}J_{\zeta_{2}}v,\quad &v_{11}=J_{\zeta_{1}}J_{\zeta_{3}}v,\quad &v_{12}=J_{\zeta_{1}}J_{\zeta_{4}}v,
\\
&v_{13}=J_{\zeta_{1}}J_{\zeta_{5}}v,\quad &v_{14}=J_{\zeta_{1}}J_{\zeta_{6}}v,\quad &v_{15}=J_{\zeta_{1}}J_{\zeta_{7}}v,\quad
&v_{16}=J_{\zeta_{1}}J_{\zeta_{8}}v
\end{array}
\end{equation}
is an orthonormal invariant basis for $\mathfrak v_{\mu,\nu}$. The set $\Sigma$ mentioned in Proposition~\ref{prop:basis} is the following
$$
\Sigma=\{\zeta_k,\ k=1,\ldots,8,\ \ \zeta_1\zeta_i,\ i=2,\ldots,8\}.
$$

It is well known that the Clifford algebras $\operatorname{Cl}(\mathbb{R}^{r,s})$ admit the Atiyah--Boot periodicity~\cite{ABS64}:
$$
\operatorname{Cl}(\mathbb{R}^{r+8,s})=\operatorname{Cl}(\mathbb{R}^{r,s})\otimes\operatorname{Cl}(\mathbb{R}^{8,0}),\quad
\operatorname{Cl}(\mathbb{R}^{r,s+8})=\operatorname{Cl}(\mathbb{R}^{r,s})\otimes\operatorname{Cl}(\mathbb{R}^{0,8}).
$$
$$
\operatorname{Cl}(\mathbb{R}^{r+4,s+4})=\operatorname{Cl}(\mathbb{R}^{r,s})\otimes\operatorname{Cl}(\mathbb{R}^{4,4}).
$$
This periodicity has the following affect on the structure of the $H$-type Lie algebras. Let $\mathfrak v_{r,s}$ be a minimal admissible module in $\mathfrak n_{r,s}=\mathfrak v_{r,s}\oplus \mathfrak z$ and $\ell(r,s)$ the number of positive involutions in the generating set for the maximal subgroup $\mathcal S\subset G(B_{r,s})$. Then
$$
\dim(\mathfrak v_{r+\mu,s+\nu})=16\dim(\mathfrak v_{r,s}),\quad (\mu,\nu)\in\{(8,0),(0,8),(4,4)\},
$$
$$
\ell(r+\mu,s+\nu) = \ell(r,s)+\ell(\mu,\nu)=\ell(r,s)+4.
$$
The tensor product
\begin{equation}\label{tensor product rep to direct sum}
\mathfrak v_{r,s}\otimes \mathfrak v_{\mu,\nu}
    =(\mathfrak v_{r,s}\otimes E_{\mu,\nu}^{+1})\bigoplus_{i=1}^{8}\big(\mathfrak v_{r,s}\otimes J_{\zeta_i}(E_{\mu,\nu}^{+1})\big)\bigoplus_{j=2}^{8}\big(\mathfrak v_{r,s}\otimes J_{\zeta_1}J_{\zeta_{j}}(E_{\mu,\nu}^{+1})\big)
\end{equation}
is a minimal admissible module $\mathfrak v_{r+\mu,s+\nu}$ of the Clifford algebra $\operatorname{Cl}(\mathbb{R}^{r+\mu,s+\nu})$.

Conversely, if $\mathfrak v_{r+\mu,s+\nu}$ is a minimal admissible module of
$\operatorname{Cl}(\mathbb{R}^{r+\mu,s+\nu})$, then the common 1-eigenspace $E^{+1}\subset \mathfrak v_{r+\mu,s+\nu}$ of the involutions
$J_{p_{k}}$, $k=1,2,3,4$ from~\eqref{mcpi2} can be considered as a minimal admissible module $\mathfrak v_{r,s}$ of the algebra
$\operatorname{Cl}(\mathbb{R}^{r,s})$. The action of the Clifford algebra $\operatorname{Cl}(\mathbb{R}^{r,s})$ on
$E^{+1}$ is the restricted action of $\operatorname{Cl}(\mathbb{R}^{r+\mu,s+\nu})$ obtained by the natural inclusion
$\operatorname{Cl}(\mathbb{R}^{r,s})\subset \operatorname{Cl}(\mathbb{R}^{r+\mu,s+\nu})$.

The group of automorphisms and the isometries of the $H$-type Lie algebras also has periodic structure, see~\cite{Rie82,Saal96,FuMa21}.

\section{Pseudo-Riemannian $H$-type nilmanifolds $N_{r,s}$ with $1\leq r+s\leq 3$}\label{sec.example}

To find out whether a pseudo $H$-type Lie group $N_{r,s}$, $1\leq r+s\leq 3$ is geodesic orbit we apply Propositions~\ref{gonil1n.2} to its Lie
algebra $\mathfrak{n}_{r,s}=\mathfrak{z}\oplus \mathfrak{v}$. We choose an orthonormal basis $Z_1,Z_2,\dots,Z_{r+s}$ for $\mathfrak z$ and compute the operators $J_{Z_k}$, $1\leq k \leq r+s$, which constitute the basis for $\mathbf{V}=J(\mathfrak z)\subset \mathfrak{so}(l,l)$.
Then the products $J_{Z_k}J_{Z_l}$,  $1\leq k <l \leq r+s$, form a basis for the subalgebra $[\mathbf{V},\mathbf{V}]$. We find also a basis for the centralizer $\mathbf{Z}$ of $\mathbf{V}$ in $\mathfrak{so}(l,l)$ and obtain the base for the normalizer $\mathbf{N}= [\mathbf{V},\mathbf{V}]\oplus\mathbf{Z}$ of $\mathbf{V}$ in $\mathfrak{so}(l,l)$.
After this auxiliary computation we check the {\it transitive normalizer condition} for $\mathbf{V}$, i.e.
for any $X\in \mathfrak{v}$ and any $Z\in \mathbf{V}$, we are looking for $B\in \mathbf{N}$ such that $[B,Z]=0$ and $B(X)=Z(X)$.

Note that up to similarity and the action of the isotropy subgroup, we have only three classes in the center $\mathfrak{z}$ of $\mathfrak{n}$.
They are represented by some  $Z_1$ (a positive vector with $\langle Z_1,Z_1\rangle >0$), $Z_2$ (a negative vector with $\langle Z_2,Z_2\rangle <0$),
and $Z_3$ (a null vector $\langle Z_3,Z_3\rangle =0$).

Indeed, the group $O(r,s)$ with the Lie algebra $\mathfrak{so}(r,s)$, that is a linear span of the operators $[J_{Z_i},J_{Z_j}]$, $1\leq i <j \leq r+s$,
acts naturally by automorphisms on the centre $\mathfrak{z}$. It is well known that $O(r,s)$ acts transitively on every hyper-quadric
$Q(r)=\{Z \in \mathfrak{z} \,|\,\langle Z, Z\rangle_{r,s}=\rho\}$ with $\rho \neq 0$, see e.g.~\cite[page 239]{ONeill} or~\cite[Theorem 2.4.4]{Wolf2011}.

Let $\mathfrak{z}= \mathfrak{z}_p \oplus \mathfrak{z}_n$ be an orthogonal decomposition, where the restriction of the scalar product on $\mathfrak{z}_p$ and $\mathfrak{z}_n$ are positive and negative definite, respectively. Let us suppose that $Z,\overline{Z} \in \mathfrak{z}$ and $\langle Z, Z\rangle_{r,s}=\langle \overline{Z}, \overline{Z} \rangle_{r,s}=0$. If
$\langle Z\vert_{\mathfrak z_p}, Z\vert_{\mathfrak z_p}\rangle_{r,s}=\langle \overline{Z}\vert_{\mathfrak z_p}, \overline{Z}\vert_{\mathfrak z_p} \rangle_{r,s}=\rho\neq 0$ or, equivalently,
$\langle Z\vert_{\mathfrak z_n}, Z\vert_{\mathfrak z_n}\rangle_{r,s}=\langle \overline{Z}\vert_{\mathfrak z_n}, \overline{Z}\vert_{\mathfrak z_n} \rangle_{r,s}=-\rho\neq 0$, where $Z\vert_{\mathfrak z_p}$ and $Z\vert_{\mathfrak z_p}$ are the components
of vector $Z$ in $\mathfrak{z}_p$ and $\mathfrak{z}_n$,
then it is easy to see  that there is $Q\in O(r,0)\times O(0,s) \subset O(r,s)$ such that $Q(Z\vert_{\mathfrak z_p})=\overline{Z}\vert_{\mathfrak z_p}$ and $Q(Z\vert_{\mathfrak z_n})=\overline{Z}\vert_{\mathfrak z_n}$.

Hence, up to a similarity, it suffices to consider only one point in every of the following hyperquadrics: $Q(1)$, $Q(-1)$, $Q(0)$ (non-trivial in the last case).

Suppose that we fix $Z \in \mathbf{V}$, then we do not need to check all $X \in \mathfrak{v}$. Indeed,
if $Q \in \exp(\mathbf{N})$ (or, more generally, $Q$ is an isometric automorphism of the Lie algebra $\mathfrak{n}$) such that $[Q,Z]=0$,
then the equality $B(X)=Z(X)$ is equivalent to the following one: $QBQ^{-1}(QX)=QZQ^{-1}(QX)=Z(QX)$.
Note that $[B,Z]=0$ implies $[QBQ^{-1},Z]=0$, therefore, we can find $B \in \mathbf{N}$ such that $[B,Z]=0$ and $B(X)=Z(X)$ if and only if we can find
$\widetilde{B} \in \mathbf{N}$ such that $[B,Z]=0$ and $B(QX)=Z(QX)$.

Consider a Lie  subalgebra $\mathbf{N}^Z=\left\{B \in \mathbf{N}\,|\, [B,Z]=0\right\}$ of $\mathbf N$. Any orbit $\operatorname{Orb}$
in $\mathfrak{v}$
under the action of  $\mathbf{N}^Z$ consists of equivalent elements in the sense that if $X_1,X_2 \in \operatorname{Orb}$, then there is $B_1\in \mathbf{N}$ such that $[B_1,Z]=0$ and $B_1(X_1)=Z(X_1)$ if and only if there is $B_2\in \mathbf{N}$ such that $[B_2,Z]=0$ and $B_2(X_2)=Z(X_2)$.

In what follows, we omit all classical cases with the signature $(r,0)$.

\subsection{Pseudo-Riemannian $H$-type nilmanifold $N_{0,1}$}
We denote by $N_{0,1}$ the $H$-type nilmanifold with the Lie algebra isometric to
$\mathbb{R}^{0,1}\oplus\mathbb{R}^{n,n}$. This will be the only example of $H$-type Lie algebra having the module $(\mathfrak v, \langle.\,,.\rangle_{n,n})\cong \mathbb{R}^{n,n}$ which is not minimal dimensional, but rather the direct sum of $n$ minimal dimensional modules isometric to $\mathbb R^{1,1}$.
We choose an orthonormal basis $\{X_1,\ldots X_n,Y_1\ldots,Y_n,Z\}$ satisfying
$$
\langle X_i,X_i\rangle_{n,n}=-\langle Y_i,Y_i\rangle_{n,n}=1,\quad \langle Z,Z\rangle_{0,1}=-1.
$$
%Recall that $A\in\mathfrak{so}(n,n)$ if
%$$
%\eta A^T\eta=-A,\quad\text{where}\quad\eta=\begin{pmatrix}\Id_n&0\\0&-\Id_n\end{pmatrix}.
%$$
Then
$$
J_Z=\begin{pmatrix}0&\Id_n\\\Id_n&0\end{pmatrix}\in \mathfrak{so}(n,n),\quad J_Z^2=-\langle Z,Z\rangle_{1,0} \Id_{2n}=\Id_{2n}.
$$
It is clear that $J_Z$
spans the one dimensional subspace $\mathbf{V}=J(\mathfrak{z})\subset \mathfrak{so}(n,n)$. The space $\mathbf{V}$ is an abelian subalgebra of
$\mathfrak{so}(n,n)$ and
$$
[J(Z_1), J(Z_2)] = [J(aZ), J(bZ)] =ab[J(Z), J(Z)] =0=J\big(\tau_{Z_1}(Z_2)\big),\quad a,b\in\mathbb{R},\ a\neq 0,\ b\neq 0,
$$
for the operator $\mathfrak{so}(\mathfrak{z})\ni\tau_{Z_1}\equiv 0$ for any $Z_1\in\mathfrak{z}$.
Therefore, we see that $N_{0,1}$ is naturally reductive by Theorem~\ref{th:Ovanto}, hence it is a geodesic orbit nilmanifold.

\subsection{Pseudo-Riemannian $H$-type nilmanifold $N_{1,1}$}

The pseudo-Riemannian $H$-type nilmanifold $N_{1,1}$ is the $H$-type Lie group of dimension $6$ with the Lie algebra isometric to $\mathbb{R}^{1,1}\oplus\mathbb{R}^{2,2}$ and satisfying
$$
[X_1,X_2]=[X_3,X_4]=Z_1,\quad [X_1,X_3]=[X_2,X_4]=-Z_2
$$
with respect to an orthonormal basis $\{X_1,\ldots, X_4,Z_1,Z_2\}$ such that
$$
\langle X_k,X_k\rangle_{2,2}=-\langle X_i,X_i\rangle_{2,2}=1,\quad k=1,2,\ \ i=3,4,\quad
\langle Z_1,Z_1\rangle_{1,1}=-\langle Z_2,Z_2\rangle_{1,1}=1.
$$
The maps
$$
J_{Z_1}=\begin{pmatrix}
0&-1&0&0
\\
1&0&0&0
\\
0&0&0&-1
\\
0&0&1&0
\end{pmatrix},\qquad
J_{Z_2}=
\begin{pmatrix}
0&0&1&0
\\
0&0&0&-1
\\
1&0&0&0
\\
0&-1&0&0
\end{pmatrix}
$$
span a 2-dimensional subspace $\mathbf{V}\subset\mathfrak{so}(2,2)$.
We calculate
$$
[aJ_{Z_1}+bJ_{Z_2}, cJ_{Z_1}+dJ_{Z_2}]=(ad-bc)[J_{Z_1},J_{Z_2}]=(ad-bc)
\begin{pmatrix}
0&0&0&2
\\
0&0&2&0
\\
0&2&0&0
\\
2&0&0&0
\end{pmatrix}
\notin \operatorname{span}\{J_{Z_1},J_{Z_2}\}.
$$
Thus the vector space $\mathbf{V}$ is not a Lie subalgebra of $\mathfrak{so}(2,2)$ and therefore $N_{1,1}$ is not a naturally reductive pseudo-Riemannian nilmanifold.
It is easy to check that the normalizer $\mathbf{N}=\mathbf{Z} \oplus [\mathbf{V},\mathbf{V}]$ consists of matrices of the following type:
\begin{equation}\label{eq:BB}
B=\begin{pmatrix}
0&-b_3&b_2&2a-b_1
\\
b_3&0&2a+b_1&b_2
\\
b_2&2a+b_1&0&b_3
\\
2a-b_1&b_2&-b_3&0
\end{pmatrix}, \qquad  a,b_1,b_2,b_3 \in \mathbb{R}.
\end{equation}
Here, the variable $a$ corresponds to $[\mathbf{V},\mathbf{V}]$, whereas $b_1,b_2,b_3$ parameterize $\mathbf{Z}$.

Let us consider
$Y=X_1+X_2+X_3+X_4\in \mathfrak{v}$ and $Z=J_{Z_1}\in \mathbf{V}$. Assume that $N_{1,1}$ is geodesic orbit. Then
there is $B \in \mathbf{N}$ such that $[B,Z]=[B,J_{Z_1}]=0$ and $B(Y)=Z(Y)$. Since $\langle Y,Y \rangle+\langle Z_1,Z_1 \rangle=\langle Z_1,Z_1 \rangle \neq 0$, then we set $k=0$, see Proposition~\ref{gonil1n.2}).
The first condition implies $a=0$.
The second condition is equivalent to the linear system of equations:
$$
\left(\begin{array}{rrrr}
2&-1&1&-1
\\
2&1&1&1
\\
2&1&1&1
\\
2&-1&1&-1
\\
\end{array}\right)
\left(\begin{array}{c}
a
\\
b_1
\\
b_2
\\
b_3
\end{array}\right)=
\left(\begin{array}{r}
-1
\\
1
\\
-1
\\
1
\end{array}\right).
$$
Since the system has no solution (it suffices to compare the second and the third equations in the system), $N_{1,1}$ is not geodesic orbit.

\begin{remark}\label{re.N_1.1}
The pseudo-Riemannian $H$-type nilmanifold $N_{1,1}$ was studied in \cite{Bar}. It was proved in~\cite[Theorem 4.4]{Bar} that
$N_{1,1}$ is an almost geodesic orbit pseudo-Riemannian space
which is not geodesic orbit; and it admits null homogeneous
geodesics with non-zero parameter~$k$.

We are going to reproduce this result in our notation.
If we consider a vector $Z=\alpha Z_1+\beta Z_2 \in \mathfrak{z}$ and a vector $Y=y_1X_1+y_2X_2+y_3X_3+y_4X_4 \in \mathfrak{v}$ with
$\alpha, \beta, y_k \in \mathbb{R}$, such that $y_1^2 + y_2^2 - y_3^2 - y_4^2 \neq 0$, then we easily
can find $B \in \mathbf{N}$ such that $[B,J_Z]=0$ and $B(Y)=Z(Y)$. Indeed, it suffices to set $k=0$ and  in the matrix $B$ in~\eqref{eq:BB} to choose
$a=0$, and
$$b_1 = 2 \cdot \frac{\alpha(y_1y_3+y_2y_4) - \beta(y_1y_2+y_3y_4)}{y_1^2 + y_2^2 - y_3^2 - y_4^2},\quad
b_2 = \frac{2\alpha(y_1y_4 - y_2y_3) - \beta(y_1^2 -y_2^2 -y_3^2 +y_4^2)}{y_1^2 + y_2^2 - y_3^2 - y_4^2},
$$
$$
b_3 = \frac{2\beta (y_1y_4 + y_2y_3)-\alpha(y_1^2 + y_2^2 + y_3^2 + y_4^2)}{y_1^2 + y_2^2 - y_3^2 - y_4^2}.
$$
Therefore, we conclude that $N_{1,1}$ is almost geodesic orbit, since a unique restriction for the existence of a solution to $[B,J_Z]=0$ and $B(Y)=Z(Y)$ is $y_1^2 + y_2^2 - y_3^2 - y_4^2 \neq 0$.

On the other hand, let us take a null initial vector, by setting $\alpha_1=-\alpha_2 = y_1=y_3=1$, and $y_2=y_4 = t$ for any $t\in\mathbb R$, and set $k = -2 \cdot \frac{t - 1}{t + 1}$, $t\neq-1$. It is not difficult to see that the matrix $B$ in~\eqref{eq:BB}, where
$$a = \frac{t - 1}{2(t + 1)},\quad
b_1=s,\quad b_2 = 3\cdot \frac{t - 1}{t + 1},\quad
b_3=-s,\quad s\in\mathbb R
$$
solves the equations
$[B,J_Z]+kJ_Z=0$ and $B(Y)+kY=Z(Y)$ for $Z=Z_1-Z_2$, $Y=X_1+tX_2+X_3+tX_4$. Thus, if $|t|\neq 1$, we get a homogeneous geodesic with $k \neq 0$.
\end{remark}

\subsection{Pseudo-Riemannian $H$-type nilmanifold $N_{0,2}$}

The pseudo-Riemannian $H$-type nilmanifold $N_{0,2}$ is the pseudo $H$-type Lie group of dimension $6$ whose Lie algebra carries a scalar product isometric to $\mathbb{R}^{0,2}\times\mathbb{R}^{2,2}$. The Lie algebra has the commutation relations
$$
[X_1,X_3]=[X_2,X_4]=-Z_1,\quad [X_1,X_4]=-[X_2,X_3]=-Z_2
$$
with respect to an orthonormal basis $\{X_1,\ldots, X_4,Z_1,Z_2\}$
such that
$$
\langle X_k,X_k\rangle_{2,2}=-\langle X_i,X_i\rangle_{2,2}=1,\quad k=1,2,\ \ i=3,4,\quad
\langle Z_1,Z_1\rangle_{0,2}=\langle Z_2,Z_2\rangle_{0,2}=-1.
$$
We calculate
$$
J_{Z_1}=\begin{pmatrix}
0&0&1&0
\\
0&0&0&1
\\
1&0&0&0
\\
0&1&0&0
\end{pmatrix}
\in\mathfrak{so}(2,2)\quad \mbox{ and } \quad
J_{Z_2}=
\begin{pmatrix}
0&0&0&1
\\
0&0&-1&0
\\
0&-1&0&0
\\
1&0&0&0
\end{pmatrix}
\in\mathfrak{so}(2,2),
$$
$$
[aJ_{Z_1}+bJ_{Z_2}, cJ_{Z_1}+dJ_{Z_2}]=(ad-bc)
\begin{pmatrix}
0&-2&0&0
\\
2&0&0&0
\\
0&0&0&2
\\
0&0&-2&0
\end{pmatrix}
\notin \operatorname{span}\{J_{Z_1},J_{Z_2}\}.
$$
Thus $N_{0,2}$ is not a naturally reductive manifold.
The normalizer $\mathbf{N}=\mathbf{Z} \oplus [\mathbf{V},\mathbf{V}]$ consists of:
$$
\begin{pmatrix}
0&-2a+b_3&-b_2&b_1
\\
2a-b_3&0&b_1&b_2
\\
-b_2&b_1&0&2a+b_3
\\
b_1&b_2&-2a-b_3&0
\end{pmatrix}, \qquad  a,b_1,b_2,b_3 \in \mathbb{R}.
$$
Here, the variable $a$ corresponds to $[\mathbf{V},\mathbf{V}]$, whereas $b_1,b_2,b_3$ parameterize $\mathbf{Z}$.

Assume that $N_{0,2}$ is geodesic orbit and pick up
$Y=y_1X_1+y_2X_2+y_3X_3+y_4X_4\in \mathfrak{v}$. Moreover, without loss of generality,
we take $Z= J_{Z_1}\in \mathbf{V}$.
The condition $[B,Z]=0$ for some $B \in \mathbf{N}$ implies $a =0$.
Then the condition $B(Y)=Z(Y)$ is equivalent to the linear system of equations:
$$
\left(\begin{array}{rrr}
y_2&-y_1&y_4
\\
y_1&y_2&-y_3
\\
y_4&-y_3&y_2
\\
y_3&y_4&-y_1
\\
\end{array}\right)
\left(\begin{array}{c}
b_1
\\
b_2
\\
b_3
\end{array}\right)=
\left(\begin{array}{r}
y_1
\\
y_2
\\
y_3
\\
y_4
\end{array}\right).
$$
If $y_1^2 + y_2^2 \neq y_3^2 + y_4^2$, then we obtain the following solution of this system:
$$
b_1 = 2\frac{y_1y_2 - y_3y_4}{y_1^2 + y_2^2 - y_3^2 - y_4^2}, \quad
b_2 = -\frac{y_1^2 - y_2^2 - y_3^2 + y_4^2}{y_1^2 + y_2^2 - y_3^2 - y_4^2},\quad
b_3 = -2\frac{y_1y_4 - y_2y_3}{y_1^2 + y_2^2 - y_3^2 - y_4^2}.
$$
On the other hand, if $(y_1,y_2,y_3,y_4)=(3,4,5,0)$, then we obtain the system
$$
\left(\begin{array}{rrr}
4&-3&0
\\
3&4&-5
\\
0&-5&4
\\
5&0&-3
\\
\end{array}\right)
\left(\begin{array}{c}
b_1
\\
b_2
\\
b_3
\end{array}\right)=
\left(\begin{array}{r}
3
\\
4
\\
5
\\
0
\end{array}\right),
$$
that have no solution. The rank of the basic matrix of the system is 2, and the rank of the extended matrix of the system is 3.
In this case $k=0$ since we consider a non-null vector.
We conclude that $N_{0,2}$ is not geodesic orbit.

\subsection{Pseudo-Riemannian $H$-type nilmanifold $N_{1,2}$}

The $H$-type nilmanifold $N_{1,2}$ is the $H$-type group of dimension $7$, the Lie algebra is isometric to $\mathbb{R}^{1,2}\times\mathbb{R}^{2,2}$ and has non-vanishing commutation relations
$$
[X_1,X_2]=-[X_3,X_4]=Z_1,\quad [X_1,X_3]=-[X_2,X_4]=Z_2,\quad [X_1,X_4]=[X_2,X_3]=Z_3
$$
in an orthonormal basis satisfying
$$
\langle X_k,X_k\rangle_{2,2}=-\langle X_i,X_i\rangle_{2,2}=1,\quad k=1,2,\ \ i=3,4,\quad
$$
$$
\langle Z_1,Z_1\rangle_{1,2}=-\langle Z_2,Z_2\rangle_{1,2}=-\langle Z_3,Z_3\rangle_{1,2}=1.
$$
Then, in this basis we obtain the matrix representations
$$
J_{Z_1}=
\begin{pmatrix}
0&-1&0&0
\\
1&0&0&0
\\
0&0&0&-1
\\
0&0&1&0
\end{pmatrix},\quad
J_{Z_2}=
\begin{pmatrix}
0&0&1&0
\\
0&0&0&-1
\\
1&0&0&0
\\
0&-1&0&0
\end{pmatrix},
\quad
J_{Z_3}=
\begin{pmatrix}
0&0&0&1
\\
0&0&1&0
\\
0&1&0&0
\\
1&0&0&0
\end{pmatrix}.
$$
Due to the commutation relations
$$
[J_{Z_1},J_{Z_2}]=2J_{Z_3}, \quad [J_{Z_1},J_{Z_3}]=-2J_{Z_2},\quad [J_{Z_2},J_{Z_3}]=-2J_{Z_1},
$$
we conclude that $\mathbf {V}=\operatorname{span}\{J_{Z_1},J_{Z_2},J_{Z_3}\}\subset \mathfrak{so}(2,2)$ is a Lie subalgebra by Theorem~\ref{th:Ovanto}.
Moreover, the operators $\tau_{Z_i}\in\mathfrak{so}(\mathfrak{z})$ for $i,j=1,2,3$ (recall that $[J(Z_i), J(Z_j)] = J\big(\tau_{Z_i}(Z_j)\big)$,
are given by:
$$
\tau_{Z_1}=
2\begin{pmatrix}
0&0&0
\\
0&0&-1
\\
0&1&0
\end{pmatrix},\quad
\tau_{Z_2}=
2\begin{pmatrix}
0&0&-1
\\
0&0&0
\\
-1&0&0
\end{pmatrix},\quad
\tau_{Z_3}=
2\begin{pmatrix}
0&1&0
\\
1&0&0
\\
0&0&0
\end{pmatrix}.
$$
We conclude that $H$-type nilmanifold $N_{1,2}$ is naturally reductive, hence, geodesic orbit.

\subsection{Pseudo-Riemannian $H$-type nilmanifold $N_{2,1}$}

The $H$-type nilmanifold $N_{2,1}$ is the $H$-type group of dimension $11$ with a Lie algebra $\mathfrak n_{2,1}=\mathfrak z\oplus\mathfrak v$ isometric to
 $\mathbb R^{2,1}\times\mathbb R^{4,4}$.
We define an orthonormal basis $\{V_1,\ldots, V_8,Z_1,Z_2,Z_3\}$ by taking $v\in \mathbb R^{4,4}$ with $\|v\|^2=1$ and setting
$$
\begin{array}{lllllll}
&V_1=v,\quad &V_2=J_{Z_1}v,\quad &V_3=J_{Z_2}v,\quad &V_4=J_{Z_1}J_{Z_2}v,\quad
\\
&V_5=J_{Z_3}v,\quad &V_6=J_{Z_1}J_{Z_3}v,\quad &V_7=J_{Z_2}J_{Z_3}v,\quad &V_8=J_{Z_1}J_{Z_2}J_{Z_3}v.
\end{array}
$$
Writing the operators
$$
J_{Z_1}=\left(\begin{smallmatrix}
    0&-1&0&0&0&0&0&0
    \\
    1&0&0&0&0&0&0&0
    \\
    0&0&0&-1&0&0&0&0
    \\
    0&0&1&0&0&0&0&0
    \\
    0&0&0&0&0&-1&0&0
    \\
    0&0&0&0&1&0&0&0
    \\
    0&0&0&0&0&0&0&-1
    \\
    0&0&0&0&0&0&1&0
\end{smallmatrix}\right),\quad
J_{Z_2}=
\left(\begin{smallmatrix}
    0&0&-1&0&0&0&0&0
    \\
    0&0&0&1&0&0&0&0
    \\
    1&0&0&0&0&0&0&0
    \\
    0&-1&0&0&0&0&0&0
    \\
    0&0&0&0&0&0&-1&0
    \\
    0&0&0&0&0&0&0&1
    \\
    0&0&0&0&1&0&0&0
    \\
    0&0&0&0&0&-1&0&0
\end{smallmatrix}\right),
\quad
J_{Z_3}=
\left(\begin{smallmatrix}
    0&0&0&0&1&0&0&0
    \\
    0&0&0&0&0&-1&0&0
    \\
    0&0&0&0&0&0&-1&0
    \\
    0&0&0&0&0&0&0&1
    \\
    1&0&0&0&0&0&0&0
    \\
    0&-1&0&0&0&0&0&0
    \\
    0&0&-1&0&0&0&0&0
    \\
    0&0&0&1&0&0&0&0
\end{smallmatrix}\right)
$$
and calculating their commutators, we see that
$\mathbf{V}=\spann\{J_{Z_1},J_{Z_2},J_{Z_3}\}\subset \mathfrak{so}(4,4)$ is not a Lie subalgebra of $\mathfrak{so}(4,4)$ and, therefore,
$N_{2,1}$ is not a naturally reductive manifold.

The normalizer $\mathbf{N}=\mathbf{Z} \oplus [\mathbf{V},\mathbf{V}]$ consists of matrices of the following type:
{\small
$$
\left(\begin{matrix}
0&b_6&-b_5&-2a_1+b_4&b_3&2a_2+b_2&2a_3-b_1&0
\\
-b_6&0&-2a_1+b_4&-b_5&2a_2-b_2&b_3&0&2a_3-b_1
\\
b_5&2a_1-b_4&0&-b_6&2a_3+b_1&0&b_3&-2a_2-b_2
\\
2a_1+b_4&b_5&b_6&0&0&2a_3+b_1&-2a_2+b_2&b_3
\\
b_3&2a_2-b_2&2a_3+b_1&0&0&-b_6&b_5&-2a_1-b_4
\\
2a_2+b_2&b_3&0&2a_3+b_1&b_6&0&-2a_1+b_4&2b_5
\\
2a_3-b_1&0&b_3&-2a_2+b_2&-b_5&2a_1-b_4&0&b_6
\\
0&2a_3-b_1&-2a_2-b_2&b_3&2a_1+b_4&-b_5&-b_6&0
\end{matrix}\right),
$$}
Here, the variables $a_1,a_2,a_3$ correspond to $[\mathbf{V},\mathbf{V}]$, whereas $b_1,b_2,b_3,b_4,b_5,b_6$ parameterize $\mathbf{Z}$.

Assume that $N_{2,1}$ is geodesic orbit and choose $Y=V_1+V_5\in \mathfrak{v}$ and $Z= J_{Z_1}\in \mathbf{V}$.
In this case $k=0$ since we consider a non-null vector.
If $B \in \mathbf{N}$ then $[B,Z]=0$ implies $a_1=a_2=0$. The second condition $B(Y)=Z(Y)$ gives
$$
(b_3,-b_2-b_6,b_1+b_5+2a_3,b_4,b_2,b_2+b_6,-b_1-b_5+2a_3,b_4)^t=(0,1,0,0,0,1,0,0)^t,
$$
which leads to a contradiction: it suffices to compare the second and  and the sixth  entries in these columns.
Therefore, $N_{2,1}$ is not geodesic orbit.

\subsection{Pseudo-Riemannian $H$-type nilmanifold $N_{0,3}$}

The $H$-type Lie algebra $\mathfrak n_{0,3}$ isometric to $\mathbb{R}^{0,3}\times\mathbb{R}^{4,4}$
and has the following non-vanishing commutation relations
$$
[V_1,V_5]=[V_2,V_6]=[V_3,V_7]=[V_4,V_8]=Z_1,\quad
[V_1,V_6]=-[V_2,V_5]=-[V_3,V_8]=[V_4,V_7]=Z_2,\quad
$$
$$
[V_1,V_7]=[V_2,V_8]=-[V_3,V_5]=[V_4,V_6]=Z_3.
$$
It implies that the matrices
$$
J_{Z_1}=\left(\begin{smallmatrix}
0&0&0&0&1&0&0&0
\\
0&0&0&0&0&1&0&0
\\
0&0&0&0&0&0&1&0
\\
0&0&0&0&0&0&0&1
\\
1&0&0&0&0&0&0&0
\\
0&1&0&0&0&0&0&0
\\
0&0&1&0&0&0&0&0
\\
0&0&0&1&0&0&0&0
\end{smallmatrix}\right)
,\quad
J_{Z_2}=
\left(\begin{smallmatrix}
0&0&0&0&0&1&0&0
\\
0&0&0&0&-1&0&0&0
\\
0&0&0&0&0&0&0&-1
\\
0&0&0&0&0&0&1&0
\\
0&-1&0&0&0&0&0&0
\\
1&0&0&0&0&0&0&0
\\
0&0&0&1&0&0&0&0
\\
0&0&-1&0&0&0&0&0
\end{smallmatrix}\right),\quad
J_{Z_3}=
\left(\begin{smallmatrix}
0&0&0&0&0&0&1&0
\\
0&0&0&0&0&0&0&1
\\
0&0&0&0&-1&0&0&0
\\
0&0&0&0&0&-1&0&0
\\
0&0&-1&0&0&0&0&0
\\
0&0&0&-1&0&0&0&0
\\
1&0&0&0&0&0&0&0
\\
0&1&0&0&0&0&0&0
\end{smallmatrix}\right)
$$
form the basis of the space $\mathbf V\subset\mathfrak{so}(4,4)$.
Calculating their commutators, we show that
$\mathbf{V}$ is not a Lie subalgebra of $\mathfrak{so}(4,4)$, and therefore $N_{0,3}$ is not a naturally reductive manifold.
The normalizer $\mathbf{N}=\mathbf{Z} \oplus [\mathbf{V},\mathbf{V}]$ consists of matrices of the following type:
{\small
$$
\left(\begin{matrix}
0&-2a_1+b_6&-2a_2-b_5&-2a_3+b_4&b_3&-b_2&b_1&0
\\
2a_1-b_6&0&2a_3+b_4&-2a_2+b_5&-b_2&-b_3&0&b_1
\\
2a_2+b_5&-2a_3-b_4&0&2a_1+b_6&b_1&0&-b_3&b_2
\\
2a_3-b_4&2a_2-b_5&-2a_1-b_6&0&0&b_1&b_2&b_3
\\
b_3&-b_2&b_1&0&0&2a_1+b_6&2a_2-b_5&-2a_3+b_4
\\
-b_2&-b_3&0&b_1&-2a_1-b_6&0&2a_3+b_4&2a_2+b_5
\\
b_1&0&-b_3&b_2&-2a_2+b_5&-2a_3-b_4&0&-2a_1+b_6
\\
0&b_1&b_2&b_3&2a_3-b_4&-2a_2-b_5&2a_1-b_6&0
\end{matrix}\right),
$$}
where the variables $a_1,a_2,a_3$ correspond to $[\mathbf{V},\mathbf{V}]$, and $b_1,b_2,b_3,b_4,b_5,b_6$ parameterize $\mathbf{Z}$.

Assuming that $N_{0,3}$ is geodesic orbit, and taking $Y=\sum_{i=1}^8 y_i V_i\in \mathfrak{v}$, $Z=J_{Z_1}\in \mathbf{V}$, we obtain that for an arbitrary
$B\in \mathbf{N}$ the condition $[B,Z]=0$ implies $a_1=a_2=0$.
Now, we choose $(y_1,y_2,y_3,y_4,y_5,y_6,y_7,y_8)=(3,4,0,0,5,0,0,0)$ and the condition $B(Y)=Z(Y)$ (assuming $a_1=a_2=0$) is equivalent to the linear system of equations
$$
\left(
\begin{array}{rrrrrrr}
  0 & 3 & 0 & 0 & 0 & 5 & 0 \\
  0 & 0 & 5 & 0 & 0 & 0 & 3 \\
  -6 & 0 & 0 & 0 & 3 & 4 & 0 \\
  10 & 4 & 0 & 0 & -5 & 0 & 0 \\
  8 & -5 & 0 & 0 & 4 & -3 & 0 \\
  0 & 0 & -4 & 3 & 0 & 0 & 0 \\
  0 & 0 & -3 & -4 & 0 & 0 & -5 \\
  0 & 0 & 0 & 5 & 0 & 0 & 4 \\
\end{array}\right)
\left(\begin{array}{c}
a_3
\\
b_1
\\
b_2
\\
b_3
\\
b_4
\\
b_5
\\
b_6
\end{array}\right)=
\left(\begin{array}{r}
0\\
0\\
0\\
0\\
0\\
3
\\
4
\\
5
\end{array}\right),
$$
that does not have solutions. The rank of the basic matrix of the system is 5, and the rank of the extended matrix of the system is 6.
In this case $k=0$ since we consider a non-null vector.
Therefore, $N_{0,3}$ is not geodesic orbit.

\begin{remark}
We note that there are shorter arguments to prove that
$N_{2,1}$ and $N_{0,3}$ are not geodesic orbit.
One can use Corollary~\ref{co.totgeod.1} and the ideas of Section~\ref{se.tot.geod} to check that
$N_{1,1}$ is a totally geodesic submanifold in $N_{2,1}$
as well as  $N_{0,2}$ is a totally geodesic submanifold in $N_{0,3}$.
Recall that $N_{1,1}$ and $N_{0,2}$ are not geodesic orbit by the above considerations.
Nevertheless, we decided to show the details of working with these pseudo $H$-type groups to prepare readers
for a more sophisticated study of the $H$-type group $N_{3,4}$ in Section~\ref{sec.N3.4}.
\end{remark}

\subsection{On naturally reductive pseudo-Riemannian $H$-type nilmanifolds}

The above examples allow us to obtain a complete classification of naturally reductive pseudo $H$-type nilmanifolds~$N_{r,s}$.

\begin{prop}\label{pr.natred.1}
An $H$-type nilmanifold $N_{r,s}$ is naturally reductive if and only if \linebreak $(r,s)\in\{(1,0),(0,1),(3,0),(1,2)\}$.
\end{prop}

\begin{proof}
We know that $N_{1,0}$, $N_{0,1}$, $N_{3,0}$, and $N_{1,2}$  are naturally reductive
(see Theorem~\ref{th.class.riemmannain} and Section~\ref{sec.example}).

Now, suppose that $N_{r,s}$ is naturally reductive and $r+s >1$ (if $r+s =1$, then it is isometric to $N_{1,0}$ or to $N_{0,1}$).
Then the linear space $\mathbf{V}=J(\mathfrak{z})$ is a Lie subalgebra in
$\mathbf{L}= \mathbf{V}+[\mathbf{V},\mathbf{V}]$. Since $[\mathbf{V},[\mathbf{V},\mathbf{V}]] \subset \mathbf{V}$,
then $\mathbf{V}$ is an ideal in $\mathbf{L}$. Therefore, either $\mathbf{L}$ is not simple, or $\mathbf{V}=[\mathbf{V},\mathbf{V}]$ is a simple Lie algebra.
If $(r,s)\not\in \{(1,0),(0,1),(3,0),(1,2)\}$, then $\mathbf{L}$ is simple  by Proposition \ref{pr.triple.1} and $\mathbf{V}=[\mathbf{V},\mathbf{V}]$
by Lemma \ref{le.ideal_h1}.

Since $\dim(\mathbf{V})=r+s$ and  $\dim([\mathbf{V},\mathbf{V}])=(r+s)(r+s-1)/2$ (see Proposition \ref{pr.triple.1}), then $r+s=3$. Now, it suffices to note that
$N_{2,1}$ and $N_{0,3}$ are not geodesic orbit, hence, are not naturally reductive.
\end{proof}

It {\color{blue}will be} proved in Section~\ref{sec.N3.4}, that the pseudo $H$-type nilmanifold $N_{3,4}$ is geodesic orbit,
although it is not naturally reductive.

\section{On totally geodesic submanifolds of geodesic orbit pseudo-Riemannian manifolds}\label{se.tot.geod}

In this section we consider the relation between the geodesic orbit submanifolds and totally geodesic submanifolds and apply this to
 pseudo $H$-type Lie groups.
P.~Eberlein~\cite{Eber94.1, Eber94.2} studied totally geodesic subalgebras and totally
geodesic subgroups in nonsingular 2-step nilpotent Lie groups endowed with left-invariant Riemannian metrics.
The pseudo $H$-type Lie algebras are examples of nonsingular 2-step nilpotent Lie algebras.
It was proved in~\cite[Theorem 11]{BerNik08} that
every closed totally geodesic submanifold of a
GO Riemannian manifold is
GO itself. Let us consider a version of this result for pseudo-Riemannian manifolds.
Note that any GO pseudo-Riemannian manifold is geodesically complete.

\begin{theorem}\label{th_ps_totgeod}
Every geodesically complete totally geodesic submanifold of a
GO pseudo-Riemannian manifold is
geodesic orbit itself.
\end{theorem}

\begin{proof}
Let $N$ be a geodesically complete totally geodesic submanifold of a
GO pseudo-Riemannian manifold  $M$ (which is also
geodesically complete).  Let $U\neq 0$ be a tangent vector at some
point $x\in N$. It is enough to prove that there is a Killing vector
field $Y$ on $N$ with the following properties:

1) the value $Y(x)=U(x)$;

2) $x$ is a critical point of $\|Y\|^2$ on
$N$.

Indeed, in this case a geodesic passing through $x$ in the
direction of $U$ is an orbit of a one-dimensional isometry group
generated by the Killing field $Y$ (this one-parameter group is
defined because of the geodesic completeness of $N$), see e.g.~\cite[Exercise 10, page~259]{ONeill}.

It is known that a point $z\in M$ is a critical point of $\|Z\|^2$ for a Killing vector field $Z$ on $M$ if and only if
the integral curve of $Z$ through the point $z$ is a geodesic in $M$, see e.g.~\cite[Exercises 9 and 10, page~259]{ONeill}.
Since $M$ is a GO manifold, there is a
Killing vector field $X$ on $M$, such that $X(x)=U(x)$ with
$x$ being a critical point of $\|X\|^2$.
Now we define $Y$ to be the tangent component $\widetilde{X}$ of the Killing vector field $X$ to $N$.
According to~\cite[page~259, Exercise 7 ]{ONeill}), $\widetilde{X}$ is a Killing vector field on $N$,
and, moreover, $\widetilde{X}(x)=X(x)$.

Now we need only to prove that $x$ is a critical point of $\|\widetilde{X}\|^2$ on $N$. Let $Z=X-\widetilde{X}$ be the
normal component of the vector field $X$ on the manifold $N$. It is clear that
$\|\widetilde{X}\|^2=\|X\|^2-\|Z\|^2$ on $M$. The point $x$
is a zero point for $\|Z\|^2$, therefore, $x$ is a critical point of $\|Z\|^2$ on $N$. Consequently, $x$ is a critical
point for both functions: $\|X\|^2$ and $\|Z\|^2$
on the manifold $N$. But in this case $x$ is a critical point for
$\|\widetilde{X}\|^2$, since $\widetilde{X}(x)=U(x)\neq 0$.
This completes the proof.
\end{proof}

\begin{theorem}\label{pr_totgeod_h1}
Let $(N,g)$ be a 2-step pseudo-Riemannian nilmanifold with a non-degenerate centre. Let us consider non-degenerate subspaces
$\mathfrak{z}_1\subset \mathfrak{z}$ and $\mathfrak{v}_1 \subset \mathfrak{v}$ and
the corresponding
orthogonal decompositions $\mathfrak{z}=\mathfrak{z}_1\oplus \mathfrak{z}_2$ and $\mathfrak{v}=\mathfrak{v}_1\oplus \mathfrak{v}_2$.
Then the following assertions hold:

{\rm 1)} If $J_Z(\mathfrak{v}_1)\subset \mathfrak{v}_2$ for any $Z\in \mathfrak{z}_2$,
then $\mathfrak{n}_1:=\mathfrak{z}_1\oplus \mathfrak{v}_1$ is a Lie subalgebra of $\mathfrak{n}$.

{\rm 2)} If $J_Z(\mathfrak{v}_1)\subset \mathfrak{v}_2$ for any $Z\in \mathfrak{z}_2$ and $J_Z(\mathfrak{v}_1)\subset \mathfrak{v}_1$ for any $Z\in \mathfrak{z}_1$,
then the Lie subalgebra $\mathfrak{n}_1$ {\rm(}with the induced scalar product{\rm)} generates
a totally geodesic submanifold $(N_1,g_1)$ of $(N,g)$.
\end{theorem}

\begin{proof} The first assertion is almost obvious. Indeed, $[\mathfrak{z}_1,\mathfrak{n}_1]\subset [\mathfrak{z},\mathfrak{n}]=0$ and we have
$\langle [X,Y],Z\rangle=\langle J_Z(X),Y \rangle=0$ for all $X,Y \in \mathfrak{v}_1$ and any  $Z\in \mathfrak{z}_2$ by the condition in 1).
Hence, $[\mathfrak{n}_1,\mathfrak{n}_1]=[\mathfrak{v}_1,\mathfrak{v}_1]\subset \mathfrak{z}_1 \subset \mathfrak{n}_1$, that proves 1).

Let us show the second assertion.
If $\nabla$ is the Levi-Civita connection on $(N,g)$, then we need to prove that $\nabla_X Y \in \mathfrak{n}_1$ for all $X,Y \in \mathfrak{n}_1$, where
$X,Y$ are identified with left-invariant vector fields.
Recall the Koszul formula
$$
2\langle \nabla_XY,W\rangle=\langle [X,Y],W\rangle +\langle [W,X],Y\rangle+\langle X,[W,Y]\rangle,\quad X,Y,W \in \mathfrak{n}.
$$
If $U=Z_1+X_1$ and $V=Z_2+X_2$, with $Z_1,Z_2 \in  \mathfrak{z}$, $X_1,X_2 \in \mathfrak{v}$, then
$$
2\nabla_{U} V=[X_1,X_2]-J_{Z_2} (X_1)-J_{Z_1} (X_2),
$$
see details, e.g.~\cite[page 813]{Eber94.2}.

Now, if $Z_1,Z_2 \in  \mathfrak{z}_1$, $X_1,X_2 \in \mathfrak{v}_1$,
then $[X_1,X_2]\in  \mathfrak{n}_1$ by 1) and $J_{Z_2} (X_1), J_{Z_1} (X_2) \in \mathfrak{v}_1$ by the condition of 2).
Therefore, $\nabla_{U} V \in \mathfrak{n}_1$ for all $U,V \in \mathfrak{n}_1$. This implies 2).
\end{proof}
\smallskip

Theorems~\ref{th_ps_totgeod} and~\ref{pr_totgeod_h1} will be useful for us in the following reformulation.

\begin{corollary}\label{co.totgeod.1}
Let $N$ be a 2-step pseudo $H$-type Lie group with the Lie algebra $\mathfrak n$. Let $\mathfrak n_1\subset\mathfrak n$
be a subalgebra which generates a totally geodesic submanifold $N_1$ of $N$. If $N_1$ is not a geodesic orbit manifold, then $N$ is also not a geodesic orbit manifold.
\end{corollary}

\section{Pseudo-Riemannian $H$-type nilmanifolds $N_{r,s}$ with $r+s> 3$, $(r,s)\neq (3,4)$}\label{sec:mod4}

\subsection{Pseudo-Riemannian $H$-type nilmanifolds $N_{r,s}$ with $r+s=0\mod 4$}

We generalize a result of~\cite{Kap83} to the pseudo $H$-type Lie algebras $\mathfrak n_{r,s}$.

\begin{lemma}\label{lem:even}
Let $D\in\mathfrak h$ be a skew-symmetric derivation of a pseudo $H$-type Lie algebra $\mathfrak n_{r,s}$, $r+s=0\mod 4$, see~\eqref{isotalg1}. If we write $D=(C,A)$, then $A$ satisfies
    $$
    A\prod_{i=1}^{r+s}J_{Z_i}=\prod_{i=1}^{r+s}J_{Z_i}A \quad\text{for}\ \ r+s=0\mod 4
    $$
    and any orthonormal basis $\{Z_1,\ldots,Z_{r+s}\}$ as in~\eqref{eq:brs}.
\end{lemma}
\begin{proof}
We make some preliminary calculations. Let $\{Z_1,Z_2\}$ be orthonormal vectors in $\mathfrak z$.

If $\|Z_1\|^2=\|Z_2\|^2=\pm 1$ then any restriction to $\spann\{Z_1,Z_2\}$ of a skew symmetric map $C$ acting on $\mathfrak z$ is given by $ a\begin{pmatrix}0&1\\-1&0\end{pmatrix}$, $a\in\mathbb R$.
    It implies
    \begin{equation*}
J_{C(Z_1)}J_{Z_2}+J_{Z_1}J_{C(Z_2)}
=
a\Big(J_{Z_1}^2- J_{Z_2}^2\Big)=
a\Big(\Id-\Id\Big)=0.
    \end{equation*}
If $\|Z_1\|^2=-\|Z_2\|^2=\pm 1$ then a restriction to $\spann\{Z_1,Z_2\}$ of a skew symmetric map $C$ acting on $\mathfrak z$ is given by  $a\begin{pmatrix}0&1\\1&0\end{pmatrix}$, $a\in\mathbb R$, and analogous calculations show
$J_{C(Z_1)}J_{Z_2}+J_{Z_1}J_{C(Z_2)}
=
a\Big(J_{Z_1}^2+J_{Z_2}^2\Big)=
a\Big(-\Id+\Id\Big)=0$.

We proceed by induction of the dimension of the center.
    Let us write a matrix $C=\{c_{ij}\}\subset \mathfrak{so}(r,s)$ with $r+s=4$. Then we will obtain
    \begin{eqnarray*}
        A\prod_{i=1}^4J_{Z_i}&=&\prod_{i=1}^4J_{Z_i}A
        \\
        &+&\Big(c_{12}J_{Z_2}^2+c_{21}J_{Z_1}^2\Big)J_{Z_3}J_{Z_4}
        -\Big(c_{13}J_{Z_3}^2+c_{31}J_{Z_1}^2\Big)J_{z_2}J_{Z_4}
        +\Big(c_{14}J_{Z_4}^2+c_{41}J_{Z_1}^2\Big)J_{z_2}J_{Z_3}
        \\
        &+&\Big(c_{23}J_{Z_3}^2+c_{32}J_{Z_2}^2\Big)J_{z_1}J_{Z_4}
        -\Big(c_{24}J_{Z_4}^2+c_{42}J_{Z_2}^2\Big)J_{z_1}J_{Z_3}
        +\Big(c_{34}J_{Z_4}^2+c_{43}J_{Z_3}^2\Big)J_{z_1}J_{Z_2}
        \\
        &=&\prod_{i=1}^4J_{Z_i}A,
    \end{eqnarray*}
    since $c_{ij}J_{Z_j}^2+c_{ji}J_{Z_i}^2=0$ due to the skew symmetry in $\mathfrak{so}(r,s)$ and calculations at the beginning of the proof.
    By the induction of this arguments  we obtain
    $$
    \prod_{i=1}^{r+s}J_{Z_1}\ldots J_{C(Z_i)}\ldots J_{Z_{2k}}=
    \sum_{i<j}(-1)^{j-i+1}\big(c_{ij}J_{Z_j}^2+c_{ji}J_{Z_i}^2\big)J_{Z_1}\ldots \hat J_{Z_i}\ldots \hat J_{Z_j}\ldots J_{Z_{2k}}=0,
    $$
    where $\hat J_{Z_i}$ denotes the omitted term in the product and $C=\{c_{ij}\}\in \mathfrak{so}(r,s)$, $r+s=0\mod 4$.
\end{proof}

\begin{prop}\cite[Proposition 3.3]{LowMich89}\label{prop:3page22}
    The volume element $\omega=\prod_{i=1}^{r+s}Z_i$ in $\Cl(\mathbb R^{r,s})$ has the following basic
    properties. Let $n=r+s$. Then
    \begin{equation}\label{eq:3.7}
        \omega^{2}=(-1)^{\frac{n(n+1)}{2}+s}.
    \end{equation}
    In particular, if $n=r+s$ is odd, then $Z\omega=\omega Z$ for all $Z\in \mathbb R^{r,s}$, and if $n=r+s$ is even, then $Z\omega=-\omega Z $ for all $Z\in \mathbb R^{r,s}$. Formula~\eqref{eq:3.7} can be also written as
\begin{equation}\label{eq:3.7-1}
\omega^2=
\begin{cases}
(-1)^s\quad&\text{if}\quad r+s=0\ \text{or}\ 3\mod 4,
\\
(-1)^{s+1}\quad&\text{if}\quad r+s=1\ \text{or}\ 2\mod 4.
\end{cases}
\end{equation}
\end{prop}

\begin{theorem}\label{th:42}
    If $r+s=0\mod 4$, $s=0\mod 2$, then the pseudo-Riemannian $H$-type nilmanifold $N_{r,s}$ is not geodesic orbit.
\end{theorem}
\begin{proof}
Let us suppose that some pseudo-Riemannian $H$-type nilmanifold $N_{r,s}$ with $r+s=0\mod 4$ and $s=0\mod 2$ is geodesic orbit.

Let $Z_1,\ldots, Z_{r+s}$ be an orthonormal basis of the centre as in~\eqref{eq:brs} of the pseudo $H$-type Lie algebra $\mathfrak n_{r,s}=\mathfrak z\oplus\mathfrak v$.
    Consider the volume form $\omega=\prod_{i=1}^{r+s}Z_i$ in $\Cl(\mathbb R^{r,s})$. Then $\omega^2=1$ by Proposition~\ref{prop:3page22} since
$\frac{n(n+1)}{2}+s$ is even. Then operator
$J_{\omega}\colon \mathfrak v\to \mathfrak v$ decomposes the module $\mathfrak v$ into the eigenspaces of $\omega$:
    $$
    \mathfrak v=\mathfrak v^+\oplus\mathfrak v^-.
    $$
Moreover, since $A\omega=\omega A$ by Lemma~\ref{lem:even}, any $A$ in $D=(C,A)$ will leave the spaces $\mathfrak v^{\pm}$ invariant.

Take any $X\in \mathfrak v^+$ and $Z\in\mathfrak z$ as initial vector of a geodesic in the group $N_{r,s}$ (we can assume that $Z=Z_1$). Note that since $n=r+s$ is even we have
$
J_{Z_1}\omega=-\omega J_{Z_1}$
by Proposition~\ref{prop:3page22}. It implies that
\begin{equation}\label{eq:J1}
J_{Z_1}\colon \mathfrak v^+\to \mathfrak v^-.
\end{equation}
We choose $C\in \mathfrak {so}(r,s)$ being a matrix with the first row (and therefore column) vanish and the rest is skew symmetric in $\mathbb R^{r,s}$.
It gives $C(Z_1)=0$. We set $A=J_{Z_1}\in \mathfrak{so}(\mathfrak v)$. Then the map $A$ in
the skew symmetric derivation $D=(C,A)$ does not leave the spaces $\mathfrak v^{\pm}$ invariant by~\eqref{eq:J1}. We got a contradiction.
\end{proof}

\begin{corollary}\label{cor:mu-nu}
The pseudo $H$-type Lie groups $N_{\mu,\nu}$, $(\mu,\nu)\in\{(8,0),(0,8),(4,4)\}$ are not geodesic orbit.
\end{corollary}

\begin{remark}
In the proof of Theorem~\ref{th:42} we implicitly used that for any non-null vector $Z\in\mathfrak z$ the map
$(0,J_{Z})=(0,A)$ is a skew-symmetric derivation acting on the space $\mathfrak X=\spann_{\mathbb R}\{Z,X,J_{Z}X\}$ for any $X\in\mathfrak v$. Note also that for $\|X\|^2=\|Z\|^2=1$ the space $\mathfrak X$ is isometric to the Heisenberg algebra $\mathfrak n_{1,0}$, and for $\|X\|^2=\pm 1$, $\|Z\|^2=-1$ the space $\mathfrak X$ is isometric to the pseudo $H$-type algebra $\mathfrak n_{0,1}$.
\end{remark}

Before we proceed to show that $N_{r,s}$ is not geodesic orbit for $r+s=0\mod 4$, $s=1\mod 2$ we formulate a generalization of~\cite[Theorem 6]{Rie84} for the pseudo $H$-type Lie algebras.
Let us write $\mathfrak h=\mathfrak h_{r,s}$ in~\eqref{isotalg1} as
$$
\mathfrak h_{r,s}=(\mathfrak h_{r,s})_1\oplus (\mathfrak h_{r,s})_0,
$$
where
$$
(\mathfrak h_{r,s})_0=\{ D_0=(0,A),\ [A,J_{Z_j}]=0,\ \ j=1,\ldots,r+s \}
$$
is the Lie algebra of automorphisms of $\mathfrak n_{r,s}$ acting as identity on the center, and
$$
(\mathfrak h_{r,s})_1=\{D_1=(C,A),\ AJ_{Z}-J_{Z}A=J_{C^{\tau}(Z)} \}.
$$
According to formula~\eqref{eq:Phi} for an orthonormal basis $\{ Z_1,\ldots,Z_{r+s} \}$ for $\mathfrak z$ we have that a corresponding skew-symmetric derivation $D_1=(C,A)\in (\mathfrak h_{r,s})_1$ can be written as
$$
A(X)=J_{Z_i}J_{Z_j}(X),\quad C(Z)=2(\langle Z_i, Z \rangle_{r,s}Z_j-\langle Z_j, Z \rangle_{r,s}Z_i)=\ad_{Z_iZ_j}Z=[Z_iZ_j,Z].
$$
Let $Z_0\in\mathfrak z$
%and $X_0\in\mathfrak v$
be a non-null vector and
$\mathfrak z_0=\spann\{Z_0\}^{\perp}$ be the orthogonal complement in $\mathfrak z$. We also write $\Cl(\mathbb R^{r',s'})=\Cl(\mathfrak z_0)$, $r'+s'=r+s-1$ for the Clifford algebra generated by the space $(\mathfrak z_0,\langle\cdot\,,\cdot\rangle_{r,s})\vert_{\mathfrak z_0}$ and acting on the module $\mathfrak v$.

\begin{lemma}\label{lem:01}
Any extension of the skew-symmetric  derivation $(0,J_{Z_0})$, $\|Z_0\|\neq0$
to a skew-symmetric derivation $D=(C,A)\in \mathfrak h_{r,s}$  must belong to
    $$
    \mathfrak h_{r,s}=(\mathfrak h_{r',s'})_1+(\mathfrak h_{r,s})_0,\quad r'+s'=r+s-1.
    $$
\end{lemma}
\begin{proof}
    Let us assume that $D=(C,A)\in\mathfrak h_{r,s}$ is an extension of $(0,J_{Z_0})$. Then $C(Z_0)=0$, where we used that the vector $(0,Z_0)$ is not null. Any $D\in(\mathfrak h_{r,s})_0$ will satisfy it. Let
    $Z_0,Z_1,\ldots, Z_{r+s-1}$ be an orthonormal basis of $\mathfrak z_{r,s}$. We write $\mathfrak z_{r',s'}=\spann\{Z_1,\ldots, Z_{r+s-1}\}$ then the set
    $$
    \{Z_kZ_l,\ \  0< k<l\leq r+s-1\}
    $$
    form a basis of $\mathfrak{so}(\mathfrak z_{r',s'})$ and $D_1=D_1(Z_kZ_l)=(C,A)$ generated by $Z_kZ_l$ satisfies
    $$
    A(X)=J_{Z_k}J_{Z_l}(X),\quad
    C(Z_0)=
    \begin{cases}
        0\quad&\text{if}\quad 0<k<l\leq r+s-1
        \\
        2\|Z_0\|^2 Z_l &\text{if}\quad 0=k<l\leq r+s-1.
    \end{cases}
    $$
    Thus $D_1(Z_kZ_l)=(C,A)\in (\mathcal D_{r',s'})_1$ and $C(Z_0)=0$.
\end{proof}

\begin{theorem}
    If $r+s=0\mod 4$, $s=1\mod 2$, then the pseudo $H$-type Lie group $N_{r,s}$ is not geodesic orbit manifold.
\end{theorem}
\begin{proof}
Let us write $s=s'+1$ and consider the  last vector $Z_{r+s}$, $\|Z_{r+s}\|^2=-1$ in the orthonormal basis $\{Z_1,\ldots, Z_{r+s}\}$ for $\mathfrak z_{r,s}$.
Note that $\omega^2=1$, where $\omega=\prod_{i=1}^{r+s}Z_i$ by~\eqref{eq:3.7-1}. The module $\mathfrak v_{r,s}$ of $\Cl(\mathbb R^{r,s})$ is decomposed into the direct sum of two subspaces
$$
\mathfrak v_{r,s}=\mathfrak v^+_{r,s'}\oplus \mathfrak v^-_{r,s'},
$$
that are the eigenspaces of the volume form $\omega$. The spaces $\mathfrak v^{\pm}_{r,s'}$ are the non-equivalent modules of
$\Cl(\mathbb R^{r,s'})=\Cl(\mathfrak z_{r,s}^{\perp})$, where $\mathfrak z_{r,s}^{\perp}$ is the orthogonal complement of $\spann\{Z_{r+s}\}$ in $\mathfrak z_{r,s}$.
%The spaces $\mathfrak v^{\pm}_{r,s'}$ are not submodules of $\Cl(\mathbb R^{r,s})$.
Due to $J_{Z_{r+s}}\omega=-\omega J_{Z_{r+s}}$, the map $J_{Z_{r+s}}\colon
\mathfrak v^{+}_{r,s'}\to \mathfrak v^{-}_{r,s'}$ is an isomorphism  of vectors spaces.

If $D=(0,A)\in(\mathfrak h_{r,s})_0$, then
$$
[A,J_{Z_i}]=0\quad\Longrightarrow\quad A\omega=\omega A
$$
by Lemma~\ref{lem:even}. Therefore $A$ must preserve the spaces $\mathfrak v^{\pm}_{r,s'}$. If  $D=(C,A)\in(\mathfrak h_{r,s'})_1$, then $A$ also leaves spaces $\mathfrak v^{\pm}_{r,s'}$ invariant since $\mathfrak v^{\pm}_{r,s'}$ are submodules of the Clifford algebra $\Cl(\mathbb R^{r,s'})=\Cl(\mathfrak z_{r,s}^{\perp})$.

The extension $D=(0,A)$ of $(0,J_{Z_{r+s}})$ acting on the $\mathfrak X=\spann\{Z_{r+s},X,J_{Z_{r+s}}(X)\}$, $X\in \mathfrak v^+_{r,s'}$ must belong to $(\mathfrak h_{r,s'})_1+(\mathfrak h_{r,s})_0$ by Lemma~\ref{lem:01}. By the above arguments $A$ should preserve the submodule $\mathfrak v^+_{r,s'}$. But this contradicts to the fact that the restriction $A\vert_{\mathfrak X}=J_{Z_{r+s}}$ of $A$ to $\mathfrak X$ has the property $J_{Z_{r+s}}(X)\in\mathfrak v^-_{r,s'}$.
\end{proof}

\subsection{The geodesic orbit property and the periodicity property}

As a further application of Theorem~\ref{pr_totgeod_h1}, Corollary~\ref{co.totgeod.1}, and the structure of admissible modules under the periodicity property,
we prove the following theorem.

\begin{theorem}\label{th:NGO1}
If a pseudo $H$-type Lie group $N_{r,s}$ is not geodesic orbit, then $N_{r+\mu,s+\nu}$ with $(\mu,\nu)\in\{(8,0),(0,8),(4,4)\}$ is also not geodesic orbit.
\end{theorem}

\begin{proof}
Consider a pseudo $H$-type Lie algebra $\mathfrak n_{r+\mu,s+\nu}$ with the minimal admissible module $\mathfrak v_{r+\mu,s+\nu}=\mathfrak v_{r,s}\otimes \mathfrak v_{\mu,\nu}$, the scalar product
$$
\langle\cdot\,,\cdot\rangle_{\mathfrak v_{r+\mu,s+\nu}}=\langle\cdot\,,\cdot\rangle_{\mathfrak v_{r,s}}\cdot\langle\cdot\,,\cdot\rangle_{\mathfrak v_{\mu,\nu}},
$$
and the basis $\{u_i\otimes v_k,\ i=1,\ldots,\dim(\mathfrak v_{r,s}),\ k=1,\ldots,16\}$. Here $\{u_i\}_{i=1}^{\dim(\mathfrak v_{r,s})}$
is an orthonormal invariant basis for $\mathfrak v_{r,s}$ generated by a vector $u\in E^{+1}_{r,s}$ with $\langle u, u\rangle_{\mathfrak v_{r,s}}=1$,
see Proposition~\ref{prop:basis}. Analogously, $\{v_k\}_{k=1}^{16}$ is an orthonormal invariant basis for $\mathfrak v_{\mu,\nu}$ from~\eqref{eq:basis-munu}
generated by a unit vector $v\in E^{+1}_{\mu,\nu}$. The orthonormal basis
\begin{equation}\label{eq:brsmn}
Z_1,\ldots,Z_{r+s},\zeta_1,\ldots,\zeta_{\mu+\nu}
\end{equation}
%{\color{magenta} Is this basis a union of two arbitrary bases for the centers of $\mathfrak n_{r,s}$ and $\mathfrak n_{\mu,\nu}$?}
for $\mathbb R^{r+\mu,s+\nu}$ is the union of the basis~\eqref{eq:brs} for $\mathbb R^{r,s}$ and the basis~\eqref{eq:zeta-basis} for $\mathbb R^{\mu,\nu}$. The basis~\eqref{eq:brsmn}
acts on
$\mathfrak v_{r+\mu,s+\nu}=\mathfrak v_{r,s}\otimes \mathfrak v_{\mu,\nu}$ by
$$
\tilde J_{Z_{l}}=J_{Z_{l}}\otimes\Id, \ l=1,\ldots,\dim{\mathfrak v_{r,s}},
\qquad \tilde J_{\zeta_{m}}=\Id\otimes J_{\zeta_{m}},\ m=1,\ldots,16.
$$
The commutators in $\mathfrak n_{r+\mu,s+\nu}$ satisfy the following relations
\begin{equation}\label{eq:111}
    [u_i\otimes v_k,u_j\otimes v_k]_{\mathfrak v_{r+\mu,s+\nu}}=[u_i,u_j]_{\mathfrak v_{r,s}}\|v_k\|^2,\quad i,j=1,\ldots,\dim{\mathfrak v_{r,s}},
\end{equation}
for any $k=1,\ldots,16$. Indeed, let $Z\in\mathbb R^{r,s}$, then
\begin{eqnarray*}
    \langle Z,[u_i\otimes v_k,u_j\otimes v_k]_{\mathfrak v_{r+\mu,s+\nu}}\rangle_{r+\mu,s+\nu}
    &=&
    \langle \tilde J_Z(u_i\otimes v_k),u_j\otimes v_k\rangle_{\mathfrak v_{r+\mu,s+\nu}}
    \\
    &=&\langle J_Z(u_i)\otimes v_k,u_j\otimes v_k\rangle_{\mathfrak v_{r+\mu,s+\nu}}
    \\
    &=&
    \langle J_Z(u_i),u_j\rangle_{\mathfrak v_{r,s}}\cdot\langle v_k,v_k\rangle_{\mathfrak v_{\mu,\nu}}
    \\
    &=&
    \langle Z, [u_i,u_j]_{\mathfrak v_{r,s}}\rangle_{r,s}\|v_k\|^2.
\end{eqnarray*}
since
$
\langle Z,\bullet\rangle_{r+\mu,s+\nu}=\langle Z,\bullet\rangle_{r,s}
$
for any $Z\in \mathbb R^{r,s}$. Analogously
\begin{equation}\label{eq:112}
    [u_k\otimes v_i,u_k\otimes v_j]_{\mathfrak v_{r+\mu,s+\nu}}=[v_i,v_j]_{\mathfrak v_{\mu,\nu}}\|u_k\|^2\quad i,j=1,\ldots,16,
\end{equation}
and any $k=1,\ldots,\dim{\mathfrak v_{r,s}}$.

We denote
$$\mathbb R^{r,s}=\mathfrak z_1=\spann\{Z_1,\ldots, Z_{r+s}\},\qquad\mathbb R^{\mu,\nu}=\mathfrak z_2=\spann\{\zeta_1,\ldots, \zeta_{\mu+\nu}\}.
$$
Equality~\eqref{eq:111} and Theorem~\ref{pr_totgeod_h1} show that the pseudo $H$-type Lie algebra $\mathfrak n_{r,s}=\mathbb R^{r,s}\oplus \mathfrak v_{r,s}$ is a subalgebra of $\mathfrak n_{r+\mu,s+\nu}=\mathbb R^{r+\mu,s+\nu}\oplus \mathfrak v_{r+\mu,s+\nu}$. We have 16 such subalgebras of the form
$$
\mathbb R^{r,s}\oplus\mathfrak v_{r,s}\otimes\spann\{v_k\},\quad v_k\in \mathfrak v_{\mu,\nu},\quad k=1,\ldots, 16.
$$

Analogously,~\eqref{eq:112} implies that there are $m=\dim(\mathfrak v_{r,s})$ subalgebras isomorphic to $\mathfrak n_{\mu,\nu}$ inside the Lie algebra $\mathfrak n_{r+\mu,s+\nu}=\mathbb R^{r+\mu,s+\nu}\oplus\mathfrak v_{r+\mu,s+\nu}$, and they have the form
$$
\mathbb R^{\mu,\nu}\oplus\spann\{u_k\}\otimes \mathfrak v_{\mu,\nu},\quad u_k\in \mathfrak v_{r,s},\quad k=1,\ldots, \dim(\mathfrak v_{r,s}).
$$

Now, we apply Corollary~\ref{co.totgeod.1}. We set (compare with \eqref{eq:decom08e} and \eqref{eq:decom08})
$$
\mathfrak v_1=\mathfrak v_{r,s}\otimes E^{+1}_{\mu,\nu},\quad \mathfrak v_2=\bigoplus_{j=1}^{8}\Big(\mathfrak v_{r,s}\otimes J_{\zeta_{j}}E^{+1}_{\mu,\nu}\Big)\bigoplus_{j=2}^{8}\Big(\mathfrak v_{r,s}\otimes J_{\zeta_1}J_{\zeta_j}E^{+1}_{\mu,\nu}\Big).
$$
Under this notation we obtain from~\eqref{eq:111}
$$
\mathfrak v_{r+\mu,s+\nu}=\mathfrak v_1\oplus \mathfrak v_2,\quad \mathfrak n_1=\mathfrak z_1\oplus\mathfrak v_1\cong\mathfrak n_{r,s}
$$
with the vector $v_1=v$, $\|v\|^2=1$ generating the basis for $E^{+1}_{\mu,\nu}$. Moreover
$$
J_{Z_i}(\mathfrak v_1)\subset\mathfrak v_1\ \ \text{for all}\ \ Z_i\in\mathfrak z_1,\qquad
J_{\zeta_j}(\mathfrak v_1)\subset\mathfrak v_2\ \ \text{for all}\ \ \zeta_j\in\mathfrak z_2.
$$
This implies that the pseudo $H$-type Lie algebra $\mathfrak n_1\cong\mathfrak n_{r,s}$ generates a totally geodesic subgroup $N_{r,s}$ in $N_{r+\mu,s+\nu}$. By the hypothesis of Theorem~\ref{th:NGO1} the nilmanifold $N_{r,s}$ is not geodesic orbit. Then Corollary~\ref{co.totgeod.1} implies that $N_{r+\mu,s+\nu}$ is also not geodesic orbit.
\end{proof}
\smallskip

\begin{theorem}
Pseudo-Riemannian $H$-type nilmanifolds $N_{r,s}$ with
$\max\{r,s\}\geq 8$ and $\min\{r,s\}\geq 4$ are not geodesic orbit.
\end{theorem}
\begin{proof}
Let $N_{r',s'}$ be a pseudo $H$-type Lie group with $\max\{r,s\}\geq 8$ and $\min\{r,s\}\geq 4$. Then write $r'=r+\mu$, $s'=s+\nu$ for some $(\mu,\nu)\in\{(8,0),(0,8),(4,4)\}$. Arguing as in the proof of Theorem~\ref{th:NGO1}, we show that for
$$\mathbb R^{r,s}=\mathfrak z_1=\spann\{Z_1,\ldots, Z_{r+s}\},\qquad\mathbb R^{\mu,\nu}=\mathfrak z_2=\spann\{\zeta_1,\ldots, \zeta_{\mu+\nu}\},
$$
and
$$
\mathfrak v_1=\spann\{u\}\otimes \mathfrak v_{\mu,\nu},\quad
\mathfrak v_2=\mathfrak v_{r,s}\otimes E^{+1}_{\mu,\nu},
$$
where $u\in E^{+1}_{r,s}$ with $\langle u, u\rangle_{\mathfrak v_{r,s}}=1$, the Lie algebra
$$
\mathfrak n_2=\mathfrak z_2\oplus\mathfrak v_2\cong \mathfrak n_{\mu,\nu}
$$
is a subalgebra of $\mathfrak n_{r+\mu,s+\nu}$. Moreover
$$
J_{Z_i}(\mathfrak v_2)\subset\mathfrak v_2\ \ \text{for all}\ \ Z_i\in\mathfrak z_1,\qquad
J_{\zeta_j}(\mathfrak v_2)\subset\mathfrak v_1\ \ \text{for all}\ \ \zeta_j\in\mathfrak z_2.
$$
It implies that the pseudo $H$-type Lie algebra $\mathfrak n_2\cong\mathfrak n_{\mu,\nu}$ generates a totally geodesic subgroup $N_{\mu,\nu}$ in $N_{r+\mu,s+\nu}$.
By Corollary~\ref{cor:mu-nu}
the group $N_{\mu,\nu}$ is not geodesic orbit and Corollary~\ref{co.totgeod.1} implies that $N_{r+\mu,s+\nu}$ is also not geodesic orbit.
\end{proof}

\subsection{Pseudo-Riemannian $H$-type nilmanifolds $N_{r,s}$, for  $(r,1)$, $r\geq 3$ and $(0,s)$, $s>0$.}

\begin{theorem}\label{th:s>3}
The pseudo-Riemannian $H$-type nilmanifolds $N_{0,s}$ with $s\geq 4$ are not geodesic orbit.
\end{theorem}

\begin{proof}
Let $\{Z_{i}\}_{i=1}^{s}$ be an orthonormal basis of the centre $\mathbb{R}^{0,s}$.
Consider a positive involution $p=Z_{1}Z_{2}Z_{3}Z_{4}$ and a vector $X\in\mathfrak v$, $\langle X,X\rangle_{\mathfrak v}=-1$ such that $J_{p}(X)=X$.

Assume that a geodesic defined by an initial vector $(Z_{1},X)\in \mathfrak n_{0,s}$ is homogeneous.
Then there is $D=(C,A)\in\mathfrak h$ such that
$$
A(X)=J_{Z_{1}}(X),\quad  C(Z_{1})=0,\quad [A,J_{w}]=J_{C(w)}, \quad\text{for any}\quad w\in \mathbb{R}^{0,s},
$$
where we used $k=0$, since the vector $(Z_{1},X)$ is not a null vector. Then we obtain
\begin{eqnarray}\label{eq:YYY}
J_{Z_{1}}(X)&=&A(X)=A(J_{p}(X))
\\
&=&
J_{p}A(X)-J_{Z_1}J_{C(Z_2)}J_{Z_3}J_{Z_4}(X)-J_{Z_1}J_{Z_2}J_{C(Z_3)}J_{Z_4}(X)-J_{Z_1}J_{Z_2}J_{Z_3}J_{C(Z_4)}(X)\nonumber
\\
&=&
-J_{Z_{1}}(X)-J_{Z_1}J_{Z_2}J_{Z_3}J_{C(Z_4)}(X),\nonumber
\end{eqnarray}
where in the last step we used
\begin{equation}\label{eq:Th10}
J_{p}A(X)=J_{p}J_{Z_{1}}(X)=-J_{Z_{1}}J_{p}(X)=-J_{Z_{1}}(X).
\end{equation}
and 
\begin{equation*}\label{eq:any case}
J_{Z_1}\big(J_{C(Z_2)}J_{Z_3}+J_{Z_2}J_{C(Z_3)}\big)J_{Z_4}(X)=0
\end{equation*}
as was shown in the proof of Lemma~\ref{lem:even}.
 Hence $2J_{Z_{1}}(X)=-J_{Z_1}J_{Z_2}J_{Z_3}J_{C(Z_4)}(X)$.
From one side
$$
\langle 2J_{Z_{1}}(X),2J_{Z_{1}}(X)\rangle=4\langle Z_1,Z_1\rangle_{0,s}\langle X,X\rangle_{\mathfrak v}=4.
$$
From the other side
$$
\langle J_{Z_1}J_{Z_2}J_{Z_3}J_{C(Z_4)}(X),J_{Z_1}J_{Z_2}J_{Z_3}J_{C(Z_4)}(X)\rangle=(-1)^4\langle C(Z_4),C(Z_4)\rangle_{0,s}<0,
$$
which is a contradiction.
\end{proof}

\begin{theorem}\label{th:s=1}
The pseudo-Riemannian $H$-type nilmanifolds $N_{r,1}$, $r\geq 2$, are not geodesic orbit.
\end{theorem}

\begin{proof}
Fix a basis $\{Z_i\}_{i}^{r+1}$ as in~\eqref{eq:brs}. Let $\mathcal S$ be a maximal subgroup of  $G(B_{r,1})$ of positive involutions and let $PI_{r,1}$ be its generating set. Note that $PI_{r,1}=PI_{r,0}$ by the structure of involutions, see~\eqref{eqT1T2}, which says that none of positive involutions in $PI_{r,1}$  contains the basis vector $Z_{r+1}$. We assume that the set $PI_{r,1}$ contains an involution $q=Z_{i_1}\ldots J_{Z_{i_a}}$ of Type 2 as it was defined in~\eqref{eqT1T2}. This assumption always can be achieved by removing one of the basis vectors or multiplying some of the involutions.

Let us choose a geodesic $\gamma(t)=(z(t),x(t))$ with an initial vector $(Z_{r+1},X)$ such that $\langle X,X\rangle_{\mathfrak v}=-1$, and $J_q(X)=X$ and assume that the geodesic is homogeneous.
Then there is $D=(C,A)\in\mathfrak h$ such that
$$
A(X)=J_{Z_{r+1}}(X),\quad C(Z_{r+1})=0,\quad
[A,J_{W}]=J_{C^{\tau}(W)},\quad\text{for all}\quad W\in \mathbb{R}^{r,0},
$$
where we set $k=0$, since $(Z_{r+1},X)$ is not a null vector.
The condition $C(Z_{r+1})=0$ implies that $C$ has vanishing last column. Now as in~\eqref{eq:YYY} of Theorem~\ref{th:s>3} we have an equality
    \begin{equation}\label{eq:25-1}
J_{Z_{r+1}}(X) =A(X)= A(J_q(X))=J_qA(X)-J_{\dot{q}}(X)
    =-J_{Z_{r+1}}(X)-J_{\dot{q}}(X).
    \end{equation}
In the last equality we used $J_{q}J_{Z_{r+1}}(X)=-J_{Z_{r+1}}J_{q}(X)$ since Type 2 involution $q$ contains odd number of basis vectors and does not contain $Z_{r+1}$. Here
 \[
\dot{q}=C(Z_{i_1})\cdots Z_{i_a}+Z_{i_1}C(Z_{i_2})\cdots Z_{i_a}+
\ldots +Z_{i_1}Z_{i_2}\cdots C(Z_{i_a})=Z_{i_1}Z_{i_2}\cdots C(Z_{i_a})
\]
by skew symmetry of $C$ and
$$
Z_{i_1}\cdots \big(C(Z_{i_j})Z_{i_{j+1}}+Z_{i_j}C(Z_{i_{j+1}})\big)\cdots Z_{i_a}=0,
$$
see the proof of Lemma~\ref{lem:even}.  From~\eqref{eq:25-1}, we obtain
$
2J_{Z_{r+1}}(X)=-J_{\dot{q}}(X)$.
Note that $\dot{q}$ does not include $Z_{r+1}$  because $C(Z_{i})=\sum_{j=1}^{r}c_{ij}Z_j$ with $c_{ij}=0$ for $j=r+1$. Therefore
$$
\langle J_{\dot{q}}(X),J_{\dot{q}}(X)\rangle_{\mathfrak v}=\langle C(Z_{i_a}),C(Z_{i_a})\rangle_{r,1}\langle X,X\rangle_{\mathfrak v}< 0
$$ by~\eqref{eq:JJJ}. From the other side $\langle 2J_{Z_{r+1}}(X),2J_{Z_{r+1}}(X)\rangle_{\mathfrak v}=4$, which is a contradiction.
\end{proof}

We summarize the results of Sections~\ref{sec.example}, \ref{se.tot.geod}, and~\ref{sec:mod4} and show in Table~\ref{tab:1}
the cases of pseudo $H$-type Lie groups $N_{r,s}$ that need to be studied.

    \begin{table}[h] %%%%%%%%%%%%%%%%%%%%%%%%%%%%%%%%%%%
    \caption{{\small Pseudo-Riemannian $H$-type nilmanifolds, which need to be studied}}
%\scalebox{0.7}[0.7]{%
        \begin{tabular}{|c||c|c|c|c|c|c|c|c|c|c|c|c|c|c|c|c|c|c|c|c||}\hline
7 &  && $N_{2,7}$ & $N_{3,7}$ &&&&\\\hline
6  &  & $N_{1,6}$ & $$ & $N_{3,6}$ &&&& \\\hline
5  &  & $N_{1,5}$ & $N_{2,5}$ & $$ &&&& \\\hline
4  & & $N_{1,4}$ & $N_{2,4}$ & $N_{3,4}$ &&&&  \\\hline
3 & &  & $N_{2,3}$ & $N_{3,3}$ & $N_{4,3}$ & $$  &$N_{6,3}$ & $N_{7,3}$  \\\hline
2 &  & &  & $N_{3,2}$ & $N_{4,2}$ & $N_{5,2}$  &$$ & $N_{7,2}$ \\\hline
1 &  &  &  &  & &   & &  \\\hline
0 & & &  &  &  & & &   \\\hline\hline
s/r  & $0$ & $1$ & $2$ & $3$ & $4$ & $5$  &$6$ & $7$ \\\hline
        \end{tabular}
%}
    \label{tab:1}
\end{table}

\subsection{Relation between
    $\mathfrak{n}_{r,s}$, $\mathfrak{n}_{r+1,s}$, and  $\mathfrak{n}_{r,s+1}$}
In this section we present some of the arguments based on the structure of the generating set for the maximal group of positive involutions, which allows us to find out which of the groups in Table~\ref{tab:1} are not geodesic orbit, see Section~\ref{sec:31}.

Let $\ell(r,s)$ be  the number of the mutually commuting positive
involutions in the generating set $PI_{r,s}$ of the maximal group of positive involutions $\mathcal S$.
Recall that the number $\ell(r,s)$ is periodic with the three periods $(8,0),(4,4)$ and $(0,8)$; that is,
\[
\ell(r+8,s)=\ell(r,s+8)=\ell(r+4,s+4)=\ell(r,s)+4.
\]

Table~\ref{tab:2} shows the values $\ell(r,s)$ which are interesting for us.
\begin{table}[th] %%%%%%%%%%%%%%%%%%%%%%%%%%%%%%%%%%%
    \caption{{ The value $\ell(r,s)$ for $r+s\leq 16$}}
%   \scalebox{1.3}[1.2]{%
\begin{tabular}{|c||c|c|c|c|c|c|c|c|c|c|c|c|c|c|c|c|c|c|c|c||}\hline
7 & $3$ & $3$ & $3$ & $4$ & $$ & $$  &$$ & $$\\\hline
6  & $2$ & $3$ & $3$ & $4$ & $$ & $$  &$$ & $$\\\hline
5  & $1$ & $2$ & $3$ & $4$ & $$ & $$  &$$ & $$\\\hline
4  & $1$ & $2$ & $3$ & $4$ & $$ & $$  &$$ & $$\\\hline
3 & $0$ & $1$ & $2$ & $3$ & $3$ & $3$  &$3$ & $4$ \\\hline
2 & $0$ & $1$ & $1$ & $2$ & $2$ & $3$  &$3$ & $4$\\\hline
1 & $0$ & $0$ & $0$ & $1$ & $1$ & $2$  &$3$ & $4$\\\hline
0 & $0$ & $0$ & $0$ & $1$ & $1$ & $2$  &$3$ & $4$\\\hline\hline
s/r  & $0$ & $1$ & $2$ & $3$ & $4$ & $5$  &$6$ & $7$\\\hline
        \end{tabular}
%   }
    \label{tab:2}
\end{table}
In general, there are two cases of the dimensions of minimal admissible modules $\mathfrak v_{r,s}$
according to the relations of the values $\ell(r,s)$, $\ell(r+1,s)$ and $\ell(r,s+1)$:
\[
\ell(r,s)\leq \ell(r+1,s)\leq \ell(r,s)+1~\text{and}~\ell(r,s)\leq \ell(r,s+1)\leq \ell(r,s)+1.
\]
Then
\begin{align*}
    & [1]\quad \ell(r,s)=\ell(r+1,s)~~~\text{implies}~~\dim \mathfrak v_{r+1,s}=2\dim \mathfrak v_{r,s},\\
    & [2]\quad \ell(r,s)+1=\ell(r+1,s)~~~\text{implies}~~\dim \mathfrak v_{r+1,s}=\dim \mathfrak v_{r,s},\\
    & [3]\quad \ell(r,s)=\ell(r,s+1)~~~\text{implies}~~\dim \mathfrak v_{r,s+1}=2\dim \mathfrak v_{r,s},\\
    & [4]\quad \ell(r,s)+1=\ell(r,s+1)~~~\text{implies}~~\dim \mathfrak v_{r,s+1}=\dim \mathfrak v_{r,s}.
\end{align*}

In cases $[1]$  (and $[3]$) a module $\mathfrak v_{r,s}$ is a submodule of $\mathfrak v_{r+1,s}$ ($\mathfrak v_{r,s+1}$), which is not true for cases $[2]$ and $[4]$. This can be shown as follows.

We fix a set $PI_{r,s}$, construct an orthonormal basis $\{X_i\}_{i=1}^{\dim \mathfrak v_{r,s}}$ for $\mathfrak v_{r,s}$ as in Proposition~\ref{prop:basis} generated by a vector $v$, $\langle v,v\rangle_{\mathfrak v_{r,s}}=1$ which is $1$-eigen vector for all the involutions in $PI_{r,s}$.
We restrict the proof to the case $\mathfrak n_{r,s+1}$, since for $\mathfrak n_{r+1,s}$ the arguments are similar.
The set $\mathfrak v_{r,s}=\spann\{{X_i,\ i=1,\ldots,\dim \mathfrak v_{r,s}}\}$ is a subset of $\mathfrak v_{r,s+1}$ and a module under the action of $\Cl(\mathbb R^{r,s})$. The map $J_{Z_{r+s+1}}$ is an orthogonal transformation on $\mathfrak v_{r,s+1}$ and therefore  $J_{Z_{r+s+1}}(\mathfrak v_{r,s})\subset \mathfrak v_{r,s+1}$ and moreover $J_{Z_{r+s+1}}(\mathfrak v_{r,s})$ is a minimal admissible module for $\Cl(\mathbb R^{r,s})$. Thus
$$
\mathfrak v_{r,s+1}=\mathfrak v_{r,s}\oplus J_{Z_{r+s+1}}(\mathfrak v_{r,s})
$$
is an orthogonal decomposition in two minimal admissible modules for $\Cl(\mathbb R^{r,s})$. The set $\{X_i,J_{Z_{r+s+1}}(X_i),\ i=1,\ldots,\dim(\mathfrak v_{r,s})\}$ is an orthonormal basis for $\mathfrak v_{r,s+1}$.

\begin{theorem}\label{th:1-3}
In the above notation in cases $[1]$ and $[3]$ the pseudo-Riemannian $H$-type nilmanifolds $N_{r,s}$ are totally geodesic submanifolds of $N_{r+1,s}$ and $N_{r,s+1}$, respectively.
\end{theorem}
\begin{proof} We write for $\mathfrak n_{r,s+1}$
$$
\mathbb R^{r,s}=\mathfrak z_1=\spann\{Z_1,\ldots,Z_{r+s}\},\quad\mathfrak z_2=\spann\{Z_{r+s+1}\},\quad
%$$
%and
%$$
\mathfrak v_1=\mathfrak v_{r,s}\quad \mathfrak v_2=J_{Z_{r+s+1}}(\mathfrak v_{r,s}).
$$
Then we obtain that $
J_{Z_k}(\mathfrak v_1)\subset \mathfrak v_1$ for any $Z_k\in\mathfrak z_1$
since $\mathfrak v_1$ is a submodule of $\Cl(\mathbb R^{r,s})$. We also have
$
J_{Z_{r+s+1}}(\mathfrak v_1)\subset \mathfrak v_2$.
Applying Theorem~\ref{pr_totgeod_h1} and Corollary~\ref{co.totgeod.1} we finish the proof.
\end{proof}

\begin{corollary}
The pseudo-Riemannian $H$-type nilmanifolds
$$
N_{2,7},\quad N_{3,6},\quad N_{3,7},\quad N_{6,3},\quad N_{7,2},\quad N_{7,3}
$$
are not geodesic orbit
\end{corollary}
\begin{proof}
Applying Theorem~\ref{th:1-3} to the groups in Table~\ref{tab:1}  and using Table~\ref{tab:2}   we finish the proof. For instance since $N_{7,1}$ is not geodesic orbit and
$$
\ell(7,1)=\ell(7,2)=\ell(7,3),
$$
we conclude that $N_{7,2}$ and $N_{7,3}$ are not geodesic orbit.
\end{proof}
\begin{remark} One can show the following.
In the cases $[2]$ and $[4]$, the module $\mathfrak v_{r+1,s}$ (or $\mathfrak v_{r,s+1}$)
is also a minimal admissible module
of the Clifford algebra $\Cl(\mathbb R^{r,s})$. In these  cases the natural inclusion map
    \[
    \mathfrak{n}_{r,s}(\mathfrak v_{r,s})=
    \mathfrak{n}_{r,s}(\mathfrak v_{r+1,s})\subset  \mathfrak{n}_{r+1,s}(\mathfrak v_{r+1,s})
    \]
    is not a Lie algebra homomorphism.
\end{remark}

\begin{theorem}
The pseudo $H$-type Lie groups $N_{r,s}$ for
$$
(r,s)\in\{(1,4),(1,5),(1,6),(2,3),(2,4),(2,5),(3,2),(3,3),(4,2),(4,3),(5,2)\}
$$
are not geodesic orbit.
\end{theorem}
\begin{proof}
{\it Pseudo-Riemannian $H$-type nilmanifolds $N_{1,4}$ and $N_{3,2}$}.

We choose an orthonormal basis for centres as in~\eqref{eq:brs} and the generating set
$$
PI=\{p_1=Z_1Z_2Z_3,\quad p_2=Z_2Z_3Z_4Z_5\}
$$
for the maximal group of positive involutions $\mathcal S$.
Let $v\in \{X\in \mathfrak v:\ J_{p_1}(X)=J_{p_2}(X)=X\}$, $\|v\|^2=1$.
We construct an orthonormal invariant basis for $\mathfrak v$:
$$
\begin{array}{lllll}
& X_1=v,\quad & X_2=J_{Z_2}(v),\quad & X_3=J_{Z_4}(v),\quad & X_4=J_{Z_2}J_{Z_4}(v),\quad
\\
&X_5=J_{Z_1}(v),\quad &X_6=J_{Z_3}(v),\quad &X_7=J_{Z_5}(v),\quad &X_8=J_{Z_3}J_{Z_4}(v).
\end{array}
$$
We denote
$$
\mathfrak z_1=\spann\{Z_2,Z_4\},\quad \mathfrak z_2=\spann\{Z_1,Z_3,Z_5\},\quad \mathfrak z=\mathfrak z_1\oplus \mathfrak z_2,
$$
$$
\mathfrak v_1=\spann\{X_1,\ldots X_4\},\quad \mathfrak v_2=\spann\{X_5,\ldots ,X_8\},\quad \mathfrak v=\mathfrak v_1\oplus \mathfrak v_2.
$$
It is obvious that $J_{Z_k}(\mathfrak v_1)\subset \mathfrak v_1$ for $k=2,4$. If $(r,s)=(1,4)$, then the pseudo $H$-type Lie algebra $\mathfrak z_1\oplus \mathfrak v_1$ is isomorphic to $\mathfrak n_{0,2}$. If $(r,s)=(3,2)$, then the pseudo $H$-type Lie algebra $\mathfrak z_1\oplus \mathfrak v_1$ is isomorphic to $\mathfrak n_{1,1}$.

To show that $J_{Z_k}(\mathfrak v_1)\subset \mathfrak v_2$ for $k=1,3,5$ we observe that $p_1p_2=-Z_1Z_4Z_5$ and $J_{Z_1Z_4Z_5}(v)=-v$. Then it is easy to see the following
$$
\begin{array}{lllll}
&J_{Z_1} (X_1)=\pm X_5,\quad &J_{Z_1} (X_2)=\pm X_6,\quad &J_{Z_1} (X_3)=\pm X_7,\quad &J_{Z_1} (X_4)=\pm X_8,\quad
\\
&J_{Z_3} (X_1)=\pm X_6,\quad &J_{Z_3} (X_2)=\pm X_5,\quad &J_{Z_3} (X_3)=\pm X_8,\quad &J_{Z_3} (X_4)=\pm X_7,\quad
\\
&J_{Z_5} (X_1)=\pm X_7,\quad &J_{Z_5} (X_2)=\pm X_8,\quad &J_{Z_5} (X_3)=\pm X_5,\quad &J_{Z_5} (X_4)=\pm X_6.
\end{array}
$$
Theorem~\ref{pr_totgeod_h1} and Corollary~\ref{co.totgeod.1} imply that pseudo $H$-type Lie groups $N_{1,4}$ and $N_{3,2}$ are not geodesic orbit.

{\it Pseudo-Riemannian $H$-type nilmanifold $N_{1,6}$}.

We choose an orthonormal basis for the center as in~\eqref{eq:brs} and the generating set
\begin{equation}\label{eq:16-52}
PI=\{p_1=Z_1Z_2Z_3,\quad p_2=Z_2Z_3Z_4Z_5, \quad p_3=Z_2Z_3Z_6Z_7\}
\end{equation}
for the maximal group of positive involutions.
Let $v\in \{X\in \mathfrak v:\ J_{p_1}(X)=J_{p_2}(X)=J_{p_3}(X)=X\}$, $\|v\|^2=1$.
We construct an orthonormal invariant basis for $\mathfrak v$:
$$
\begin{array}{lllll}
&X_1=v,\quad &X_2=J_{Z_2}(v),\quad &X_3=J_{Z_4}(v),\quad &X_4=J_{Z_2}J_{Z_4}(v),\quad
\\
&X_5=J_{Z_1}(v),\quad &X_6=J_{Z_3}(v),\quad &X_7=J_{Z_5}(v),\quad &X_8=J_{Z_3}J_{Z_4}(v),\quad
\\
&X_9=J_{Z_6}(v),\quad &X_{10}=J_{Z_7}(v),\quad &X_{11}=J_{Z_2}J_{Z_6}(v),\quad &X_{12}=J_{Z_2}J_{Z_7}(v),\quad
\\
&X_{13}=J_{Z_4}J_{Z_6}(v),\quad &X_{14}=J_{Z_4}J_{Z_7}(v),\quad &X_{15}=J_{Z_2}J_{Z_4}J_{Z_6}(v),\quad &X_{16}=J_{Z_2}J_{Z_4}J_{Z_7}(v).
\end{array}
$$
We denote
$$
\mathfrak z_1=\spann\{Z_2,Z_4\},\quad \mathfrak z_2=\spann\{Z_1,Z_3,Z_5,Z_6,Z_7\},\quad \mathfrak z_{1,6}=\mathfrak z_1\oplus \mathfrak z_2
$$
$$
\mathfrak v_1=\spann\{X_1,\ldots X_4\},\quad \mathfrak v_2=\spann\{X_5,\ldots ,X_{16}\},\quad \mathfrak v_{1,6}=\mathfrak v_1\oplus \mathfrak v_2
$$
Note that $\mathfrak z_1\oplus \mathfrak v_1$ is isomorphic to $\mathfrak n_{0,2}$. To show that $J_{Z_k}(\mathfrak v_1)\subset \mathfrak v_2$ for $k=1,3,5,6,7$ one can check the following
$$
\begin{array}{lllll}
&J_{Z_1} (X_1)=\pm X_5,\quad &J_{Z_1} (X_2)=\pm X_6,\quad &J_{Z_1} (X_3)=\pm X_7,\quad &J_{Z_1} (X_4)=\pm X_8,
\\
&J_{Z_3} (X_1)=\pm X_6,\quad &J_{Z_3} (X_2)=\pm X_5,\quad &J_{Z_3} (X_3)=\pm X_8,\quad &J_{Z_3} (X_4)=\pm X_7,
\\
&J_{Z_5} (X_1)=\pm X_7,\quad &J_{Z_5} (X_2)=\pm X_8,\quad &J_{Z_5} (X_3)=\pm X_5,\quad &J_{Z_5} (X_4)=\pm X_6,
\\
&J_{Z_6} (X_1)=\pm X_9,\quad &J_{Z_6} (X_2)=\pm X_{11},\quad &J_{Z_6} (X_3)=\pm X_{13},\quad &J_{Z_6} (X_4)=\pm X_{15},
\\
&J_{Z_7} (X_1)=\pm X_{10},\quad &J_{Z_7} (X_2)=\pm X_{12},\quad &J_{Z_7} (X_3)=\pm X_{14},\quad &J_{Z_7} (X_4)=\pm X_{16}.
\end{array}
$$
Theorem~\ref{pr_totgeod_h1} and Corollary~\ref{co.totgeod.1} imply that the group $N_{1,6}$ is not geodesic orbit.
\\

{\it Pseudo-Riemannian $H$-type nilmanifold $N_{5,2}$}.

We choose an orthonormal basis for the centre as in~\eqref{eq:brs} and the generating set as in~\eqref{eq:16-52}.
Let $v\in \{X\in \mathfrak v:\ J_{p_1}(X)=J_{p_2}(X)=J_{p_3}(X)=X\}$, $\|v\|^2=1$ and fix the orthonormal invariant basis for $\mathfrak v$:
$$
\begin{array}{lllll}
&X_1=v,\quad &X_2=J_{Z_2}(v),\quad &X_3=J_{Z_6}(v),\quad &X_4=J_{Z_2}J_{Z_6}(v),\quad
\\
&X_5=J_{Z_1}(v),\quad &X_6=J_{Z_3}(v),\quad &X_7=J_{Z_4}(v),\quad &X_8=J_{Z_5}(v),\quad
\\
&X_9=J_{Z_7}(v),\quad &X_{10}=J_{Z_2}J_{Z_4}(v),\quad &X_{11}=J_{Z_2}J_{Z_5}(v),\quad &X_{12}=J_{Z_2}J_{Z_7}(v),\quad
\\
&X_{13}=J_{Z_4}J_{Z_6}(v),\quad &X_{14}=J_{Z_4}J_{Z_7}(v),\quad &X_{15}=J_{Z_2}J_{Z_4}J_{Z_6}(v),\quad &X_{16}=J_{Z_2}J_{Z_5}J_{Z_6}(v).
\end{array}
$$
We denote
$$
\mathfrak z_1=\spann\{Z_2,Z_6\},\quad \mathfrak z_2=\spann\{Z_1,Z_3,Z_4,Z_5,Z_7\},\quad \mathfrak z_{5,2}=\mathfrak z_1\oplus \mathfrak z_2
$$
$$
\mathfrak v_1=\spann\{X_1,\ldots X_4\},\quad \mathfrak v_2=\spann\{X_5,\ldots ,X_{16}\},\quad \mathfrak v_{5,2}=\mathfrak v_1\oplus \mathfrak v_2
$$
The Lie algebra $\mathfrak z_1\oplus \mathfrak v_1$ is isomorphic and isometric to $\mathfrak n_{1,1}$. To show that $J_{Z_k}(\mathfrak v_1)\subset \mathfrak v_2$ for $k=1,3,4,5,7$ we observe that
$$
\begin{array}{lllll}
&J_{Z_1} (X_1)=\pm X_5,\quad &J_{Z_1} (X_2)=\pm X_6,\quad &J_{Z_1} (X_3)=\pm X_9,\quad &J_{Z_1} (X_4)=\pm X_{12},
\\
&J_{Z_3} (X_1)=\pm X_6,\quad &J_{Z_3} (X_2)=\pm X_5,\quad &J_{Z_3} (X_3)=\pm X_{12},\quad &J_{Z_3} (X_4)=\pm X_9,
\\
&J_{Z_4} (X_1)=\pm X_7,\quad &J_{Z_4} (X_2)=\pm X_{10},\quad &J_{Z_4} (X_3)=\pm X_{13},\quad &J_{Z_4} (X_4)=\pm X_{15},
\\
&J_{Z_5} (X_1)=\pm X_8,\quad &J_{Z_5} (X_2)=\pm X_{11},\quad &J_{Z_5} (X_3)=\pm X_{14},\quad &J_{Z_5} (X_4)=\pm X_{16},
\\
&J_{Z_7} (X_1)=\pm X_{9},\quad &J_{Z_7} (X_2)=\pm X_{12},\quad &J_{Z_7} (X_3)=\pm X_{5},\quad &J_{Z_7} (X_4)=\pm X_{6}.
\end{array}
$$
Theorem~\ref{pr_totgeod_h1} and Corollary~\ref{co.totgeod.1} imply that the group $N_{5,2}$ is not geodesic orbit.
\\

{\it Pseudo-Riemannian $H$-type nilmanifold $N_{2,3}$}.
We choose an orthonormal basis for the center as in~\eqref{eq:brs} and the generating set
$$
PI=\{p_1=Z_1Z_4Z_5,\quad p_2=Z_1Z_2Z_3Z_4\}.
$$
Let $v\in \{X\in \mathfrak v_{2,3}:\ J_{p_1}(X)=J_{p_2}(X)=X\}$, $\|v\|^2=1$.
We construct an orthonormal invariant basis as in the case $N_{3,2}$.
All other calculations are analogous to the case $N_{3,2}$.

{\it Pseudo-Riemannian $H$-type nilmanifold $N_{2,4}$}.
We choose an orthonormal basis for the center as in~\eqref{eq:brs} and the generating set
$$
PI=\{p_1=Z_1Z_4Z_5,\quad p_2=Z_1Z_2Z_3Z_4,\quad p_3=Z_1Z_2Z_5Z_6\}.
$$
Let $v\in \{X\in \mathfrak v:\ J_{p_1}(X)=J_{p_2}(X)=J_{p_3}(X)=X\}$, $\|v\|^2=1$.
We construct an orthonormal invariant basis for $\mathfrak v$:
$$
\begin{array}{lllll}
&X_1=v,\quad &X_2=J_{Z_5}(v),\quad &X_3=J_{Z_6}(v),\quad &X_4=J_{Z_5}J_{Z_6}(v),\quad
\\
&X_5=J_{Z_1}(v),\quad &X_6=J_{Z_2}(v),\quad &X_7=J_{Z_3}(v),\quad &X_8=J_{Z_4}(v).
\end{array}
$$
We denote
$$
\mathfrak z_1=\spann\{Z_5,Z_6\},\quad \mathfrak z_2=\spann\{Z_1,Z_2,Z_3,Z_4\},\quad \mathfrak z_{2,4}=\mathfrak z_1\oplus \mathfrak z_2,
$$
$$
\mathfrak v_1=\spann\{X_1,\ldots X_4\},\quad \mathfrak v_2=\spann\{X_5,\ldots ,X_{8}\},\quad \mathfrak v_{2,4}=\mathfrak v_1\oplus \mathfrak v_2.
$$
It is obvious that $J_{Z_k}(\mathfrak v_1)\subset \mathfrak v_1$ for $k=5,6$. The pseudo $H$-type Lie algebra $\mathfrak z_1\oplus\mathfrak v_1$ is isomorphic to $\mathfrak n_{0,2}$.
To show that $J_{Z_k}(\mathfrak v_1)\subset \mathfrak v_2$ for $k=1,2,3,4$ one can check the following
$$
\begin{array}{lllll}
&J_{Z_1} (X_1)=\pm X_5,\quad &J_{Z_1} (X_2)=\pm X_8,\quad &J_{Z_1} (X_3)=\pm X_7,\quad &J_{Z_1} (X_4)=\pm X_6,\quad
\\
&J_{Z_2} (X_1)=\pm X_6,\quad &J_{Z_2} (X_2)=\pm X_7,\quad &J_{Z_2} (X_3)=\pm X_8,\quad &J_{Z_2} (X_4)=\pm X_5,\quad
\\
&J_{Z_3} (X_1)=\pm X_7,\quad &J_{Z_3} (X_2)=\pm X_6,\quad &J_{Z_3} (X_3)=\pm X_5,\quad &J_{Z_3} (X_4)=\pm X_8,\quad
\\
&J_{Z_4} (X_1)=\pm X_8,\quad &J_{Z_4} (X_2)=\pm X_5,\quad &J_{Z_4} (X_3)=\pm X_6,\quad &J_{Z_4} (X_4)=\pm X_7.
\end{array}
$$
Theorem~\ref{pr_totgeod_h1} and Corollary~\ref{co.totgeod.1} imply that pseudo $H$-type Lie group $N_{2,4}$ is not geodesic orbit.

{\it Pseudo-Riemannian $H$-type nilmanifold $N_{3,3}$}.
We choose an orthonormal basis for the center as in~\eqref{eq:brs} and the generating set
$$
PI=\{p_1=Z_1Z_2Z_3,\quad p_2=Z_2Z_3Z_4Z_5,\quad p_3=Z_1Z_2Z_5Z_6\}.
$$
Let $v\in \{X\in \mathfrak v:\ J_{p_1}(X)=J_{p_2}(X)=J_{p_3}(X)=X\}$, $\|v\|^2=1$.
We construct an orthonormal invariant basis for $\mathfrak v$:
$$
\begin{array}{lllll}
&X_1=v,\quad &X_2=J_{Z_3}(v),\quad &X_3=J_{Z_4}(v),\quad &X_4=J_{Z_3}J_{Z_4}(v),\quad
\\
&X_5=J_{Z_1}(v),\quad &X_6=J_{Z_2}(v),\quad &X_7=J_{Z_5}(v),\quad &X_8=J_{Z_6}(v).
\end{array}
$$
We denote
$$
\mathfrak z_1=\spann\{Z_3,Z_4\},\quad \mathfrak z_2=\spann\{Z_1,Z_2,Z_5,Z_6\},\quad \mathfrak z_{3,3}=\mathfrak z_1\oplus \mathfrak z_2,
$$
$$
\mathfrak v_1=\spann\{X_1,\ldots X_4\},\quad \mathfrak v_2=\spann\{X_5,\ldots ,X_{8}\},\quad \mathfrak v_{3,3}=\mathfrak v_1\oplus \mathfrak v_2.
$$
It is obvious that $J_{Z_k}(\mathfrak v_1)\subset \mathfrak v_1$ for $k=3,4$. The pseudo $H$-type Lie algebra $\mathfrak z_1\oplus\mathfrak v_1$ is isomorphic to $\mathfrak n_{1,1}$.
To show that $J_{Z_k}(\mathfrak v_1)\subset \mathfrak v_2$ for $k=1,2,5,6$ one can check the following
$$
\begin{array}{lllll}
&J_{Z_1} (X_1)=\pm X_5,\quad &J_{Z_1} (X_2)=\pm X_6,\quad &J_{Z_1} (X_3)=\pm X_7,\quad &J_{Z_1} (X_4)=\pm X_8,\quad
\\
&J_{Z_2} (X_1)=\pm X_6,\quad &J_{Z_2} (X_2)=\pm X_5,\quad &J_{Z_2} (X_3)=\pm X_8,\quad &J_{Z_2} (X_4)=\pm X_7,\quad
\\
&J_{Z_5} (X_1)=\pm X_7,\quad &J_{Z_5} (X_2)=\pm X_8,\quad &J_{Z_5} (X_3)=\pm X_5,\quad &J_{Z_5} (X_4)=\pm X_6,\quad
\\
&J_{Z_6} (X_1)=\pm X_8,\quad &J_{Z_6} (X_2)=\pm X_7,\quad &J_{Z_6} (X_3)=\pm X_6,\quad &J_{Z_6} (X_4)=\pm X_5.
\end{array}
$$
Theorem~\ref{pr_totgeod_h1} and Corollary~\ref{co.totgeod.1} imply that pseudo $H$-type Lie group $N_{3,3}$ is not geodesic orbit.

To finish the proof we apply Theorem~\ref{th:1-3} to the cases
$$
\ell(1,4)=\ell(1,5),\quad \ell(3,2)=\ell(4,2),\quad \ell(2,4)=\ell(2,5),\quad \ell(3,3)=\ell(4,3).
$$
\end{proof}

\section{Pseudo-Riemannian $H$-type nilmanifold ${N}_{3,4}$}\label{sec.N3.4}

It is known that  the Clifford algebra $\operatorname{Cl}(\mathbb{R}^{3,4})$ has two non-equivalent minimal admissible $8$-dimensional modules.
Moreover, the corresponding $15$-dimensional pseudo $H$-type Lie algebras are isomorphic and isometric, see~\cite[Theorem 12]{FuMa17}.
Therefore, it suffices to consider only one such pseudo $H$-type nilmanifold $N_{3,4}$.

The pseudo $H$-type nilmanifold $N_{3,4}$ has dimension $15$,
and the corresponding left invariant metric is generated by
the scalar product
$\langle \cdot\, , \cdot\rangle_{3,4}+\langle \cdot\, , \cdot\rangle_{4,4}$
on the Lie algebra $\mathfrak{n}_{3,4}=\mathbb R^{3,4}\oplus\mathbb R^{4,4}$.
We define the orthonormal basis $\{Z_1,\ldots,Z_7,V_1,\ldots, V_8\}$ by taking $v\in \mathbb R^{4,4}$ with $\|v\|^2=1$ and setting
$\|Z_1\|^2=\|Z_2\|^2=\|Z_3\|^2=1$, $\|Z_4\|^2=\|Z_5\|^2=\|Z_6\|^2=\|Z_7\|^2=-1$.
We use involutions $p_i$, $i=1,2,3,4$, where
\begin{equation}\label{eq:inv34}
    \begin{array}{lll}
&J_{p_1}(v)=J_{Z_1Z_2Z_4Z_5}(v)=v,\quad &J_{p_2}(v)=J_{Z_1Z_2Z_6Z_7}(v)=v,\quad
\\
 &J_{p_3}(v)=J_{Z_1Z_3Z_5Z_7}(v)=v,\quad &J_{p_4}(v)=J_{Z_1Z_2Z_3}(v)=v,
\end{array}
\end{equation}
with
$$
J_{Z_1}^2=J_{Z_2}^2=J_{Z_3}^2=-\Id,\quad J_{Z_4}^2=J_{Z_5}^2=J_{Z_6}^2=J_{Z_7}^2=\Id \,.
$$
%Let us recall some useful relations:
%$$
%P_1P_3P_4v=-J_1J_4J_7v=v,\quad P_2P_3P_4v=-J_1J_5J_6v=v,\quad P_1P_2P_3P_4v=-J_2J_4J_6v=v
%$$
%$$
%P_3P_4v=J_2J_5J_7v=v,\quad P_1P_4v=-J_3J_4J_5v=v,\quad P_2P_4v=-J_3J_6J_7v=v.
%$$
Note that
\begin{equation*}
    \begin{array}{llllllllll}
J_{Z_1}(v)&=&-J_{Z_2}J_{Z_3}(v)&=&J_{Z_4}J_{Z_7}(v)&=&J_{Z_5}J_{Z_6}(v),\quad
\\
J_{Z_2}(v)&=&J_{Z_1}J_{Z_3}(v)&=&J_{Z_4}J_{Z_6}(v)&=&-J_{Z_5}J_{Z_7}(v),\quad
\\
J_{Z_3}(v)&=&-J_{Z_1}J_{Z_2}(v)&=&J_{Z_4}J_{Z_5}(v)&=&J_{Z_6}J_{Z_7}(v),
\\
J_{Z_4}(v)&=&J_{Z_1}J_{Z_7}(v)&=&J_{Z_2}J_{Z_6}(v)&=&J_{Z_3}J_{Z_5}(v),
\\
J_{Z_5}(v)&=&-J_{Z_2}J_{Z_7}(v)&=&-J_{Z_3}J_{Z_4}(v)&=&J_{Z_1}J_{Z_6}(v),\quad
\\
J_{Z_6}(v)&=&-J_{Z_1}J_{Z_5}(v)&=&-J_{Z_2}J_{Z_4}(v)&=&J_{Z_3}J_{Z_7}(v),
\\
J_{Z_7}(v)&=&-J_{Z_1}J_{Z_4}(v)&=&J_{Z_2}J_{Z_5}(v)&=&-J_{Z_3}J_{Z_6}(v).
\end{array}
\end{equation*}
We construct a basis for  $\mathfrak{v}$ by setting $V_1=v$ and $V_i=J_{Z_{i-1}}(v)$.
%$$
%V_1=v,\quad V_2=J_1v,\quad V_3=J_2v,\quad V_4=J_3v,\quad
%V_5=J_4v,\quad V_6=J_5v,\quad V_7=J_6v,\quad V_8=J_7v,\quad
%$$

Then we have
$$
\langle V_k,V_k\rangle_{4,4}=-\langle V_l,V_l\rangle_{4,4}=1,\quad k=1,\ldots,4,\ \ l=5,\ldots,8,
$$
and the operators $J_{Z_k}$, $k=1,\ldots,7$, with respect to this basis take the form
%Note, that $\eta =\diag(1,1,1,1,-1-1,-1,-1)$.
\begin{equation}\label{eq:mat34}
    \begin{array}{lllll}
&J_{Z_1}=\left(\begin{smallmatrix}
0& -1& 0& 0& 0& 0& 0& 0\\
1& 0& 0& 0& 0& 0& 0& 0\\
0& 0& 0& 1& 0& 0& 0& 0\\
0& 0& -1& 0& 0& 0& 0& 0\\
0& 0& 0& 0& 0& 0& 0& 1\\
0& 0& 0& 0& 0& 0& 1& 0\\
0& 0& 0& 0& 0& -1& 0& 0\\
0& 0& 0& 0& -1& 0& 0& 0
\end{smallmatrix}\right),
&J_{Z_2}=
\left(\begin{smallmatrix}
 0& 0& -1& 0& 0& 0& 0& 0\\
0& 0& 0& -1& 0& 0& 0& 0\\
1& 0& 0& 0& 0& 0& 0& 0\\
0& 1& 0& 0& 0& 0& 0& 0\\
0& 0& 0& 0& 0& 0& 1& 0\\
0& 0& 0& 0& 0& 0& 0& -1\\
0& 0& 0& 0& -1& 0& 0& 0\\
0& 0& 0& 0& 0& 1& 0& 0
\end{smallmatrix}\right),
&J_{Z_3}=
\left(\begin{smallmatrix}
 0& 0& 0& -1& 0& 0& 0& 0\\
0& 0& 1& 0& 0& 0& 0& 0\\
0& -1& 0& 0& 0& 0& 0& 0\\
1& 0& 0& 0& 0& 0& 0& 0\\
0& 0& 0& 0& 0& 1& 0& 0\\
0& 0& 0& 0& -1& 0& 0& 0\\
0& 0& 0& 0& 0& 0& 0& 1\\
0& 0& 0& 0& 0& 0& -1& 0
\end{smallmatrix}\right),
\\
&J_{Z_4}=\left(\begin{smallmatrix}
0& 0& 0& 0& 1& 0& 0& 0\\
0& 0& 0& 0& 0& 0& 0& 1\\
0& 0& 0& 0& 0& 0& 1& 0\\
0& 0& 0& 0& 0& 1& 0& 0\\
1& 0& 0& 0& 0& 0& 0& 0\\
0& 0& 0& 1& 0& 0& 0& 0\\
0& 0& 1& 0& 0& 0& 0& 0\\
0& 1& 0& 0& 0& 0& 0& 0
\end{smallmatrix}\right),
&J_{Z_5}=
\left(\begin{smallmatrix}
0& 0& 0& 0& 0& 1& 0& 0\\
0& 0& 0& 0& 0& 0& 1& 0\\
0& 0& 0& 0& 0& 0& 0& -1\\
0& 0& 0& 0& -1& 0& 0& 0\\
0& 0& 0& -1& 0& 0& 0& 0\\
1& 0& 0& 0& 0& 0& 0& 0\\
0& 1& 0& 0& 0& 0& 0& 0\\
0& 0& -1& 0& 0& 0& 0& 0
\end{smallmatrix}\right),
&J_{Z_6}=\left(\begin{smallmatrix}
0& 0& 0& 0& 0& 0& 1& 0\\
 0& 0& 0& 0& 0& -1& 0& 0\\
 0& 0& 0& 0& -1& 0& 0& 0\\
 0& 0& 0& 0& 0& 0& 0& 1\\
 0& 0& -1& 0& 0& 0& 0& 0\\
 0& -1& 0& 0& 0& 0& 0& 0\\
 1& 0& 0& 0& 0& 0& 0& 0\\
 0& 0& 0& 1& 0& 0& 0& 0
\end{smallmatrix}\right),\quad
\\
&&J_{Z_7}=
\left(\begin{smallmatrix}
0& 0& 0& 0& 0& 0& 0& 1\\
 0& 0& 0& 0& -1& 0& 0& 0\\
 0& 0& 0& 0& 0& 1& 0& 0\\
 0& 0& 0& 0& 0& 0& -1& 0\\
 0& -1& 0& 0& 0& 0& 0& 0\\
 0& 0& 1& 0& 0& 0& 0& 0\\
 0& 0& 0& -1& 0& 0& 0& 0\\
 1& 0& 0& 0& 0& 0& 0& 0
\end{smallmatrix}\right).
\end{array}
\end{equation}

The operators $J_{Z_k}$, $k=1,\dots,7$, span a $7$-dimensional subspace in $\mathbf{V}\subset\mathfrak{so}(4,4)$, nevertheless the vector space $\mathbf{V}$ is not a Lie subalgebra of $\mathfrak{so}(4,4)$. Hence, $N_{3,4}$ is not a naturally reductive manifold.

\begin{table}[t]
\center\caption{Commutators in $\mathfrak n_{3,4}$}
\bigskip
\begin{tabular}{|c||c|c|c|c|c|c|c|c|}
\hline
&$ V_1$&$ V_2$&$ V_3$&$ V_4$&$ V_5$&$ V_6$&$ V_7$&$V_8$
\\
\hline
\hline
$V_1$&$0$&$Z_1$&$Z_2$&$Z_3$&$Z_4$&$Z_5$&$Z_6$&$Z_7$
\\
\hline
$V_2$ &$-Z_1$&$0$&$-Z_3$&$Z_2$&$-Z_7$&$-Z_6$&$Z_5$&$Z_4$
\\
\hline
$V_3$ &$-Z_2$&$Z_3$&$0$&$-Z_1$&$-Z_6$&$Z_7$&$Z_4$&$-Z_5$
\\
\hline
$V_4$&$-Z_3$&$-Z_2$&$Z_1$&$0$&$-Z_5$&$Z_4$&$-Z_7$&$Z_6$
\\
\hline
$V_5$&$-Z_4$&$Z_7$&$Z_6$&$Z_5$&$0$&$Z_3$&$Z_2$&$Z_1$
\\
\hline
$V_6$&$-Z_5$&$Z_6$&$-Z_7$&$-Z_4$&$-Z_3$&$0$&$Z_1$&$-Z_2$
\\
\hline
$V_7$&$-Z_6$&$-Z_5$&$-Z_4$&$Z_7$&$-Z_2$&$-Z_1$&$0$&$Z_3$
\\
\hline
$V_8$ &$-Z_7$&$-Z_4$&$Z_5$& $-Z_6$&$-Z_1$&$Z_2$&$-Z_3$&$0$
\\
\hline
\end{tabular}\label{t:dim}
\end{table}

Direct computations with matrices show that the centralizer $\mathbf{Z}$ of $\mathbf{V}$ in the isotropy subalgebra is trivial. By Proposition~\ref{pr.triple.1},
the normalizer $\mathbf{N}=\mathbf{Z} \oplus [\mathbf{V},\mathbf{V}]=[\mathbf{V},\mathbf{V}]$ of $\mathbf{V}$
is the linear span of all matrices of the following type: $J_{ik}:=[J_{Z_i},J_{Z_k}]$, $1\leq i < k \leq 7$.
In particular, $\dim \mathbf{N}=21$ and any operator $B \in \mathbf{N}$ has the form
$B=\sum_{i,k} x_{ik} J_{ik}$ for some $x_{ik} \in \mathbb{R}$, $1\leq i <k \leq 7$.
An explicit form of the operator $B$ is as follows:

{\small
$$
2\left(\begin{smallmatrix}
0& x_{23}-x_{47}-x_{56}& -x_{13}-x_{46}+x_{57}& x_{12}-x_{45}-x_{67}& x_{17}+x_{26}+x_{35}& x_{16}-x_{27}-x_{34}& -x_{15}-x_{24}+x_{37}& -x_{14}+x_{25}-x_{36}\\
-x_{23}+x_{47}+x_{56}& 0& -x_{12}-x_{45}-x_{67}& -x_{13}+x_{46}-x_{57}& x_{14}+x_{25}-x_{36}& x_{15}-x_{24}+x_{37}& x_{16}+x_{27}+x_{34}& x_{17}-x_{26}-x_{35}\\
x_{13}+x_{46}-x_{57}& x_{12}+x_{45}+x_{67}& 0& -x_{23}-x_{47}-x_{56}& -x_{15}+x_{24}+x_{37}& x_{14}+x_{25}+x_{36}& -x_{17}+x_{26}-x_{35}& x_{16}+x_{27}-x_{34}\\
-x_{12}+x_{45}+x_{67}& x_{13}-x_{46}+x_{57}& x_{23}+x_{47}+x_{56}& 0& x_{16}-x_{27}+x_{34}& -x_{17}-x_{26}+x_{35}& -x_{14}+x_{25}+x_{36}& x_{15}+x_{24}+x_{37}\\
x_{17}+x_{26}+x_{35}& x_{14}+x_{25}-x_{36}& -x_{15}+x_{24}+x_{37}& x_{16}-x_{27}+x_{34}& 0& x_{12}+x_{45}-x_{67}& -x_{13}+x_{46}+x_{57}& x_{23}+x_{47}-x_{56}\\
x_{16}-x_{27}-x_{34}& x_{15}-x_{24}+x_{37}& x_{14}+x_{25}+x_{36}& -x_{17}-x_{26}+x_{35}& -x_{12}-x_{45}+x_{67}& 0& x_{23}-x_{47}+x_{56}& x_{13}+x_{46}+x_{57}\\
-x_{15}-x_{24}+x_{37}& x_{16}+x_{27}+x_{34}& -x_{17}+x_{26}-x_{35}& -x_{14}+x_{25}+x_{36}& x_{13}-x_{46}-x_{57}& -x_{23}+x_{47}-x_{56}& 0& x_{12}-x_{45}+x_{67}\\
-x_{14}+x_{25}-x_{36}& x_{17}-x_{26}-x_{35}& x_{16}+x_{27}-x_{34}& x_{15}+x_{24}+x_{37}& -x_{23}-x_{47}+x_{56}& -x_{13}-x_{46}-x_{57}& -x_{12}+x_{45}-x_{67}& 0
\end{smallmatrix}\right)
$$
}
\medskip

\begin{remark}\label{re.act.rank.1}
Recall that the matrices $J_{Z_i}$, $i=1,\dots,7$, as well as $B$ are skew-symmetric with respect the inner product $\langle \cdot\, , \cdot\rangle_{4,4}$.
Therefore, if $A$ is one of such matrices and $Y\in \mathfrak{v}=\mathbb{R}^{4,4}$, then the vector $\overline{Y}=A(Y)$ is such that
$$
\langle Y, \overline{Y}\rangle_{4,4}=y_1\overline{y}_1+y_2\overline{y}_2+y_3\overline{y}_3+y_4\overline{y}_4-y_5\overline{y}_5-y_6\overline{y}_6-y_7\overline{y}_7-y_8\overline{y}_8=0.
$$
This implies that for a given $Y$, the linear space $\mathbf{L}(Y)$ for $\mathbf{L}=\mathbf{V} \oplus [\mathbf{V},\mathbf{V}]$ has dimension $\leq 7$.
\end{remark}

The following result is useful for various computations.

\begin{lemma}\label{le.invarform.1}
If a matrix Lie group $G$ with the Lie algebra $\mathfrak{g}\subset \mathfrak{gl}(m,\mathbb{R})$ preserves the value $\sum_{ij} c_{ij} y_i y_j$ for some fixed
$Y=(y_1,\dots,y_m)$, where $c_{ij} \in \mathbb{R}$, then for any matrix $A\in \mathfrak{g}$ we have $\sum_{ij} c_{ij} \Bigl((A(Y))_i y_j+(A(Y))_j y_i\Bigr)=0$,
where $A(Y)$ is the image of $Y$ under the action of $A$.
\end{lemma}

\begin{proof} For any $Q\in G$ and any $Y\in \mathbb{R}^m$ we have $\sum_{ij} c_{ij} (Q(Y))_i (Q(Y))_j=\sum_{ij} c_{ij} y_i y_j$, where $G(Y)$ is the image of $Y$ under the action of the matrix $G$. If $Q=\exp(tA)=\Id +t A +o(t)$ when $t \to 0$ for some $A\in \mathfrak{g}$, then
\begin{eqnarray*}
\sum_{ij} c_{ij} (Q(Y))_i (Q(Y))_j&=&\sum_{ij} c_{ij} (Y+tA(Y))_i (Y+tA(Y))_j +o(t)\\
&=&\sum_{ij} c_{ij} y_i y_j +t\sum_{ij} c_{ij} \Bigl(A(Y)_i y_j+A(Y)_j y_i\Bigr) +o(t)
\end{eqnarray*}
when $t \to 0$, that proves the lemma.
\end{proof}

\medskip

If $N_{3,4}$ is geodesic orbit, then for any $Y\in \mathfrak{v}$ and any $Z\in \mathbf{V}$, there is some
$B \in \mathbf{N}$ such that $[B,Z]=0$ and $B(Y)=Z(Y)$. We aim to find such matrix $B$.
Without loss of generality we may take
\begin{equation}\label{eq.yn34.1}
Y=(y_1,y_2,y_3,y_4,y_5,y_6,y_7,y_8)=\sum_{j=1}^8y_jV_j\in \mathfrak{v}
\end{equation}
for $y_i \in \mathbb{R}$,  $i=1,\dots,8$.

\begin{theorem}\label{the.0case}
For any $Z \in \mathbf{V}$ and for any $Y \in \mathbb{R}^{4,4}$, it is possible to find $B\in \mathbf{N}= [\mathbf{V},\mathbf{V}]$ such that  $[B,Z]=0$ and $B(Y)=Z(Y)$.
Hence, $N_{3,4}$ is a pseudo-Riemaniann geodesic orbit nilmanifold.
\end{theorem}

\begin{remark}
In particular, $N_{3,4}$ is the first example of pseudo $H$-type geodesic orbit manifolds that is not naturally reductive.
Moreover, it is the first pseudo $H$-type geodesic orbit manifold such that the space $\mathbf{V}$
satisfies the {\it strong transitive normalizer condition}; %(compare with Definition~\ref{???});
that is
for every $Y\in \mathfrak{v}=\mathbb{R}^{l,l}$ and every $Z\in \mathbf{V}=J(\mathfrak{z})$
    there is some $B \in [\mathbf{V}, \mathbf{V}]$ such that
    $[B,Z]_{\mathfrak{so}(l,l)}=0$ and $B(Y)=Z(Y)$.
\end{remark}

\begin{remark}\label{re.dusec.conj}
It should be especially emphasized that all homogeneous vectors in Theorem \ref{the.0case} are constructed for $k=0$ (see Proposition \ref{gonil1n.2}),
hence, the obtained pseudo $H$-type geodesic orbit manifold supports Conjecture 4.4 in \cite{Du2009}
for geodesic orbit pseudo-Riemannian spaces. Recall that this conjecture for almost geodesic orbit pseudo-Riemannian spaces was refuted in
\cite{Bar}, see Remark \ref{re.N_1.1}.
On the other hand, Theorem \ref{the.0case} refutes Conjecture 4.3  in \cite{Du2009}, which states that
any geodesic orbit pseudo-Riemannian reductive homogeneous space $G/H$ with noncompact
isotropy group $H$ is naturally reductive.
\end{remark}

The proof of Theorem \ref{the.0case} is based on Propositions \ref{pr.1case}, \ref{pr.2case}, and Corollary \ref{pr.3case}.

We also use the fact that we have only three classes in the center of $\mathfrak{n}_{3,4}$, up to similarity and the action of the isotropy subgroup.
They are represented by $Z_1$ (positive vector), $Z_4$ (negative vector), and $Z_1+Z_4$ (null vector).

Indeed, the group $O(3,4)$ with the Lie algebra $\mathfrak{so}(3,4)$, that is a linear span of the operators $[J_{Z_i},J_{Z_j}]$, $1\leq i <j \leq 7$,
acts naturally by automorphisms on the center $\mathfrak{z}$. It is well known that $O(3,4)$ acts transitively on every hyperquadric
$Q(r)=\{z \in \mathfrak{z} \,|\,\langle z, z\rangle_{3,4}=r\}$ with $r\neq 0$, see e.g.~\cite[page~239 ]{ONeill} or~\cite[Theorem 2.4.4 ]{Wolf2011}.

Now, let us suppose that $z,\overline{z} \in \mathfrak{z}$ and $\langle z, z\rangle_{3,4}=\langle \overline{z}, \overline{z} \rangle_{3,4}=0$. If
$\langle z_p, z_p\rangle_{3,4}=\langle \overline{z}_p, \overline{z}_p \rangle_{3,4}=\rho\neq 0$ or, equivalently,
$\langle z_n, z_n\rangle_{3,4}=\langle \overline{z}_n, \overline{z}_n \rangle_{3,4}=-\rho\neq 0$, where the subscripts $p$ and $n$ mean the components
of vectors in $\spann(Z_1,Z_2,Z_3)$ and $\spann(Z_4,Z_5,Z_6,Z_7)$ respectively.
Then it is easy to see  that there is $Q\in O(3)\cdot O(4) \subset O(3,4)$ such that $Q(z_p)=\overline{z}_p$ and $Q(z_n)=\overline{z}_n$.

Hence, up to a similarity, it suffices to consider only one point in every of the following hyperquadrics: $Q(1)$, $Q(-1)$, $Q(0)$ (non-trivial in the last case).
For instance, we can take $Z$ satisfying one of the following  possibilities:
$$
Z= J_{Z_1}\in \mathbf{V}, \quad Z=J_{Z_4}\in \mathbf{V}, \quad Z=J_{Z_1}+J_{Z_4}=J_{Z_1+Z_4}\in \mathbf{V},
$$
which corresponds to positive, negative or zero length of $Z$.

\begin{remark}
It would be interesting to find a shorter and more conceptual proof of Theorem~\ref{the.0case}. In our proof we have used some standard results on Lie algebras,
on representations of Lie groups, as well as classical results in the linear algebra and properties of polynomial ideals.
\end{remark}

It is clear that for trivial $Y$ (i.e. $y_i=0$ for all $i=1,\dots,8)$ and $Z$ as above, we can take the trivial (zero) operator $B$ in order to get equalities
$B(Y)=Z(Y)$ and $[B,Z]=0$. In what follows, it suffices to check only non-trivial $Y$.

\subsection{The case $Z= J_{Z_1}$}

The first condition $[B,Z]=0$ implies $x_{12}=x_{13}=x_{14}=x_{15}=x_{16}=x_{17}=0$.
Let us consider the condition $B(Y)=Z(Y)$ for a non-trivial $B$.
Let us assume in addition that
$x_{23}=x_{24}=x_{25}=x_{26}=x_{27}=x_{34}=x_{35}=x_{36}=x_{37}=0$.
Then we get the following explicit solutions:
\smallskip

If $y_1^2+y_2^2+y_3^2+y_4^2\neq 0$ and $y_5^2+y_6^2+y_7^2+y_8^2=0$, then we can take $x_{56}=x_{57}=x_{67}=0$ and
\begin{equation*}
x_{45}=-\frac{y_1y_3-y_2y_4}{y_1^2+y_2^2+y_3^2+y_4^2},\quad x_{46}=\frac{y_1y_4+y_2y_3}{y_1^2+y_2^2+y_3^2+y_4^2},\quad x_{47}=\frac{y_1^2+y_2^2-y_3^2-y_4^2}{2(y_1^2+y_2^2+y_3^2+y_4^2)}.
\end{equation*}
If $y_5^2+y_6^2+y_7^2+y_8^2\neq 0$ and $y_1^2+y_2^2+y_3^2+y_4^2=0$, then we can take $x_{56}=x_{57}=x_{67}=0$ and
\begin{equation*}
x_{45}=-\frac{y_5y_7-y_6y_8}{y_5^2+y_6^2+y_7^2+y_8^2},\quad x_{46}=\frac{y_5y_6+y_7y_8}{y_5^2+y_6^2+y_7^2+y_8^2},\quad x_{47}=\frac{y_5^2-y_6^2-y_7^2+y_8^2}{2(y_5^2+y_6^2+y_7^2+y_8^2)}.
\end{equation*}
Finally, if $y_1^2+y_2^2+y_3^2+y_4^2\neq 0$ and $y_5^2+y_6^2+y_7^2+y_8^2\neq 0$, then we can take
\begin{eqnarray*}
x_{45}&=&-\frac{1}{2}\left(\frac{y_1y_3-y_2y_4}{y_1^2+y_2^2+y_3^2+y_4^2}+\frac{y_5y_7-y_6y_8}{y_5^2+y_6^2+y_7^2+y_8^2}\right),\\
x_{46}&=&\frac{1}{2}\left(\frac{y_1y_4+y_2y_3}{y_1^2+y_2^2+y_3^2+y_4^2}+\frac{y_5y_6+y_7y_8}{y_5^2+y_6^2+y_7^2+y_8^2}\right),\\
x_{47}&=&\frac{(y_1^2+y_2^2)(y_5^2+y_8^2)-(y_3^2+y_4^2)(y_6^2+y_7^2)}{2(y_1^2+y_2^2+y_3^2+y_4^2)(y_5^2+y_6^2+y_7^2+y_8^2)},\\
x_{56}&=&\frac{(y_1^2+y_2^2)(y_6^2+y_7^2)-(y_3^2+y_4^2)(y_5^2+y_8^2)}{2(y_1^2+y_2^2+y_3^2+y_4^2)(y_5^2+y_6^2+y_7^2+y_8^2)},\\
x_{57}&=&\frac{1}{2}\left(-\frac{y_1y_4+y_2y_3}{y_1^2+y_2^2+y_3^2+y_4^2}+\frac{y_5y_6+y_7y_8}{y_5^2+y_6^2+y_7^2+y_8^2}\right),\\
x_{67}&=&\frac{1}{2}\left(\frac{-y_1y_3+y_2y_4}{y_1^2+y_2^2+y_3^2+y_4^2}+\frac{y_5y_7+y_6y_8}{y_5^2+y_6^2+y_7^2+y_8^2}\right).\\
\end{eqnarray*}

Hence, we have  the following result.

\begin{prop}\label{pr.1case}
For any $Y \in \mathbb{R}^{4,4}$, it is possible to find $B\in \mathbf{N}= [\mathbf{V},\mathbf{V}]$ of $\mathbf{V}$
such that  $[B,Z]=[B,J_{Z_1}]=0$ and $B(Y)=Z(Y)=J_{Z_1}(Y)$.
\end{prop}

\subsection{The case $Z= J_{Z_4}$}

The first condition $[B,Z]=0$ implies $x_{14}=x_{24}=x_{34}=x_{45}=x_{46}=x_{47}=0$.
Let us consider a subalgebra $\mathbf{N}_1 \subset \mathbf{N}$ that is generated by the matrices of the type $[J_{Z_k},J_{Z_l}]$, where $k,l \in \{1,2,3,5,6,7\}$.
It is clear that $[J_{Z_4}, \mathbf{N}_1]=0$.

Let us consider the condition $B(Y)=Z(Y)$. The above arguments show that we need to consider only $B\in  \mathbf{N}_1$.
It is more convenient to consider the linear system $2B(Y)=Z(Y)=J_{Z_4}(Y)$ which allows to find $B=\{x_{ij}\}$.
Since $J_{Z_4}(Y)=(y_5,y_8,y_7,y_6,y_1,y_4,y_3,y_2)$, then we have a system of linear equations with respect to the variables
$$
x_{12}, x_{13}, x_{15}, x_{16}, x_{17}, x_{23}, x_{25}, x_{26},x_{27},x_{35},x_{36}, x_{37}, x_{56},x_{57},x_{67}
$$
with the following
extended matrix (the last column is the column of free terms of our system of linear equations):
$$
{\small ME}:=
{\Large\left(
\begin{smallmatrix}
-y_6&y_7&y_3&-y_4&-y_1&-y_8&-y_2&-y_1&y_4&-y_1&y_2&-y_3&y_8&-y_7&y_6&-y_1\\
y_7& y_6& -y_4& -y_3& -y_2& y_5& -y_1& y_2& -y_3& y_2& y_1& -y_4& -y_5& y_6& y_7& -y_2\\
-y_8& -y_5& y_1& -y_2& y_3& y_6& -y_4& -y_3& -y_2& y_3& -y_4& -y_1& y_6& y_5& -y_8& -y_3\\
y_5& -y_8& -y_2& -y_1& y_4& -y_7& -y_3& y_4& y_1& -y_4& -y_3& -y_2& -y_7& -y_8& -y_5& -y_4\\
-y_4& y_3& y_7& -y_6& -y_5& -y_2& -y_8& -y_5& y_6& -y_5& y_8& -y_7& y_2& -y_3& y_4& -y_5\\
y_1& -y_2& -y_8& -y_5& y_6& -y_3& -y_7& y_6& y_5& -y_6& -y_7& -y_8& -y_3& -y_2& -y_1& -y_6\\
-y_2& -y_1& y_5& -y_8& y_7& y_4& -y_6& -y_7& -y_8& y_7& -y_6& -y_5& y_4& y_1& -y_2& -y_7\\
y_3& y_4& -y_6& -y_7& -y_8& y_1& -y_5& y_8& -y_7& y_8& y_5& -y_6& -y_1& y_4& y_3& -y_8
\end{smallmatrix}
\right).}
$$

By the Kronecker--Capelli theorem, this system has a solution if and only if the rank of $ME$ coincides with the rank of the matrix $M$, which is
obtained from $ME$ by deleting the last column (the column of free terms).

It follows that the product of the vector $(-y_5, -y_8, -y_7, -y_6, y_1, y_4, y_3, y_2)$ and $ME$ is a vector with $16$ zero entries.
Hence, $\rank(ME)\leq 7$ (see also Remark \ref{re.act.rank.1}).
Therefore, if $\rank(M)=7$, then the system $2B(Y)=J_{Z_4}(Y)$ has a solution.
By symbol $M_{\,[i_1, i_2, \dots,i_s]}^{[j_1, j_2, \dots,j_s]}$, we denote the minor of $M$ which is determined by the rows with the numbers $i_1,i_2,\dots,i_s$ and columns with the numbers $j_1, j_2, \dots,j_s$.
It is easy to check that
\begin{eqnarray*}
M_{\,[1, 2, 3, 4, 6, 7, 8]}^{[9, 10, 11, 12, 13, 14, 15]}= y_1 \cdot (y_1^2 + y_2^2 - y_3^2 - y_4^2 - y_5^2 + y_6^2 + y_7^2 - y_8^2)\times \\
\times \Bigl((y_1^2 + y_2^2 + y_3^2 + y_4^2 - y_5^2 - y_6^2 - y_7^2 - y_8^2)^2 + 4(y_1y_6 - y_2y_7 + y_3y_8 - y_4y_5)^2\Bigr).
\end{eqnarray*}
Hence, for {\it almost all} $Y\in\mathfrak v$,  we obtain $M_{\,[1, 2, 3, 4, 6, 7, 8]}^{[1, 2, 3, 4, 6, 7, 8]}\neq 0$ and $\rank(M)=\rank(ME)=7$, that implies that
we have a solution of the corresponding system.

It is possible to compute all other minors of $M$ of order $7$. For example,
\begin{eqnarray*}
M_{\,[1, 2, 3, 4, 6, 7, 8]}^{[1, 3, 7, 8, 10, 11, 15]}= -8y_1 \cdot (y_1y_3 + y_2y_4 - y_5y_7 - y_6y_8)\times \\
\Bigl((y_1^2 + y_2^2 + y_5^2 + y_8^2)(y_3y_7 + y_4y_6) - (y_3^2 + y_4^2 + y_6^2 + y_7^2)(y_1y_5 + y_2y_8)\Bigr).
\end{eqnarray*}

The set of $Y\in\mathfrak v$ with the property $\rank(M)<7$ is the zero set of several polynomials, hence, it determines a polynomial ideals.
It is interesting to describe this ideal completely.

More precisely, let $M(7)$ be the set of all minors of size 7 of the matrix $M$ (any such minor is a $7$-form in coordinates of $Y$).
There are $8\cdot C_{15}^7=51480$ minors of size 7 of the matrix $M$.
Now, let us consider the following set:
$$
ZM(7)=\left\{Y\in \mathbb{R}^8\,|\, f(Y)=0 \mbox{  for any } f\in M(7)\right\}.
$$
If $Y \not \in ZM(7)$, then there is a minor $f$ of size 7 of the matrix $M$ such that $f(y)\neq 0$. Therefore, $\rank(M)=7$, hence,
the corresponding linear system for the vector $Y$ has a solution. If $Y \in ZM(7)$ then $\rank(M)\leq 6$, but this does not mean that the corresponding linear system has no solution! But all possible ``bad'' vectors $Y$ are in the set $ZM(7)$. Note that $ZM(7)$ is a closed subset of zero measure in $\mathbb{R}^8$.
Moreover, $ZM(7)$ is an algebraic set. On the other hand, this approach demands a lot of computations. We therefore adopt a different approach.

Let us show that we may assume (without loss of generality) that
$y_1=\langle Y,  V_1\rangle_{4,4}\neq 0$.
Indeed, if $y_1=0$, then there is $y_i \neq 0$ for some $i=2,\dots,8$. We can choose a matrix of the type $Q=\exp(A)$, where $A\in\mathbf{N}_1$
such that $\langle QY,  V_1\rangle_{4,4}\neq 0$. Indeed, if $i\in\{2,3,4,5,6,7,8\}$, then we can take $Q=\exp\bigl(t [J_{Z_k},J_{Z_l}]\bigr)$
such that $(k,l)\in\{(2,3), (1,3), (1,2), (1,7), (1,6), (1,5), (2,5)\}$ for respective $i$, and $t$ is a small positive number.
For instance, $\langle QY,  V_1\rangle_{4,4}=\cos(2t)y_1+\sin(2t)y_2$ if  $Q=\exp\bigl(t[J_{Z_2},J_{Z_3}]\bigr)$ and
$\langle QY,  V_1\rangle_{4,4}=\cosh (2t)y_1+\sinh (2t)y_5=\frac{e^{2t}+e^{-2t}}{2}y_1+\frac{e^{2t}-e^{-2t}}{2}y_5$ if  $Q=\exp(t[J_{Z_1},J_{Z_7}])$.

Recall that $k\neq 4$ and $l\neq 4$ in the above pairs. In particular, $[J_{Z_3},[J_{Z_k},J_{Z_l}]]=[J_{Z_4},Q]=0$ for all these pairs $(k,l)$
and $QJ_{Z_4}Q^{-1}=J_{Z_4}$. The equality $B(Y)=J_{Z_4}(Y)$ is equivalent to $QBQ^{-1}(QY)=QJ_{Z_4}Q^{-1}(QY)=J_{Z_4}(QY)$.
Since $[QBQ^{-1},J_{Z_4}]=0$ if  $[B,J_{Z_4}]=0$, we can find a suitable $B \in \mathbf{N}$ for the pair $(J_{Z_4},Y)$ if and only if we can find a suitable
$\widetilde{B} \in \mathbf{N}$  for the pair $(J_{Z_4},QY)$, where $Q\in \mathbf{N}$ such that $[Q,J_{Z_4}]=0$.

In what follows, we assume that $y_1\neq 0$ (by the above arguments).

\begin{lemma}\label{le.2case.0}
If a vector $Y=(y_1,y_2,y_3,y_4,y_5,y_6,y_7,y_8)$ satisfies the equalities $U_i=0$, $i=1,\dots, 8$, where
\begin{eqnarray}\label{eq.vspom.im} \notag
U_1&=&y_1y_8 - y_2y_5 - y_3y_6 + y_4y_7,\\ \notag
U_2&=&y_1y_7 + y_2y_6 - y_3y_5 - y_4y_8,\\ \notag
U_3&=&y_1y_4+y_2y_3-y_5y_6-y_7y_8, \\ \notag
U_4&=&y_1y_4-y_2y_3-y_5y_6+y_7y_8,\\
U_5&=&y_1y_3 - y_2y_4 - y_5y_7 + y_6y_8,\\ \notag
U_6&=&y_1y_2 + y_3y_4 - y_5y_8 - y_6y_7,\\ \notag
U_7&=&y_1^2+y_2^2-y_3^2-y_4^2-y_5^2+y_6^2+y_7^2-y_8^2,\\ \notag
U_8&=&y_1^2 - y_2^2 + y_3^2 - y_4^2 - y_5^2 + y_6^2 - y_7^2 + y_8^2, \notag
\end{eqnarray}
then one of the following two cases holds:
\begin{eqnarray*}
1)\,\,\, Y=(y_1,y_2,y_3,y_4,y_1,y_4,y_3,y_2),  \quad
2)\,\,\, Y=(y_1,y_2,y_3,y_4,-y_1,-y_4,-y_3,-y_2),
\end{eqnarray*}
for any $y_1,y_2,y_3,y_4 \in \mathbb{R}$. In particular, in both cases we have $y_1^2+y_2^2+y_3^2+y_4^2-y_5^2-y_6^2-y_7^2-y_8^2=0$ and
$y_1y_6-y_2y_7+y_3y_8-y_4y_5=0$.
\end{lemma}

\begin{proof}
We may assume without loss of generality that $y_1\neq 0$.
From $U_3=0$ and $U_4=0$, adding and subtracting give
$y_1y_4=y_5y_6$ and
$y_2y_3=y_7y_8$.
Adding $U_7$ and $U_8$ yields
$y_1^2+y_6^2=y_4^2+y_5^2$.
Multiplying it by $y_6^2$, we get
$y_1^2y_6^2+y_6^4=y_4^2y_6^2+y_5^2y_6^2$.
Using $y_1y_4=y_5y_6$, we have $y_5^2y_6^2=y_1^2y_4^2$. Hence,
$$
0=y_6^4+(y_1^2-y_4^2)y_6^2-y_5^2y_6^2=y_6^4+(y_1^2-y_4^2)y_6^2-y_1^2y_4^2=(y_6^2-y_4^2)(y_6^2+y_1^2).
$$
Since $y_1\neq 0$, we get $y_6^2+y_1^2>0$, therefore,
$y_6^2=y_4^2$.
Now choose $\varepsilon\in\{+1,-1\}$ such that
$y_6=\varepsilon y_4$.
If $y_4\neq 0$, then from $y_1y_4=y_5y_6$ we get $y_5=\varepsilon y_1$.
If $y_4=0$, then $y_6=0$ and $y_1^2=y_5^2$, so $y_5=\pm y_1$.
In that case we simply set $\varepsilon=y_5/y_1$, and again
$$
y_5=\varepsilon y_1,\qquad y_6=\varepsilon y_4.
$$
Thus in all cases there exists $\varepsilon\in\{+1,-1\}$ such that
$(y_5,y_6)=\varepsilon(y_1,y_4)$.
Substituting it into $U_1=0$ and $U_2=0$ we get
\begin{eqnarray*}
y_4y_7+y_1y_8&=&\varepsilon(y_1y_2+y_4y_3),\\
y_1y_7-y_4y_8&=&\varepsilon(y_1y_3-y_2y_4).
\end{eqnarray*}
This is a linear system in $y_7,y_8$. Its determinant is
$ -y_1^2-y_4^2\neq 0$,
so the solution is unique. One checks immediately that
$$
y_7=\varepsilon y_3,\qquad y_8=\varepsilon y_2
$$
satisfies both equations of this system. Hence
$(y_7,y_8)=\varepsilon(y_3,y_2)$.
Combining it with the previous results, we get
$$
(y_5,y_6,y_7,y_8)=\varepsilon(y_1,y_4,y_3,y_2).
$$
If $\varepsilon=1$, we obtain the first case; if $\varepsilon=-1$, we obtain the second.

Finally, substituting either of the two obtained forms of $Y$ directly verifies the two additional identities:
$$
y_1^2+y_2^2+y_3^2+y_4^2-y_5^2-y_6^2-y_7^2-y_8^2=0
$$
and
$$
y_1y_6-y_2y_7+y_3y_8-y_4y_5=0.
$$
Thus the lemma is proved.
\end{proof}

\begin{lemma}\label{le.2case.1}
If $\rank({M})<7$, then $y_1^2 + y_2^2 + y_3^2 + y_4^2 = y_5^2 + y_6^2 + y_7^2 + y_8^2$ and $y_1y_6  + y_3y_8= y_2y_7 + y_4y_5$.
\end{lemma}

\begin{proof}
For $i=2,\ldots,9$, we denote by $D(i)$ the minor ${M}_{\,[1, 2, 4, 5, 6, 7, 8]}^{[i, 10, 11, 12, 13, 14, 15]}$. Then we get the following equalities:
\begin{equation*}
\begin{array}{llllll}
&D(2)=-2y_1\cdot U_1 \cdot U_0,\qquad &D(6)=-2y_1\cdot U_2 \cdot  U_0,\\
&D(3)=-y_1\cdot U_6 \cdot  U_0,\qquad &D(7)=2y_1\cdot U_5 \cdot  U_0,\\
&D(4)=-2y_1\cdot U_8 \cdot  U_0,\qquad &D(8)=2y_1\cdot U_3 \cdot  U_0,\\
&D(5)=2y_1\cdot U_4 \cdot  U_0,\qquad &D(9)=y_1\cdot U_7 \cdot  U_0,
\end{array}
\end{equation*}
where $U_i$, $1\leq i \leq 8$, are given in \eqref{eq.vspom.im} and
\begin{eqnarray*}
U_0&=&(y_1^2 + y_2^2 + y_3^2 + y_4^2 - y_5^2 - y_6^2 - y_7^2 - y_8^2)^2 + 4(y_1y_6 - y_2y_7 + y_3y_8 - y_4y_5)^2.
\end{eqnarray*}
If $U_0=0$, then the assertion of the theorem follows.

Let us suppose that $U_0\neq 0$.
If $U_i \neq 0$ for some $i=1,\dots,8$, then $\rank({M})=7$. On the other hand, $U_1=U_2=U_3=U_4=U_5=U_6=U_7=U_8=0$ implies $U_0=0$
by Lemma \ref{le.2case.0}.
Hence, the lemma is proved.
\end{proof}
\smallskip

Note that all matrices from $\exp(\mathbf{N})$ preserve the quadratic form $y_1^2 + y_2^2 + y_3^2 + y_4^2-y_5^2 - y_6^2-y_7^2- y_8^2$.
On the other hand, it is not true for the quadratic form $y_1y_6 - y_2y_7 + y_3y_8- y_4y_5$.

\begin{lemma}\label{le.2case.2}
If the vector $Y=(y_1,y_2,y_3,y_4,y_5,y_6,y_7,y_8)$ with $y_1 \neq 0$ is such that there is no $B\in \mathbf{N}_1$ satisfying $B(Y)=J_{Z_4}(Y)$,
then $Y$ has one of the following two forms:
\begin{eqnarray*}
1)\,\,\, Y=(y_1,y_2,y_3,y_4,y_1,y_4,y_3,y_2), \quad  2)\,\,\, Y=(y_1,y_2,y_3,y_4,-y_1,-y_4,-y_3,-y_2),
\end{eqnarray*}
where $y_1,y_2,y_3,y_4 \in \mathbb{R}$.
\end{lemma}

\begin{proof}
Let us consider any $Q\in \exp( \mathbf{N}_1)$. If $Q$ does not preserve the equality $y_1y_6 - y_2y_7 + y_3y_8- y_4y_5=0$,
then the matrix $M$ for the vector $\overline{Y}=Q(Y)$ has rank $7$ by Lemma \ref{le.2case.1}.
Therefore, we have some $\overline{B}\in \mathbf{N}_1$ such that $\overline{B}(\overline{Y})=J_{Z_4}(\overline{Y})$. Consequently,
$B=Q^{-1}\overline{B}Q \in \mathbf{N}_1$ satisfies $B(Y)=J_{Z_4}(Y)$, that is impossible by the assumptions.
Hence, any $Q\in \exp( \mathbf{N}_1)$ preserves the equality $y_1y_6  + y_3y_8= y_2y_7 + y_4y_5=0$.

By Lemma \ref{le.invarform.1} we have
\begin{equation}\label{eq.invar.2}
\overline{y}_1y_6+ y_1\overline{y}_6+  \overline{y}_3y_8+  y_3\overline{y}_8 - \overline{y}_2y_7 - y_2\overline{y}_7- \overline{y}_4y_5- y_4\overline{y}_5=0,
\end{equation}
where $\overline{Y}=A(Y)$ and $A$ is any matrix in $\mathbf{N}_1$. Taking various $A=[J_{Z_k},J_{Z_l}]$ for $k,l \in \{1,2,3,5,6,7\}$, we obtain the
equations $U_i=0$, $1\leq i \leq 8$, where $U_i$ are given in  \eqref{eq.vspom.im}.
Now, it suffices to apply Lemma \eqref{le.2case.0}.
\end{proof}

\begin{prop}\label{pr.2case}
For any $Y \in \mathbb{R}^{4,4}$, it is possible to find $B\in \mathbf{N}= [\mathbf{V},\mathbf{V}]$
such that
\begin{equation}\label{eq:B34}
[B,Z]=[B,J_{Z_4}]=0\quad\text{ and }\quad B(Y)=Z(Y)=J_{Z_4}(Y).
\end{equation}
\end{prop}

\begin{proof} Suppose that for some $Y=(y_1,y_2,y_3,y_4,y_5,y_6,y_7,y_8) \in \mathbb{R}^8$, there is no $B\in \mathbf{N}$ satisfying~\eqref{eq:B34}. The discussion above shows that (without loss of generality) we may assume that
$y_1\neq 0$. Lemma~\ref{le.2case.2} implies that we have one of the following two possibility:
$$
{\rm 1)}\,\,\, Y=(y_1,y_2,y_3,y_4,y_1,y_4,y_3,y_2),\quad {\rm 2)}\,\,\, Y=(y_1,y_2,y_3,y_4,-y_1,-y_4,-y_3,-y_2),
$$
where $y_1,y_2,y_3,y_4 \in \mathbb{R}$. We will find an explicit form of $B \in \mathbf{N}$  satisfying~\eqref{eq:B34} for both cases.

Let  us start from the case {\rm 1)}.

If $y_1y_3-y_2y_4 \neq 0$, then we have the solution $B$ with the following entries:
$$
x_{12} = -\frac{y_1y_2+y_3y_4}{y_1y_3-y_2y_4}, \quad x_{13} = -\frac{y_1^2-y_2^2-y_3^2+y_4^2}{2(y_1y_3-y_2y_4)},\quad
x_{15} = -\frac{y_1^2+y_2^2+y_3^2+y_4^2}{2(y_1y_3-y_2y_4)},
$$
and $x_{16}=x_{17}=x_{23}=x_{25}=x_{26}=x_{27}=x_{35}=x_{36}=x_{37}=x_{56}=x_{57}=x_{67}=0$.

If $y_1y_3-y_2y_4 =0$ and $y_4 \neq 0$, then $y_3=\frac{y_2y_4}{y_1}$ and we can take the matrix $B$ with the following entries:
$$
x_{12} = \frac{(y_1^2-y_2^2)(y_1^2-y_4^2)}{2y_1y_4(y_1^2+y_2^2)}, \quad x_{13} = -\frac{y_2(y_1^2-y_4^2)}{y_4(y_1^2+y_2^2)}, \quad x_{15} = 0, \quad
x_{16} = \frac{y_1^2+y_4^2}{2y_1y_4},
$$
and $x_{17}=x_{23}=x_{25}=x_{26}=x_{27}=x_{35}=x_{36}=x_{37}=x_{56}=x_{57}=x_{67}=0$.

If $y_1y_3-y_2y_4 =0$ and $y_4=0$, then $y_3=0$. Therefore, we can take $B$ with $x_{17}=1$, while $x_{ij}=0$ for all pairs $(i,j)\neq (1,7)$.
This can be verified directly from the extended matrix.  $ME$ (recall that $y_7=y_3$ and $y_6=y_4$, hence, $y_3=y_4=y_6=y_7=0$).

We next consider case {\rm 2)}.

If $y_1y_4+y_2y_3 \neq 0$, then we have the solution $B$ with the following entries:
$$
x_{12} = -\frac{y_1^2-y_2^2+y_3^2-y_4^2}{2(y_1y_4+y_2y_3)},\quad x_{13} = \frac{y_1y_2-y_3y_4}{y_1y_4+y_2y_3},\quad x_{15} = 0,\quad x_{16} = \frac{y_1^2+y_2^2+y_3^2+y_4^2}{2(y_1y_4+y_2y_3)},
$$
and $x_{17}=x_{23}=x_{25}=x_{26}=x_{27}=x_{35}=x_{36}=x_{37}=x_{56}=x_{57}=x_{67}=0$.

If $y_1y_4+y_2y_3 =0$ and $y_4 \neq 0$, then $y_2 \neq 0$, $y_3\neq 0$, $y_3=-\frac{y_2y_4}{y_1}$,  and we can take  $B$ with the following entries:
$$
x_{12} = -\frac{y_1(y_2^2-y_4^2)}{y_4(y_1^2+y_2^2)}, \quad x_{13} = -\frac{(y_1^2-y_2^2)(y_2^2-y_4^2)}{2y_2y_4(y_1^2+y_2^2)}, \quad x_{15} = \frac{y_2^2+y_4^2}{2y_2y_4},
$$
and $x_{16}=x_{17}=x_{23}=x_{25}=x_{26}=x_{27}=x_{35}=x_{36}=x_{37}=x_{56}=x_{57}=x_{67}=0$.

If $y_1y_4+y_2y_3 =0$ and $y_4 =0$, then $y_2y_3=0$. If $y_3=0$, then we can take $B$ with $x_{17}=1$, while $x_{ij}=0$ for all pairs $(i,j)\neq (1,7)$.
It is easy to verify using the extended matrix  $ME$ (recall that $y_7=-y_3$ and $y_6=-y_4$, hence, $y_3=y_4=y_6=y_7=0$).

Finally, if $y_4 =0$ and $y_2=0$, then we can take $B$ with $x_{26}=1$, while $x_{ij}=0$ for all pairs $(i,j)\neq (2,6)$.
It is easy to verify using the extended matrix  $ME$ (recall that $y_8=-y_2$ and $y_6=-y_4$, hence, $y_2=y_4=y_6=y_8=0$).
The proposition is proved.
\end{proof}

\subsection{The case $Z= J_{Z_1}+J_{Z_4}$}

The condition $[B,Z]=0$ implies
$x_{12}=-x_{24}$, $x_{13}=x_{34}$, $x_{14}=0$, $x_{15}=x_{45}$, $x_{16}=x_{46}$, $x_{17}=x_{47}$.
For the condition $B(Y)=Z(Y)$ we consider the linear equation system $2B(Y)=Z(Y)=J_{Z_1}(Y)+J_{Z_4}(Y)$
(any solution of this system is just one half of a solution of $B(Y)=Z(Y)$).
Since $J_{Z_1}(Y)+J_{Z_4}(Y)=(y_5-y_2,y_1+y_8,y_4+y_7,y_6-y_3,y_1+y_8,y_4+y_7,y_3-y_6,y_2-y_5)$,
then we have a system of linear equations with respect to the following  variables
$$
x_{23}, x_{25}, x_{24}, x_{26}, x_{27}, x_{34}, x_{35}, x_{36}, x_{37}, x_{45}, x_{46}, x_{47}, x_{56}, x_{57}, x_{67}
$$
with the following
extended matrices (the last column is the column of free terms of our system of linear equations):
$$
{\tiny \widetilde{ME}}{\small:=
\left(
\begin{smallmatrix}
y_1& -y_3 + y_6& -y_5& y_8& -y_7& -y_4 - y_7& y_8& y_5& -y_6& y_3 - y_6& -y_4 - y_7& -y_1 - y_8& -y_1& y_4& y_3& -y_1 - y_8\\
-y_8& -y_3 + y_6& -y_2& -y_1& y_4& -y_4 - y_7& -y_1& y_2& -y_3& y_3 - y_6& -y_4 - y_7& -y_1 - y_8& y_8& -y_7& y_6& -y_1 - y_8\\
-y_2& y_4 + y_7& -y_8& -y_5& y_6& -y_3 + y_6& -y_5& y_8& -y_7& y_4 + y_7& y_3 - y_6& y_2 - y_5& y_2& -y_3& y_4& y_2 - y_5\\
y_5& -y_4 - y_7& -y_1& y_2& -y_3& y_3 - y_6& y_2& y_1& -y_4& -y_4 - y_7& -y_3 + y_6& -y_2 + y_5& -y_5& y_6& y_7& -y_2 + y_5\\
-y_3& -y_1 - y_8& -y_7& y_6& y_5& y_2 - y_5& -y_6& -y_7& -y_8& -y_1 - y_8& y_2 - y_5& -y_3 + y_6& -y_3& -y_2& -y_1& y_3 - y_6\\
y_6& y_1 + y_8& -y_4& -y_3& -y_2& -y_2 + y_5& y_3& -y_4& -y_1& y_1 + y_8& -y_2 + y_5& y_3 - y_6& y_6& y_5& -y_8& -y_3 + y_6\\
-y_7& y_2 - y_5& -y_3& y_4& y_1& y_1 + y_8& -y_4& -y_3& -y_2& -y_2 + y_5& -y_1 - y_8& y_4 + y_7& -y_7& -y_8& -y_5& -y_4 - y_7\\
y_4& y_2 - y_5& -y_6& -y_7& -y_8& y_1 + y_8& y_7& -y_6& -y_5& -y_2 + y_5& -y_1 - y_8& y_4 + y_7& y_4& y_1& -y_2& -y_4 - y_7
\end{smallmatrix}
\right).}
$$

By the Kronecker--Capelli theorem, this system has a solution if and only if the rank of $\widetilde{ME}$ coincides with the rank of the matrix $\widetilde{M}$, which is
obtained from $\widetilde{ME}$ by deleting the last column (the column of free terms).
A direct computation shows that the product of the vector $(-y_2, y_5, -y_1, y_8, -y_4, y_7, y_6, -y_3)$
and the matrix $\widetilde{ME}$ is a zero vector with $16$ entries.
Hence, $\rank(\widetilde{ME})\leq 7$  (see also Remark~\ref{re.act.rank.1}).
Therefore, if $\rank(\widetilde{M})=7$, then the system $2B(Y)=J_{Z_1}(Y)+J_{Z_4}(Y)$ has a solution.
By symbol $\widetilde{M}_{\,[i_1, i_2, \dots,i_s]}^{[j_1, j_2, \dots,j_s]}$, we denote the minor of $\widetilde{M}$ which is determined by the rows with the numbers $i_1,i_2,\dots,i_s$ and columns with the numbers $j_1, j_2, \dots,j_s$.
It is easy to check that
\begin{eqnarray*}
\widetilde{M}_{\,[1, 2, 3, 4, 5, 6, 7]}^{[9, 10, 11, 12, 13, 14, 15]}= 2y_3\cdot \bigl((y_1 + y_8)(y_4 + y_7) - (y_2 - y_5)(y_3 - y_6)\bigr)\times\\
\times \bigl((y_1 + y_8)^2 + (y_2 - y_5)^2 + (y_3 - y_6)^2 + (y_4 + y_7)^2\bigr)^2.
\end{eqnarray*}
Hence, for {\it almost all} $Y\in\mathfrak v$ we have $\widetilde{M}_{\,[1, 2, 3, 4, 5, 6, 7]}^{[9, 10, 11, 12, 13, 14, 15]}\neq 0$ and $\rank(\widetilde{M})=\rank(\widetilde{ME})=7$, that implies the existence of a solution of the corresponding system.

Moreover, we can change the above minor, removing any line instead of the $8$th one, and  this new minor is distinct from the given one exactly in the first multiple ($2y_3$ will change to $\pm 2y_i$ according to the entries of the vector $(-y_2, y_5, -y_1, y_8, -y_4, y_7, y_6, -y_3)$). Therefore,
{\it if $(y_1 + y_8)(y_4 + y_7) \neq (y_2 - y_5)(y_3 - y_6)$, then $\rank(\widetilde{M})=\rank(\widetilde{ME})=7$, hence,
there is a solution of the corresponding linear system}. Furthermore, the following is true

\begin{lemma}\label{le.3case}
If $(y_1 + y_8)^2 + (y_2 - y_5)^2 + (y_3 - y_6)^2 + (y_4 + y_7)^2\neq 0$, then
$\rank(\widetilde{M})=7$. If $(y_1 + y_8)^2 + (y_2 - y_5)^2 + (y_3 - y_6)^2 + (y_4 + y_7)^2= 0$, then
the last column of the matrix $\widetilde{ME}$ (the column of free terms) has only zero entries.
\end{lemma}

\begin{proof}
Suppose (without loss of generality) that $y_1\neq 0$.
Now, we suppose that $(y_1 + y_8)^2 + (y_2 - y_5)^2 + (y_3 - y_6)^2 + (y_4 + y_7)^2\neq 0$. Let us prove that $\rank(\widetilde{M})=7$ in this case.

For $i=1,\dots,9$, we denote by $D(i)$ the minor $\widetilde{M}_{\,[1, 2, 4, 5, 6, 7, 8]}^{[i, 10, 11, 12, 13, 14, 15]}$. Then we get the following equalities:
\begin{eqnarray*}
D(3)&=&-y_1\cdot W_1 \cdot \bigl((y_1 + y_8)^2 + (y_2 - y_5)^2 + (y_3 - y_6)^2 + (y_4 + y_7)^2\bigr)^2,\\
D(4)&=&2y_1\cdot W_3 \cdot \bigl((y_1 + y_8)^2 + (y_2 - y_5)^2 + (y_3 - y_6)^2 + (y_4 + y_7)^2\bigr)^2,\\
D(7)&=&2y_1\cdot W_4 \cdot \bigl((y_1 + y_8)^2 + (y_2 - y_5)^2 + (y_3 - y_6)^2 + (y_4 + y_7)^2\bigr)^2,\\
D(8)&=&y_1\cdot W_2 \cdot \bigl((y_1 + y_8)^2 + (y_2 - y_5)^2 + (y_3 - y_6)^2 + (y_4 + y_7)^2\bigr)^2.
\end{eqnarray*}
where
\begin{eqnarray*}
W_1&=&(y_1+y_8)^2  - (y_2-y_5)^2 - (y_3-y_6)^2 + (y_4+y_7)^2,\\
W_2&=&(y_1+y_8)^2  - (y_2-y_5)^2 + (y_3-y_6)^2 - (y_4+y_7)^2,\\
W_3&=&(y_1+y_8)(y_2-y_5)-(y_3-y_6)(y_4+y_7),\\
W_4&=&(y_1+y_8)(y_2-y_5) + (y_3-y_6)(y_4 + y_7).
\end{eqnarray*}

If $W_i \neq 0$ for any $i=1,\dots,4$, then $\rank(\widetilde{M})=7$. Let us suppose that $W_1=W_2=W_3=W_4=0$.
Then $W_1=W_2=0$ implies $(y_1+y_8)^2=(y_2-y_5)^2$ and $(y_3-y_6)^2 = (y_4+y_7)^2$.
From $W_3=W_4=0$ we get $(y_1+y_8)(y_2-y_5)=(y_3-y_6)(y_4+y_7)=0$.
Therefore, a direct computation shows that $(y_1+y_8)=(y_2-y_5)=(y_3-y_6)=(y_4+y_7)=0$ that is impossible by our assumption.
Hence, $\rank(\widetilde{M})=7$ if $(y_1 + y_8)^2 + (y_2 - y_5)^2 + (y_3 - y_6)^2 + (y_4 + y_7)^2\neq 0$.

The second assertion is obvious. The lemma is proved.
\end{proof}

\begin{corollary}\label{pr.3case}
The linear system with the extended matrix $\widetilde{ME}$ has a solution for any $Y \in \mathbb{R}^{4,4}$.
Hence, for any $Y \in \mathbb{R}^{4,4}$, it is possible to find $B\in\mathbf{N}= [\mathbf{V},\mathbf{V}]$
such that  $[B,Z]=[B,J_{Z_1}+J_{Z_4}]=0$ and $B(Y)=Z(Y)=J_{Z_1}(Y)+J_{Z_4}(Y)$.
\end{corollary}

\begin{proof}
If $\rank(\widetilde{M})=7$ then we get a solution due to $\rank(\widetilde{ME})=7$ (see the discussion above and  Remark~\ref{re.act.rank.1}).
If $\rank(\widetilde{M})<7$, then
the last column of the matrix $\widetilde{ME}$ (the column of free terms) has only zero entries
by Lemma~ \ref{le.3case}. This means  $\rank(\widetilde{M})=\rank(\widetilde{ME})$ and we also have a solution.
\end{proof}
\smallskip

It is possible to compute all minors of the matrix $\widetilde{M}$ of order $7$.
The set of $Y$ with the property $\rank(\widetilde{M})<7$ is the zero set of several polynomials, hence, it determines a polynomial ideals.
In this case, such a description follows from Lemma~\ref{le.3case}.

%\section{Acknowledgements}
%
%K. Furutani was partially supported by JSPS KAKENHI Grant Number 24K06784 and the Osaka Central Advanced Mathematical Institute,
%Osaka Metropolitan University (MEXT Promotion of Distinctive Joint Research Center Program JPMXP0723833165). K.~Furutani and I.~Markina
%express their gratitude to the Osaka Central Advanced Mathematical Institute, Osaka Metropolitan University, for the hospitality,
%where a substantial part of this work was completed during I. Markina's research stay in December 2024. I. Markina was partially
%supported by the L. Meltzer University Foundation, which provided travel support for her visit to the Osaka Central Advanced Mathematical Institute,
%Osaka Metropolitan University, Japan.
%
%
%\medskip
%
%{\bf Data availability.}
%Data sharing is not applicable to this article as no data sets
%were generated or analysed during the current study.
%\medskip
%
%{\bf Declarations.}
%\medskip
%
%{\bf Competing interests.}
%The corresponding author states that there is no conflict
%of interest.
%

%\vspace{14mm}

\bibliographystyle{amsunsrt}

\vspace{14mm}

\end{document}